\documentclass[a4paper,12pt]{article} 
\usepackage[utf8]{inputenc}
\usepackage[T1]{fontenc}
\usepackage{amsmath,amssymb,amsthm}
\usepackage[]{authblk}
\usepackage{bbm}
\usepackage{mathrsfs}

\usepackage{verbatim}
\usepackage{mathtools}
\usepackage{mathabx}
\usepackage{bbm}
\usepackage{float}
\usepackage{array}
\usepackage{multirow}
\usepackage{diagbox}
\usepackage{caption}
\usepackage{subcaption}
\usepackage[numeric,initials,nobysame,msc-links,abbrev]{amsrefs}
\renewcommand{\eprint}[1]{\href{https://arxiv.org/abs/#1}{arXiv:#1}}
\newcommand{\pageafter}[1]{#1~pp.}
\BibSpec{article}{%
+{} {\PrintAuthors} {author}
+{,} { \textit} {title}
+{.} { } {part}
+{:} { \textit} {subtitle}
+{,} { \PrintContributions} {contribution}
+{.} { \PrintPartials} {partial}
+{,} { } {journal}
+{} { \textbf} {volume}
+{} { \PrintDatePV} {date}
+{,} { \issuetext} {number}
+{,} { \pageafter} {pages}
+{,} { } {status}
+{,} { \PrintDOI} {doi}
+{,} { available at \eprint} {eprint}
+{} { \parenthesize} {language}
+{} { \PrintTranslation} {translation}
+{;} { \PrintReprint} {reprint}
+{.} { } {note}
+{.} {} {transition}
+{} {\SentenceSpace \PrintReviews} {review}
}
\BibSpec{collection.article}{%
+{} {\PrintAuthors} {author}
+{,} { \textit} {title}
+{.} { } {part}
+{:} { \textit} {subtitle}
+{,} { \PrintContributions} {contribution}
+{,} { \PrintConference} {conference}
+{} {\PrintBook} {book}
+{,} { } {booktitle}
+{,} { \PrintDateB} {date}
+{,} { \pageafter} {pages}
+{,} { } {status}
+{,} { \PrintDOI} {doi}
+{,} { available at \eprint} {eprint}
+{} { \parenthesize} {language}
+{} { \PrintTranslation} {translation}
+{;} { \PrintReprint} {reprint}
+{.} { } {note}
+{.} {} {transition}
+{} {\SentenceSpace \PrintReviews} {review}
}
\usepackage{enumitem}
\usepackage{pgf,tikz}
\usetikzlibrary{arrows}
\usetikzlibrary[patterns]
\usetikzlibrary{shapes.misc}
\usetikzlibrary{hobby}

\tikzset{cross/.style={cross out, draw=black, minimum size=2*(#1-\pgflinewidth), inner sep=0pt, outer sep=0pt},
cross/.default={1pt}}

\makeatletter
\pgfdeclarepatternformonly[\LineSpace,\tikz@pattern@color]{my north east lines}{\pgfqpoint{-1pt}{-1pt}}{\pgfqpoint{\LineSpace}{\LineSpace}}{\pgfqpoint{\LineSpace}{\LineSpace}}%
{
    \pgfsetcolor{\tikz@pattern@color}
    \pgfsetlinewidth{0.4pt}
    \pgfpathmoveto{\pgfqpoint{0pt}{0pt}}
    \pgfpathlineto{\pgfqpoint{\LineSpace + 0.1pt}{\LineSpace + 0.1pt}}
    \pgfusepath{stroke}
}
\makeatother

\makeatletter
\pgfdeclarepatternformonly[\LineSpace,\tikz@pattern@color]{my horizontal lines}{\pgfqpoint{-1pt}{-1pt}}{\pgfqpoint{\LineSpace}{\LineSpace}}{\pgfqpoint{\LineSpace}{\LineSpace}}%
{
    \pgfsetcolor{\tikz@pattern@color}
    \pgfsetlinewidth{0.2pt}
    \pgfpathmoveto{\pgfqpoint{0pt}{0pt}}
    \pgfpathlineto{\pgfqpoint{\LineSpace + 0.1pt}{0pt}}
    \pgfusepath{stroke}
}
\makeatother

\makeatletter
\pgfdeclarepatternformonly[\LineSpace,\tikz@pattern@color]{my vertical lines}{\pgfqpoint{-1pt}{-1pt}}{\pgfqpoint{\LineSpace}{\LineSpace}}{\pgfqpoint{\LineSpace}{\LineSpace}}%
{
    \pgfsetcolor{\tikz@pattern@color}
    \pgfsetlinewidth{0.2pt}
    \pgfpathmoveto{\pgfqpoint{0pt}{0pt}}
    \pgfpathlineto{\pgfqpoint{0pt}{\LineSpace + 0.1pt}}
    \pgfusepath{stroke}
}
\makeatother

\newdimen\LineSpace
\tikzset{
    line space/.code={\LineSpace=#1},
    line space=10pt
}

\definecolor{ffqqqq}{rgb}{1,0,0}
\definecolor{qqffqq}{rgb}{0,1,0}
\definecolor{ffffff}{rgb}{1,1,1}

\usepackage{hyperref}

\newtheorem{thm}{Theorem}
\newtheorem{mainthm}[thm]{Theorem}
\newtheorem{cor}[thm]{Corollary}
\newtheorem{lem}[thm]{Lemma}
\newtheorem{prop}[thm]{Proposition}

\theoremstyle{definition}
\newtheorem{defn}[thm]{Definition}
\newtheorem{rem}[thm]{Remark}

\newtheorem{obs}[thm]{Observation}
\newtheorem{claim}[thm]{Claim}

\numberwithin{thm}{section}

\newcommand{\cB}{\ensuremath{\mathcal B}}
\newcommand{\cC}{\ensuremath{\mathcal C}}
\newcommand{\cD}{\ensuremath{\mathcal D}}

\newcommand{\cF}{\ensuremath{\mathcal F}}

\newcommand{\cH}{\ensuremath{\mathcal H}}

\newcommand{\cK}{\ensuremath{\mathcal K}}
\newcommand{\cL}{\ensuremath{\mathcal L}}

\newcommand{\cP}{\ensuremath{\mathcal P}}

\newcommand{\cS}{\ensuremath{\mathcal S}}
\newcommand{\cT}{\ensuremath{\mathcal T}}
\newcommand{\cU}{\ensuremath{\mathcal U}}
\newcommand{\cV}{\ensuremath{\mathcal V}}

\newcommand{\cZ}{\ensuremath{\mathcal Z}}

\newcommand{\bbE}{{\ensuremath{\mathbb E}} }

\newcommand{\bbH}{{\ensuremath{\mathbb H}} }

\newcommand{\bbP}{{\ensuremath{\mathbb P}} }
\newcommand{\bbQ}{{\ensuremath{\mathbb Q}} }
\newcommand{\bbR}{{\ensuremath{\mathbb R}} }

\newcommand{\bbT}{{\ensuremath{\mathbb T}} }
\newcommand{\bbU}{{\ensuremath{\mathbb U}} }

\newcommand{\bbZ}{{\ensuremath{\mathbb Z}} }

\newcommand{\<}{\langle}
\renewcommand{\>}{\rangle}

\newcommand{\diam}{{\ensuremath{\mathrm{diam}}} }

\newcommand{\1}{{\ensuremath{\mathbbm{1}}} }

\newcommand{\trel}{\ensuremath{T_{\mathrm{rel}}}}
\newcommand{\pl}{\ensuremath{p_{\leftarrow}}}
\newcommand{\pd}{\ensuremath{p_{\downarrow}}}

\renewcommand{\leq}{\leqslant}
\renewcommand{\geq}{\geqslant}
\renewcommand{\le}{\leqslant}
\renewcommand{\ge}{\geqslant}
\renewcommand{\to}{\rightarrow}

\begin{document}
\title{Refined universality for critical KCM:\\lower bounds}
\author[,1]{Ivailo Hartarsky\thanks{\textsf{hartarsky@ceremade.dauphine.fr}}}
\author[,2]{Laure Mar\^ech\'e\thanks{\textsf{laure.mareche@math.unistra.fr}}}
\affil[1]{CEREMADE, CNRS, UMR 7534, Universit\'e Paris-Dauphine, PSL University\protect\\Place du Mar\'echal de Lattre de Tassigny, 75016 Paris, France}
\affil[2]{Institut de Recherche Mathématique Avancée\protect\\ UMR 7501 Université de Strasbourg et CNRS \protect\\
7 rue René-Descartes, 67000 Strasbourg, France}
\date{\vspace{-0.25cm}\today}
\maketitle
\vspace{-0.75cm}
\begin{abstract}
We study a general class of interacting particle systems called kinetically constrained models (KCM) in two dimensions tightly linked to the monotone cellular automata called bootstrap percolation. There are three classes of such models \cite{Bollobas15}, the most studied being the critical one. In a recent series of works \cites{Hartarsky20,Hartarsky21a,Martinelli19a} it was shown that the KCM counterparts of critical bootstrap percolation models with the same properties \cite{Bollobas14} split into two classes with different behaviour.

Together with the companion paper by the first author \cite{Hartarsky20II}, our work determines the logarithm of the infection time up to a constant factor for all critical KCM, which were previously known only up to logarithmic corrections. This improves all previous results except for the Duarte-KCM, for which we give a new proof of the best result known \cite{Mareche20Duarte}. We establish that on this level of precision critical KCM have to be classified into seven categories instead of the two in bootstrap percolation \cite{Bollobas14}. In the present work we establish lower bounds for critical KCM in a unified way, also recovering the universality result of Toninelli and the authors \cite{Hartarsky20} and the Duarte model result of Martinelli, Toninelli and the second author \cite{Mareche20Duarte}.
\end{abstract}


\section{Introduction and results}
\label{sec:intro}

\subsection{Kinetically constrained models}
\label{subsec:KCM}
Kinetically constrained models (KCM) are a class of interacting particle systems used since the 1980s to model the liquid-glass transition \cites{Fredrickson84,Fredrickson85} (see \cites{Ritort03,Garrahan11} for reviews). We will be interested in the very general class of $\cU$-KCM first introduced by Cancrini, Martinelli, Roberto and Toninelli in 2008 \cite{Cancrini08}, which includes all previously considered cases on $\bbZ^d$.

Let us start by introducing these models directly on the two-dimensional lattice, to which we restrict our attention. A KCM is a Markov process with state space $\Omega=\{0,1\}^{\bbZ^2}$. For a configuration $\eta\in\Omega$ and a site $x\in\bbZ^2$, we denote $\eta_x$ the value of $\eta$ at $x$. We say that $x$ is \emph{empty} or \emph{infected} if $\eta_x=0$ and that it is \emph{occupied} or \emph{healthy} if $\eta_x=1$. We thus naturally identify a configuration $\eta\in\Omega$ with the set of its infected sites, so that $\eta\subset \bbZ^2$.

A $\cU$-KCM is specified by two parameters---an \emph{update family} $\cU$ and an \emph{equilibrium measure} $\mu$. The measure $\mu$ on $\Omega$ is chosen to be the product Bernoulli measure such that each site is infected with probability $q>0$ and healthy with probability $1-q$. All asymptotics hereafter are taken as $q\to0$. The update family $\cU$ is a finite set of finite nonempty subsets of $\bbZ^2\setminus\{0\}$ called \emph{update rules}. The dynamics is the following. For each site $x\in\bbZ^2$ at rate 1 (i.e.\ at the times given by a Poisson point process on $\mathbb{R}_+$ with intensity 1) we attempt to update $\eta_x$ by replacing it by an independent Bernoulli$(1-q)$ variable. However, the update is only performed if there exists an update rule $U\in\cU$ such that $\eta\supset (x+U)$, while otherwise the configuration remains unchanged. In more formal terms the generator of the Markov process is the following for any function $f:\Omega\mapsto\bbR$ depending on a finite number of sites.
\[\cL(f)(\eta)=\sum_{x\in \bbZ^2}\1_{\{\exists U\in\cU,\eta_{x+U}=0\}}(\mu_x(f)-f)(\eta),\]
where $\eta_X$ denotes the restriction of $\eta$ to $X\cap\bbZ^2$ for $X\subset\bbR^2$ and $\mu_x(f)$ denotes the average of $f$ with respect to $\eta_x$ conditionally on all other occupation variables. We further view $\eta_X$ as the element of $\{0,1\}^{\bbZ^2}$ equal to $\eta$ in $X\cap\bbZ^2$ and to $1$ elsewhere. For background on the theoretical foundations of such interacting particle systems the reader is referred to \cite{Liggett05}. In particular, $\mu$ is indeed a reversible invariant measure for the $\cU$-KCM.

For any KCM arguably the most natural quantity of interest describing the speed at which memory of the initial state is lost is the \emph{first infection time} of the origin
\[\tau_0=\inf\{t\ge 0, (\eta(t))_0=0\}\in[0,\infty],\]
where $(\eta(t))_{t \geq 0}$ denotes the $\cU$-KCM process. As we will concentrate on the equilibrium properties of the KCM, we will rather be concerned with $\bbE_\mu[\tau_0]$, that is the expectation of $\tau_0$ with respect to the law of the stationary $\cU$-KCM with initial condition distributed according to its equilibrium measure $\mu$.

\subsection{Bootstrap percolation and universality}
\label{subsec:bootstrap}
Bootstrap percolation is a close relative of KCM, though the two fields remained relatively independent for decades. Formally, the continuous time version of $\cU$-bootstrap percolation is the $\cU$-KCM with $q=1$. However, bootstrap percolation has important additional properties and has attracted a great deal of attention with different motivation, also in non-lattice settings, as well as from computer science and sociological perspectives. We direct the reader to \cites{Morris17a,DeGregorio09} and the references therein for an overview of this rich field. 

In $\cU$-bootstrap percolation (the update family $\cU$ being as above), for each integer $t \geq 0$, a set $A_t$ of infected sites at time $t$ is constructed as follows. Given an initial set of infected sites $A_0\subset \bbZ^2$ we set for all integers $t\ge 0$
\[A_{t+1}=A_t\cup\{x\in\bbZ^2:\exists U\in\cU,x+U\subset A_t\}.\]
That is, at each discrete time step the sites such that the translate of a rule by the corresponding site is already fully infected also become infected, while infections never heal. Thus, given the initial condition, $\cU$-bootstrap percolation is a monotone deterministic cellular automaton. For any set $A_0\subset \bbZ^2$ we denote by $[A_0]=\bigcup_{t\ge0}A_t$ its \emph{closure} with respect to the $\cU$-bootstrap percolation process, i.e.\ the sites that are eventually infected when the initial infection is $A_0$ (we may also use this notation when $A_0 \not\subset\bbZ^2$; it will then mean $[A_0 \cap \bbZ^2]$). The best-studied setting, which is also the one relevant to us, is taking $A_0$ random with law $\mu$, so that each site is infected independently with probability $q$. In this case the most prominent questions are for which values of $q$ we have $[A_0]=\bbZ^2$ a.s.\ and, for such values of $q$, what is the typical order of magnitude of the infection time of the origin $\tau_0$ defined as for KCM.

The general $\cU$-bootstrap percolation framework only gained visibility after the work of Bollobás, Smith and Uzzell \cite{Bollobas15}. They introduced several crucial notions, which we discuss next. We invite the reader unfamiliar with these notions to systematically consult the examples in Figure \ref{fig:example} and apply the definitions to them.

We denote by $S^1=\{z\in\bbR^2,\|z\|=1\}$ the unit circle, which we standardly identify with $\bbR/2\pi\bbZ$, where $\|\cdot\|$ is the Euclidean norm of $\bbR^2$ (all distances in this work will be Euclidean, denoted by $d(\cdot,\cdot)$). A direction $u \in S^1$ will be called \emph{rational} when $\tan u \in \bbQ \cup \{\infty\}$. For a direction $u\in S^1$ and a scalar $x\in\bbR$, we define the open half-plane $\bbH_u(x)=\{y\in\bbR^2,\<u,y\><x\}$ directed by $u$ translated by $x$, and $\bar \bbH_u(x)=\{y\in\bbR^2,\<u,y\>\leq x\}$ the corresponding closed half-plane. We further set $\bbH_u=\bbH_u(0)$. A direction $u\in S^1$ is said to be \emph{unstable} (for $\cU$) if there exists $U\in\cU$ such that $U\subset \bbH_u$ and \emph{stable} otherwise. It is not hard to see (Theorem~1.10 of~\cite{Bollobas15}, Lemma~2.6 of~\cite{Bollobas14}) that the set of stable directions is a finite union of closed intervals of $S^1$ with rational endpoints. Bollobás, Smith and Uzzell \cite{Bollobas15} introduced the following partition of update families into three classes. In order to state their results and others, we will need the following standard asymptotic notation. For any real functions $f(q)$, $g(q)$ defined for $q>0$ sufficiently small, with $g>0$, we write
\begin{itemize}
    \item $f(q)=\Theta(g(q))$ when $c g(q) \leq f(q) \leq C g(q)$,
    \item $f(q)=\Omega(g(q))$ when $f(q) \geq c g(q)$,
    \item $f(q)=O(g(q))$ when $|f(q)| \leq Cg(q)$
\end{itemize}
for some constants $0 < c \leq C < +\infty$ when $q>0$ is sufficiently small. Finally, we write $f(q)=o(g(q))$ when $\frac{|f(q)|}{g(q)} \rightarrow 0$ when $q \rightarrow 0$. Let us note that all such implicit constants are allowed to depend on the update family $\cU$ and all (finite sets of) directions considered, but never on $q$.
\begin{defn}[Definition~1.3 of~\cite{Bollobas15}]
\label{def:preuniv}
An update family $\cU$ is called 
\begin{itemize}
\item \emph{supercritical} if there exists an open semicircle of unstable directions,
\item \emph{critical} if it is not supercritical, but there exists an open semicircle with a finite number of stable directions,
\item \emph{subcritical} otherwise.
\end{itemize}
\end{defn}
In \cite{Bollobas15} it was proved that for supercritical models $\tau_0=q^{-\Theta(1)}$ with high probability\footnote{This means there exist constants $0 < c < C < +\infty$ such that $\mu(q^{-c} \leq \tau_0 \leq q^{-C}) \to 1$ when $q \to 0$.
} as $q\to0$, while for critical ones $\tau_0=\exp(q^{-\Theta(1)})$. Completing the justification of Definition \ref{def:preuniv}, Balister, Bollobás, Przykucki and Smith \cite{Balister16} proved that for subcritical models $\tau_0=\infty$ with positive probability for $q$ small enough. Albeit very general, these results were much less precise than what was known for most specific models studied previously, so one may view them as only qualitatively identifying the three different possible behaviours.

Supercritical models are fairly simple from the bootstrap percolation point of view, while still very little is known in general about subcritical ones \cite{Hartarsky21}. The remaining critical models are the most pursued. Bollobás, Duminil-Copin, Morris and Smith \cite{Bollobas14} established much more precise quantitative results for this class. To state their results, we need some more notation.
\begin{defn}[Definition 1.2 of \cite{Bollobas14}]
\label{def:diff}
Let $\cU$ be a critical update family and $u\in S^1$ be a direction. Then the \emph{difficulty} of $u$, $\alpha(u)$, is defined as follows.
\begin{itemize}
    \item If $u$ is unstable, then $\alpha(u)=0$.
    \item If $u$ is an isolated stable direction (isolated in the topological sense), then
    \[
    \alpha(u)=\min\{n\in[1,\infty) \colon \exists Z\subset\bbZ^2,|Z|=n,|[\bbH_u\cup Z  ]\setminus\bbH_u|=\infty\},\]
    i.e.\ the minimal number of additional infections allowing $\bbH_u$ to infect an infinite set of sites.\footnote{This number is indeed finite for isolated stable directions (see \cite{Bollobas14}*{Lemma 2.8}, \cite{Bollobas15}*{Lemma 5.2}).}
    \item Otherwise, $\alpha(u)=\infty$.
\end{itemize}
We define the \emph{difficulty} of $\cU$ by
\begin{equation}
\label{eq:def:alpha}
\alpha(\cU)=\min_{C\in\cC}\max_{u\in C}\alpha(u)\in[1,\infty),
\end{equation}
where $\cC=\{\bbH_u\cap S^1 \colon u\in S^1\}$ is the set of open semicircles of $S^1$. 
\end{defn}

\begin{defn}[Definition 1.3 of \cite{Bollobas14}]
\label{def:balanced}
A critical update family with difficulty $\alpha$ is \emph{balanced} if there exists a closed semicircle in which all directions have difficulty at most $\alpha$ and is \emph{unbalanced} otherwise.
\end{defn}
With these definitions, the main result of Bollobás, Duminil-Copin, Morris and Smith \cite{Bollobas14}*{Theorem 1.5}, states
\begin{equation}
\label{eq:bootstrap:universality}
\tau_0=
\begin{cases}
    \exp\left(\frac{\Theta(1)}{q^{\alpha(\cU)}}\right)&\text{for critical balanced models,}\\
    \exp\left(\frac{\Theta((\log q)^2)}{q^{\alpha(\cU)}}\right)&\text{for critical unbalanced models}
\end{cases}
\end{equation}
with high probability as $q\to0$. These are the best general estimates currently known, though for some specific choices of $\cU$ sharper results are available \cites{Hartarsky19,Duminil-Copin13,Bollobas17, Gravner08,Holroyd03}.

\subsection{Universality for KCM}
\paragraph{Supercritical KCM} As we shall see, the situation for $\cU$-KCM is far more complex. Some new features already appear for supercritical models. The general setting was treated only recently by Martinelli, Morris and Toninelli \cite{Martinelli19a} and Martinelli, Toninelli and the second author \cite{Mareche20Duarte}, who identified the two relevant classes of models.
\begin{defn}[Definition 2.11 of \cite{Martinelli19a}]
A supercritical update family $\cU$ is \emph{rooted} if there exist two non-opposite stable directions and \emph{unrooted} otherwise.
\end{defn}
In \cite{Martinelli19a} the upper bounds in the following result for $\cU$-KCM were established: as $q \to 0$, 
\begin{equation}
\label{eq:supercritical}
\bbE_\mu[\tau_0]=
\begin{cases}
    \exp\left(\Theta(\log (1/q))\right)&\text{supercritical unrooted,}\\
    \exp\left(\Theta((\log (1/q))^2)\right)&\text{supercritical rooted},
\end{cases}
\end{equation}
while the lower ones were supplied in \cite{Mareche20Duarte}, the lower bound for unrooted models being trivial. The lower bound for rooted models relied mostly on a combinatorial result by the second author \cite{Mareche20combi} roughly stating that in order to infect the origin starting from a configuration in which the infection closest to the origin is at distance $d$ from it, the dynamics has to go through a configuration with at least $\log d$ infections at distance at most $d$ from the origin. This combinatorial bottleneck was originally identified in \cite{Sollich99} (see also \cite{Chung01}) for the archetypal one-dimensional supercritical rooted model, which is, perhaps, the simplest and best-studied KCM---the East model (see \cite{Faggionato13} for a review). Extending the one-dimensional result to higher dimensions in \cite{Mareche20combi} required the development of a new approach, which will be the starting point for our analysis, although we will not be able to apply the result of \cite{Mareche20combi} itself.

\paragraph{Critical KCM} Turning to critical models, a first observation is that bootstrap percolation provides an automatic lower bound for $\bbE_\mu[\tau_0]$ in the $\cU$-KCM with the same update family. In particular, it is easy to show (see \cite{Martinelli19}*{Lemma 4.3} and its improved version from \cite{Hartarsky20FA}*{Section 2}) that the expressions in \eqref{eq:bootstrap:universality} are lower bounds for $\bbE_\mu[\tau_0]$ for the corresponding $\cU$-KCM.

Led by the intuition from the results for supercritical models, Morris \cite{Morris17}, presenting his work with Martinelli and Toninelli, introduced the following notion (we simplify his terminology slightly). 
\begin{defn}[Definition 2.3 of \cite{Morris17}]
A critical update family of difficulty $\alpha$ is called 
\emph{rooted} if there exist two non-opposite directions of difficulty strictly larger than $\alpha$ and \emph{unrooted} otherwise.
\end{defn}
Initially it was believed, as explained in \cite{Morris17}, that if a critical model is unrooted, it satisfies $\bbE_\mu[\tau_0]=\exp(q^{-\alpha(\cU)+o(1)})$ as $q \to 0$, but not if it is rooted. This conjecture was made more precise by the same authors in \cite{Martinelli19a}, who then suggested a different definition of rooted/unrooted reflecting the upper bounds they proved and conjectured to be tight.

However, both conjectures were disproved by Martinelli, Toninelli and the first author \cite{Hartarsky21a}, who proved stronger upper bounds than the ones in \cite{Martinelli19a}, in particular refuting the above conjectures for some rooted models and showing that the automatic bootstrap percolation lower bound is essentially sharp for them as well. In a parallel work Toninelli and the present authors \cite{Hartarsky20} proved for all other models lower bounds essentially matching the upper ones from \cite{Martinelli19a}. Hence, the combined results of \cites{Martinelli19a,Hartarsky20,Hartarsky21a} proved the following universality picture featuring yet a different partition of critical models. As $q \to 0$, 
\[\bbE_\mu[\tau_0]=
\begin{cases}
    \exp\left(\frac{(\log(1/q))^{O(1)}}{q^{\alpha(\cU)}}\right)&\text{critical, finitely many stable directions,}\\
    \exp\left(\frac{(\log(1/q))^{O(1)}}{q^{2\alpha(\cU)}}\right)&\text{critical, infinitely many stable directions}.
\end{cases}\]

This result shows that each universality class in bootstrap percolation (models with the same value of $\alpha(\cU)$) splits into two universality classes of KCM. This should indeed be viewed as the critical counterpart of \eqref{eq:supercritical}. While the lower bound of \cite{Hartarsky20} establishing the result for models with infinitely many stable directions does reflect (and actually uses) the combinatorial bottleneck of the one-dimensional supercritical rooted East model, the upper bound of \cite{Hartarsky21a} is based on a very peculiar efficient mechanism that uses a ``quasi-local'' East-type movement resulting in a supercritical unrooted dynamics on larger length scales. However, it remained unclear whether this fairly unnatural mechanism of \cite{Hartarsky21a} and the purely East-like one used in \cite{Martinelli19a} are indeed the correct ones for all models concerned.

\subsection{Results}
Our goal is to identify the dominant relaxation mechanisms for each model, which are reflected in the scaling of $\tau_0$. Therefore, together with the companion work by the first author \cite{Hartarsky20II}, we determine $\log\bbE_\mu[\tau_0]$ up to a constant factor. In the present paper we establish all lower bounds by identifying the correct bottleneck (see Section \ref{sec:sketch}), while the companion one \cite{Hartarsky20II} provides the remaining matching upper bounds by exhibiting efficient relaxation mechanisms for all critical $\cU$-KCM.

In order to state our results we introduce the last bit of notation needed to define the seven refined universality classes of critical $\cU$-KCM, which we identify.
\begin{defn}
A critical update family of difficulty $\alpha$ is called
\begin{itemize}
    \item \emph{isotropic} if there is no direction of difficulty strictly greater than $\alpha$,
    \item \emph{semi-directed} if there exists exactly one direction of difficulty strictly greater than $\alpha$.
\end{itemize}
\end{defn}
Notice that the isotropic/semi-directed critical models form a partition of unrooted balanced critical ones.

\begin{table}[t]
    \centering
    \begin{tabular}{r|p{4.5cm}| p{2.25cm}|p{2.25cm}}
         &\multicolumn{1}{c|}{\multirow{2}{*}{infinite stable directions}}
         & \multicolumn{2}{c}{finite stable directions}\\\cline{3-4}
         & & \multicolumn{1}{c|}{rooted} & \multicolumn{1}{c}{unrooted}\\
        \hline
        unbalanced & \ref{log4} $2,4,0$ & \ref{log3} $1,3,0$ & \ref{log2} $1, 2,0$\\
        \hline
        balanced & \ref{log0} $2,0,0$ & \ref{log1} $1,1,0$ & \diagbox[dir=NE,innerwidth=2.25cm,height=3\line]{\ref{loglog} $1,0,1$\\
        \footnotesize{s.-dir.}}{\footnotesize{iso.}\\\ref{iso} $1,0,0$}
    \end{tabular}
    \caption{Classification of critical $\cU$-KCM with difficulty $\alpha$. Assuming \cite{Hartarsky20II}, for each class $\bbE_\mu[\tau_0]=\exp\left(\Theta(1)\left(\frac{1}{q^\alpha}\right)^{\beta}\left(\log \frac{1}{q}\right)^\gamma \left(\log\log\frac{1}{q}\right)^\delta\right)$ as $q \to 0$. The label of the class and the exponents $\beta,\gamma,\delta$ are indicated in that order.}
    \label{tab:universality}
\end{table}

\begin{mainthm}
\label{th:infinite}
Let $\cU$ be a critical update family with difficulty $\alpha$ and infinite number of stable directions. We have the following alternatives as $q\to0$.
\begin{enumerate}[label=(\alph*), series=th]
    \item\label{log4} If $\cU$ is unbalanced, i.e.\ there exist two opposite directions $u\in S^1$ with $\alpha(u)>\alpha$, then
    \[\bbE_\mu[\tau_0]=\exp\left(\frac{\Theta\left(\left(\log (1/q)\right)^4\right)}{q^{2\alpha}}\right).\]
    \item\label{log0} If $\cU$ is balanced, i.e.\ there do not exist two opposite directions $u\in S^1$ with $\alpha(u)>\alpha$, then
    \[\bbE_\mu[\tau_0]=\exp\left(\frac{\Theta(1)}{q^{2\alpha}}\right).\]
\end{enumerate}
\end{mainthm}
The upper bounds are proved in \cite{Martinelli19a}*{Theorem 2(a)} and \cite{Hartarsky20II}*{Theorem 1(b)} respectively. Furthermore, for one specific model of the class \ref{log4}, the Duarte model, this result was known from \cite{Mareche20Duarte}. A different proof of the lower bound for case \ref{log0} was given in \cite{Hartarsky20}.

\begin{mainthm}
\label{th:finite}
Let $\cU$ be a critical update family with difficulty $\alpha$ and finite number of stable directions. We have the following alternatives as $q\to0$.
\begin{enumerate}[resume*=th]
\item\label{log3} If $\cU$ is unbalanced rooted, i.e.\ there exist at least three directions $u\in S^1$ with $\alpha(u)>\alpha$, two of which are opposite, then 
\[\bbE_\mu[\tau_0]=\exp\left(\frac{\Theta\left(\left(\log (1/q)\right)^3\right)}{q^{\alpha}}\right).\]
\item\label{log2} If $\cU$ is unbalanced unrooted, i.e.\ there exist exactly two directions $u\in S^1$ such that $\alpha(u)>\alpha$ and they are opposite, then
\[\bbE_\mu[\tau_0]=\exp\left(\frac{\Theta\left(\left(\log(1/q)\right)^2\right)}{q^{\alpha}}\right).\]
\item\label{log1} If $\cU$ is balanced rooted, i.e.\ there exist at least two directions $u\in S^1$ with $\alpha(u)>\alpha$, but not two opposite ones, then
\[\bbE_\mu[\tau_0]=\exp\left(\frac{\Theta\left(\log (1/q)\right)}{q^{\alpha}}\right).\]
\item\label{loglog} If $\cU$ is semi-directed, i.e.\ there exists exactly one direction $u\in S^1$ such that $\alpha(u)>\alpha$, then
\[\bbE_\mu[\tau_0]=\exp\left(\frac{\Theta\left(\log\log (1/q)\right)}{q^{\alpha}}\right).\]
\item\label{iso} If $\cU$ is isotropic, i.e.\ there exists no direction $u\in S^1$ such that $\alpha(u)>\alpha$, then
\[\bbE_\mu[\tau_0]=\exp\left(\frac{\Theta(1)}{q^{\alpha}}\right).\]
\end{enumerate}
\end{mainthm}

The upper bound in \ref{log3} was proved in \cite{Hartarsky21a}*{Theorem 1}, while the remaining upper bounds are from \cite{Hartarsky20II}*{Theorem 1}. The lower bounds for cases \ref{log2} and \ref{iso} follow automatically from bootstrap percolation results and \cite{Martinelli19}*{Lemma 4.3} as discussed in the previous section.

The classification results, assuming the remaining matching bounds of \cite{Hartarsky20II}, are summarised in Table \ref{tab:universality}. In addition, a simple representative of each class is given in Figure \ref{fig:example} for the reader's convenience.

\begin{rem}
The lower bounds we prove in Theorems \ref{th:infinite} and \ref{th:finite} for $\bbE_\mu[\tau_0]$ also hold for another important characteristic timescale of the corresponding $\cU$-KCM, the \emph{relaxation time} $\trel$ (see e.g.\ \cite{Hartarsky20}*{Definition 2.5}), which is another measure of the speed at which the memory of the initial state is lost. Indeed, \cite{Martinelli19a}*{Equation (2.8)} yields $\trel \geq q\bbE_\mu[\tau_0]$. 
\end{rem}

\begin{figure}
	\centering
	\begin{subfigure}{0.48\textwidth}
		\centering
		\begin{subfigure}{0.49\textwidth}
		\centering
		\begin{subfigure}{0.48\textwidth}
		\centering
		\begin{tikzpicture}[line cap=round,line join=round,x=0.22\textwidth,y=0.22\textwidth]
				\clip (-2.2,-1.2) rectangle (2.2,1.2);
				\draw [color=gray, xstep=1,ystep=1] (-2,-1) grid (2,1);
				\fill (-1,0) circle (2pt);
				\fill (0,1) circle (2pt);
				\draw (0,0) node[cross=3pt,rotate=0] {};
		\end{tikzpicture}
		\end{subfigure}
		\begin{subfigure}{0.48\textwidth}
		\centering
		\begin{tikzpicture}[line cap=round,line join=round,x=0.22\textwidth,y=0.22\textwidth]
				\clip (-2.2,-1.2) rectangle (2.2,1.2);
		\end{tikzpicture}
		\end{subfigure}
		\\
		\begin{subfigure}{0.48\textwidth}
		\centering
		\begin{tikzpicture}[line cap=round,line join=round,x=0.22\textwidth,y=0.22\textwidth]
				\clip (-2.2,-1.2) rectangle (2.2,1.2);
				\draw [color=gray, xstep=1,ystep=1] (-2,-1) grid (2,1);
				\fill (-1,0) circle (2pt);
				\fill (0,-1) circle (2pt);
				\fill (-2,0) circle (2pt);
				\draw (0,0) node[cross=3pt,rotate=0] {};
		\end{tikzpicture}
		\end{subfigure}
		\begin{subfigure}{0.48\textwidth}
		\centering
		\begin{tikzpicture}[line cap=round,line join=round,x=0.22\textwidth,y=0.22\textwidth]
				\clip (-2.2,-1.2) rectangle (2.2,1.2);
				\draw [color=gray, xstep=1,ystep=1] (-2,-1) grid (2,1);
				\fill (1,0) circle (2pt);
				\fill (0,-1) circle (2pt);
				\fill (2,0) circle (2pt);
				\draw (0,0) node[cross=3pt,rotate=0] {};
		\end{tikzpicture}
		\end{subfigure}		
		\end{subfigure}
		\begin{subfigure}{0.49\textwidth}
		\centering
		\begin{tikzpicture}[line cap=round,line join=round,x=0.35\textwidth,y=0.35\textwidth]
				\draw(0,0) circle (0.35\textwidth);
				\draw (0,0)-- (1,0);
				\draw (0,1)-- (0,0);
				\draw (0,0)-- (-1,0);
				\draw (0,0)-- (0,-1);
			    \fill [color=ffqqqq] (0,1) circle (3pt) node[anchor=south,black] {$2$};
			    \fill [color=ffqqqq] (1,0) circle (3pt) node[anchor=west,black] {$1$};
			    \draw [shift={(0,0)},line width=2pt,color=ffqqqq]
					plot[domain=pi:4.71,variable=\t]
					({cos(\t r)},{sin(\t r)});
				\draw (-0.9,-0.9) node {$\infty$};
		\end{tikzpicture}
		\end{subfigure}
		\caption{\label{fig:log4}A model of Theorem~\ref{th:infinite}\ref{log4}.}
	\end{subfigure}\quad%
	\begin{subfigure}{0.48\textwidth}
		\centering
		\begin{subfigure}{0.49\textwidth}
		\centering
		\begin{subfigure}{0.48\textwidth}
		\centering
		\begin{tikzpicture}[line cap=round,line join=round,x=0.22\textwidth,y=0.22\textwidth]
				\clip (-1.2,-1.2) rectangle (1.2,1.2);
				\draw [color=gray, xstep=1,ystep=1] (-1,-1) grid (1,1);
				\fill (-1,0) circle (2pt);
				\fill (0,1) circle (2pt);
				\draw (0,0) node[cross=3pt,rotate=0] {};
		\end{tikzpicture}
		\end{subfigure}
		\begin{subfigure}{0.48\textwidth}
		\centering
		\begin{tikzpicture}[line cap=round,line join=round,x=0.22\textwidth,y=0.22\textwidth]
				\clip (-1.2,-1.2) rectangle (1.2,1.2);
		\end{tikzpicture}
		\end{subfigure}
		\\
		\begin{subfigure}{0.48\textwidth}
		\centering
		\begin{tikzpicture}[line cap=round,line join=round,x=0.22\textwidth,y=0.22\textwidth]
				\clip (-1.2,-1.2) rectangle (1.2,1.2);
				\draw [color=gray, xstep=1,ystep=1] (-1,-1) grid (1,1);
				\fill (-1,0) circle (2pt);
				\fill (0,-1) circle (2pt);
				\draw (0,0) node[cross=3pt,rotate=0] {};
		\end{tikzpicture}
		\end{subfigure}
		\begin{subfigure}{0.48\textwidth}
		\centering
		\begin{tikzpicture}[line cap=round,line join=round,x=0.22\textwidth,y=0.22\textwidth]
				\clip (-1.2,-1.2) rectangle (1.2,1.2);
				\draw [color=gray, xstep=1,ystep=1] (-1,-1) grid (1,1);
				\fill (1,0) circle (2pt);
				\fill (0,-1) circle (2pt);
				\draw (0,0) node[cross=3pt,rotate=0] {};
		\end{tikzpicture}
		\end{subfigure}		
		\end{subfigure}
		\begin{subfigure}{0.49\textwidth}
		\centering
		\begin{tikzpicture}[line cap=round,line join=round,x=0.35\textwidth,y=0.35\textwidth]
				\draw(0,0) circle (0.35\textwidth);
				\draw (0,0)-- (1,0);
				\draw (0,1)-- (0,0);
				\draw (0,0)-- (-1,0);
				\draw (0,0)-- (0,-1);
			    \fill [color=ffqqqq] (0,1) circle (3pt) node[anchor=south,black] {$1$};
			    \fill [color=ffqqqq] (1,0) circle (3pt) node[anchor=west,black] {$1$};
			    \draw [shift={(0,0)},line width=2pt,color=ffqqqq]
					plot[domain=pi:4.71,variable=\t]
					({cos(\t r)},{sin(\t r)});
				\draw (-0.9,-0.9) node {$\infty$};
		\end{tikzpicture}
		\end{subfigure}
		\caption{\label{fig:log0}A model of Theorem~\ref{th:infinite}\ref{log0}.}
	\end{subfigure}
	\\
	
	\begin{subfigure}{0.48\textwidth}
		\centering
		\begin{subfigure}{0.49\textwidth}
		\centering
		\begin{subfigure}{0.48\textwidth}
		\centering
		\begin{tikzpicture}[line cap=round,line join=round,x=0.22\textwidth,y=0.22\textwidth]
				\clip (-2.2,-2.2) rectangle (2.2,2.2);
				\draw [color=gray, xstep=1,ystep=1] (-2,-2) grid (2,2);
				\fill (-1,0) circle (2pt);
				\fill (0,1) circle (2pt);
				\fill (-2,0) circle (2pt);
				\draw (0,0) node[cross=3pt,rotate=0] {};
		\end{tikzpicture}
		\end{subfigure}
		\begin{subfigure}{0.48\textwidth}
		\centering
		\begin{tikzpicture}[line cap=round,line join=round,x=0.22\textwidth,y=0.22\textwidth]
				\clip (-2.2,-2.2) rectangle (2.2,2.2);
				\draw [color=gray, xstep=1,ystep=1] (-2,-2) grid (2,2);
				\fill (1,0) circle (2pt);
				\fill (0,1) circle (2pt);
				\fill (0,2) circle (2pt);
				\fill (2,0) circle (2pt);
				\draw (0,0) node[cross=3pt,rotate=0] {};
		\end{tikzpicture}
		\end{subfigure}
				\\
		\begin{subfigure}{0.48\textwidth}
		\centering
		\begin{tikzpicture}[line cap=round,line join=round,x=0.22\textwidth,y=0.22\textwidth]
				\clip (-2.2,-2.2) rectangle (2.2,2.2);
				\draw [color=gray, xstep=1,ystep=1] (-2,-2) grid (2,2);
				\fill (-1,0) circle (2pt);
				\fill (0,-1) circle (2pt);
				\fill (-2,0) circle (2pt);
				\draw (0,0) node[cross=3pt,rotate=0] {};
		\end{tikzpicture}
		\end{subfigure}
		\begin{subfigure}{0.48\textwidth}
		\centering
		\begin{tikzpicture}[line cap=round,line join=round,x=0.22\textwidth,y=0.22\textwidth]
				\clip (-2.2,-2.2) rectangle (2.2,2.2);
				\draw [color=gray, xstep=1,ystep=1] (-2,-2) grid (2,2);
				\fill (1,0) circle (2pt);
				\fill (0,-1) circle (2pt);
				\fill (0,-2) circle (2pt);
				\fill (2,0) circle (2pt);
				\draw (0,0) node[cross=3pt,rotate=0] {};
		\end{tikzpicture}
		\end{subfigure}
		\end{subfigure}
		\begin{subfigure}{0.49\textwidth}
		\centering
		\begin{tikzpicture}[line cap=round,line join=round,x=0.35\textwidth,y=0.35\textwidth]
				\draw(0,0) circle (0.35\textwidth);
				\draw (0,0)-- (1,0);
				\draw (0,1)-- (0,0);
				\draw (0,0)-- (-1,0);
				\draw (0,0)-- (0,-1);
			    \fill [color=ffqqqq] (0,1) circle (3pt) node[anchor=south,black] {$2$};
			    \fill [color=ffqqqq] (1,0) circle (3pt) node[anchor=west,black] {$1$};
				\fill [color=ffqqqq] (-1,0) circle (3pt) node[anchor=east,black] {$2$};
				\fill [color=ffqqqq] (0,-1) circle (3pt) node[anchor=north,black] {$2$};
		\end{tikzpicture}
		\end{subfigure}
		\caption{\label{fig:log3}A model of Theorem~\ref{th:finite}\ref{log3}.}
	\end{subfigure}
	\begin{subfigure}{0.48\textwidth}
		\centering
		\begin{subfigure}{0.49\textwidth}
		\centering
		\begin{subfigure}{0.48\textwidth}
		\centering
		\begin{tikzpicture}[line cap=round,line join=round,x=0.22\textwidth,y=0.22\textwidth]
				\clip (-2.2,-1.2) rectangle (2.2,1.2);
				\draw [color=gray, xstep=1,ystep=1] (-2,-1) grid (2,1);
				\fill (-1,0) circle (2pt);
				\fill (0,1) circle (2pt);
				\fill (-2,0) circle (2pt);
				\draw (0,0) node[cross=3pt,rotate=0] {};
		\end{tikzpicture}
		\end{subfigure}
		\begin{subfigure}{0.48\textwidth}
		\centering
		\begin{tikzpicture}[line cap=round,line join=round,x=0.22\textwidth,y=0.22\textwidth]
				\clip (-2.2,-1.2) rectangle (2.2,1.2);
				\draw [color=gray, xstep=1,ystep=1] (-2,-1) grid (2,1);
				\fill (1,0) circle (2pt);
				\fill (0,1) circle (2pt);
				\fill (2,0) circle (2pt);
				\draw (0,0) node[cross=3pt,rotate=0] {};
		\end{tikzpicture}
		\end{subfigure}
		\\
		\begin{subfigure}{0.48\textwidth}
		\centering
		\begin{tikzpicture}[line cap=round,line join=round,x=0.22\textwidth,y=0.22\textwidth]
				\clip (-2.2,-1.2) rectangle (2.2,1.2);
				\draw [color=gray, xstep=1,ystep=1] (-2,-1) grid (2,1);
				\fill (-1,0) circle (2pt);
				\fill (0,-1) circle (2pt);
				\fill (-2,0) circle (2pt);
				\draw (0,0) node[cross=3pt,rotate=0] {};
		\end{tikzpicture}
		\end{subfigure}
		\begin{subfigure}{0.48\textwidth}
		\centering
		\begin{tikzpicture}[line cap=round,line join=round,x=0.22\textwidth,y=0.22\textwidth]
				\clip (-2.2,-1.2) rectangle (2.2,1.2);
				\draw [color=gray, xstep=1,ystep=1] (-2,-1) grid (2,1);
				\fill (1,0) circle (2pt);
				\fill (0,-1) circle (2pt);
				\fill (2,0) circle (2pt);
				\draw (0,0) node[cross=3pt,rotate=0] {};
		\end{tikzpicture}
		\end{subfigure}		
		\end{subfigure}
		\begin{subfigure}{0.49\textwidth}
		\centering
		\begin{tikzpicture}[line cap=round,line join=round,x=0.35\textwidth,y=0.35\textwidth]
				\draw(0,0) circle (0.35\textwidth);
				\draw (0,0)-- (1,0);
				\draw (0,1)-- (0,0);
				\draw (0,0)-- (-1,0);
				\draw (0,0)-- (0,-1);
			    \fill [color=ffqqqq] (0,1) circle (3pt) node[anchor=south,black] {$2$};
			    \fill [color=ffqqqq] (0,-1) circle (3pt) node[anchor=north,black] {$2$};
			    \fill [color=ffqqqq] (1,0) circle (3pt) node[anchor=west,black] {$1$};
			    \fill [color=ffqqqq] (-1,0) circle (3pt) node[anchor=east,black] {$1$};
		\end{tikzpicture}
		\end{subfigure}
		\caption{\label{fig:log2}A model of Theorem~\ref{th:finite}\ref{log2}.}
	\end{subfigure}
	\\
	\begin{subfigure}{0.48\textwidth}
		\centering
		\begin{subfigure}{0.49\textwidth}
		\centering
		\begin{subfigure}{0.48\textwidth}
		\centering
		\begin{tikzpicture}[line cap=round,line join=round,x=0.22\textwidth,y=0.22\textwidth]
				\clip (-2.2,-2.2) rectangle (2.2,2.2);
				\draw [color=gray, xstep=1,ystep=1] (-2,-2) grid (2,2);
				\fill (-1,0) circle (2pt);
				\fill (0,1) circle (2pt);
				\fill (-2,0) circle (2pt);
				\draw (0,0) node[cross=3pt,rotate=0] {};
		\end{tikzpicture}
		\end{subfigure}
		\begin{subfigure}{0.48\textwidth}
		\centering
		\begin{tikzpicture}[line cap=round,line join=round,x=0.22\textwidth,y=0.22\textwidth]
				\clip (-2.2,-2.2) rectangle (2.2,2.2);
				\draw [color=gray, xstep=1,ystep=1] (-2,-2) grid (2,2);
				\fill (1,0) circle (2pt);
				\fill (2,0) circle (2pt);
				\fill (0,1) circle (2pt);
				\fill (0,2) circle (2pt);
				\draw (0,0) node[cross=3pt,rotate=0] {};
		\end{tikzpicture}
		\end{subfigure}
				\\
		\begin{subfigure}{0.48\textwidth}
		\centering
		\begin{tikzpicture}[line cap=round,line join=round,x=0.22\textwidth,y=0.22\textwidth]
				\clip (-2.2,-2.2) rectangle (2.2,2.2);
				\draw [color=gray, xstep=1,ystep=1] (-2,-2) grid (2,2);
				\fill (-1,0) circle (2pt);
				\fill (0,-1) circle (2pt);
				\draw (0,0) node[cross=3pt,rotate=0] {};
		\end{tikzpicture}
		\end{subfigure}
		\begin{subfigure}{0.48\textwidth}
		\centering
		\begin{tikzpicture}[line cap=round,line join=round,x=0.22\textwidth,y=0.22\textwidth]
				\clip (-2.2,-2.2) rectangle (2.2,2.2);
				\draw [color=gray, xstep=1,ystep=1] (-2,-2) grid (2,2);
				\fill (1,0) circle (2pt);
				\fill (0,-1) circle (2pt);
				\fill (0,-2) circle (2pt);
				\draw (0,0) node[cross=3pt,rotate=0] {};
		\end{tikzpicture}
		\end{subfigure}
		\end{subfigure}
		\begin{subfigure}{0.49\textwidth}
		\centering
		\begin{tikzpicture}[line cap=round,line join=round,x=0.35\textwidth,y=0.35\textwidth]
				\draw(0,0) circle (0.35\textwidth);
				\draw (0,0)-- (1,0);
				\draw (0,1)-- (0,0);
				\draw (0,0)-- (-1,0);
				\draw (0,0)-- (0,-1);
			    \fill [color=ffqqqq] (0,1) circle (3pt) node[anchor=south,black] {$1$};
			    \fill [color=ffqqqq] (1,0) circle (3pt) node[anchor=west,black] {$1$};
				\fill [color=ffqqqq] (-1,0) circle (3pt) node[anchor=east,black] {$2$};
				\fill [color=ffqqqq] (0,-1) circle (3pt) node[anchor=north,black] {$2$};
		\end{tikzpicture}
		\end{subfigure}
	\caption{\label{fig:log1}A model of Theorem~\ref{th:finite}\ref{log1}.}
\end{subfigure}\quad%
	\begin{subfigure}{0.48\textwidth}
		\centering
		\begin{subfigure}{0.49\textwidth}
		\centering
		\begin{subfigure}{0.48\textwidth}
		\centering
		\begin{tikzpicture}[line cap=round,line join=round,x=0.22\textwidth,y=0.22\textwidth]
				\clip (-2.2,-1.2) rectangle (2.2,1.2);
				\draw [color=gray, xstep=1,ystep=1] (-2,-1) grid (2,1);
				\fill (-1,0) circle (2pt);
				\fill (0,1) circle (2pt);
				\fill (-2,0) circle (2pt);
				\draw (0,0) node[cross=3pt,rotate=0] {};
		\end{tikzpicture}
		\end{subfigure}
		\begin{subfigure}{0.48\textwidth}
		\centering
		\begin{tikzpicture}[line cap=round,line join=round,x=0.22\textwidth,y=0.22\textwidth]
				\clip (-2.2,-1.2) rectangle (2.2,1.2);
				\draw [color=gray, xstep=1,ystep=1] (-2,-1) grid (2,1);
				\fill (1,0) circle (2pt);
				\fill (2,0) circle (2pt);
				\fill (0,1) circle (2pt);
				\draw (0,0) node[cross=3pt,rotate=0] {};
		\end{tikzpicture}
		\end{subfigure}
				\\
		\begin{subfigure}{0.48\textwidth}
		\centering
		\begin{tikzpicture}[line cap=round,line join=round,x=0.22\textwidth,y=0.22\textwidth]
				\clip (-2.2,-1.2) rectangle (2.2,1.2);
				\draw [color=gray, xstep=1,ystep=1] (-2,-1) grid (2,1);
				\fill (-1,0) circle (2pt);
				\fill (0,-1) circle (2pt);
				\draw (0,0) node[cross=3pt,rotate=0] {};
		\end{tikzpicture}
		\end{subfigure}
		\begin{subfigure}{0.48\textwidth}
		\centering
		\begin{tikzpicture}[line cap=round,line join=round,x=0.22\textwidth,y=0.22\textwidth]
				\clip (-2.2,-1.2) rectangle (2.2,1.2);
				\draw [color=gray, xstep=1,ystep=1] (-2,-1) grid (2,1);
				\fill (1,0) circle (2pt);
				\fill (0,-1) circle (2pt);
				\draw (0,0) node[cross=3pt,rotate=0] {};
		\end{tikzpicture}
		\end{subfigure}
		\end{subfigure}
		\begin{subfigure}{0.49\textwidth}
		\centering
		\begin{tikzpicture}[line cap=round,line join=round,x=0.35\textwidth,y=0.35\textwidth]
				\draw(0,0) circle (0.35\textwidth);
				\draw (0,0)-- (1,0);
				\draw (0,1)-- (0,0);
				\draw (0,0)-- (-1,0);
				\draw (0,0)-- (0,-1);
			    \fill [color=ffqqqq] (0,1) circle (3pt) node[anchor=south,black] {$1$};
			    \fill [color=ffqqqq] (1,0) circle (3pt) node[anchor=west,black] {$1$};
				\fill [color=ffqqqq] (-1,0) circle (3pt) node[anchor=east,black] {$1$};
				\fill [color=ffqqqq] (0,-1) circle (3pt) node[anchor=north,black] {$2$};
		\end{tikzpicture}
		\end{subfigure}
		\caption{\label{fig:loglog}A model of Theorem~\ref{th:finite}\ref{loglog}.}
	\end{subfigure}
	\\
	\begin{subfigure}{0.48\textwidth}
		\centering
		\begin{subfigure}{0.49\textwidth}
		\centering
		\begin{subfigure}{0.48\textwidth}
		\centering
		\begin{tikzpicture}[line cap=round,line join=round,x=0.22\textwidth,y=0.22\textwidth]
				\clip (-1.2,-1.2) rectangle (1.2,1.2);
				\draw [color=gray, xstep=1,ystep=1] (-1,-1) grid (1,1);
				\fill (-1,0) circle (2pt);
				\fill (0,1) circle (2pt);
				\draw (0,0) node[cross=3pt,rotate=0] {};
		\end{tikzpicture}
		\end{subfigure}
		\begin{subfigure}{0.48\textwidth}
		\centering
		\begin{tikzpicture}[line cap=round,line join=round,x=0.22\textwidth,y=0.22\textwidth]
				\clip (-1.2,-1.2) rectangle (1.2,1.2);
				\draw [color=gray, xstep=1,ystep=1] (-1,-1) grid (1,1);
				\fill (1,0) circle (2pt);
				\fill (0,1) circle (2pt);
				\draw (0,0) node[cross=3pt,rotate=0] {};
		\end{tikzpicture}
		\end{subfigure}
		\\
		\begin{subfigure}{0.48\textwidth}
		\centering
		\begin{tikzpicture}[line cap=round,line join=round,x=0.22\textwidth,y=0.22\textwidth]
				\clip (-1.2,-1.2) rectangle (1.2,1.2);
				\draw [color=gray, xstep=1,ystep=1] (-1,-1) grid (1,1);
				\fill (-1,0) circle (2pt);
				\fill (0,-1) circle (2pt);
				\draw (0,0) node[cross=3pt,rotate=0] {};
		\end{tikzpicture}
		\end{subfigure}
		\begin{subfigure}{0.48\textwidth}
		\centering
		\begin{tikzpicture}[line cap=round,line join=round,x=0.22\textwidth,y=0.22\textwidth]
				\clip (-1.2,-1.2) rectangle (1.2,1.2);
				\draw [color=gray, xstep=1,ystep=1] (-1,-1) grid (1,1);
				\fill (1,0) circle (2pt);
				\fill (0,-1) circle (2pt);
				\draw (0,0) node[cross=3pt,rotate=0] {};
		\end{tikzpicture}
		\end{subfigure}		
		\end{subfigure}
		\begin{subfigure}{0.49\textwidth}
		\centering
		\begin{tikzpicture}[line cap=round,line join=round,x=0.35\textwidth,y=0.35\textwidth]
				\draw(0,0) circle (0.35\textwidth);
				\draw (0,0)-- (1,0);
				\draw (0,1)-- (0,0);
				\draw (0,0)-- (-1,0);
				\draw (0,0)-- (0,-1);
			    \fill [color=ffqqqq] (0,1) circle (3pt) node[anchor=south,black] {$1$};
			    \fill [color=ffqqqq] (0,-1) circle (3pt) node[anchor=north,black] {$1$};
			    \fill [color=ffqqqq] (1,0) circle (3pt) node[anchor=west,black] {$1$};
			    \fill [color=ffqqqq] (-1,0) circle (3pt) node[anchor=east,black] {$1$};
		\end{tikzpicture}
		\end{subfigure}
		\caption{\label{fig:iso}A model of Theorem~\ref{th:finite}\ref{iso}.}
	\end{subfigure}
	\caption{Representative models of each of the seven refined universality classes of critical $\cU$-KCM. For each one the update rules are depicted on the left with $0$ marked by a cross and the sites of the rule denoted by dots. The figure on the right gives the stable directions, which are thickened and have their difficulties next to them. The isolated stable directions are marked by dots. In all cases the difficulty $\alpha$ of the model is $1$, as witnessed by the right-hand open semicircle.}\label{fig:example}
\end{figure}

\subsection{Comments}
Let us emphasise that, contrary to what is the case in bootstrap percolation, where the exact asymptotics and corrective terms of $\log \tau_0$ are sometimes known, for KCM there exists only one model (the Duarte model treated by Martinelli, Toninelli and the second author \cite{Mareche20Duarte}) for which $\log\bbE_\mu[\tau_0]$ is determined up to a constant factor and none for which the exact asymptotics is known. Thereby, the combined results of \cites{Hartarsky20II,Martinelli19a,Hartarsky21a} and the present work improve the best known results for all critical $\cU$-KCM except the Duarte-KCM and give a new proof of the best known results for that last model. In view of the state of the art in the simpler setting of $\cU$-bootstrap percolation, it does not seem currently feasible to pursue higher precision in full generality for critical $\cU$-KCM. Theorem \ref{th:infinite} establishes the lower bounds of \cite{Hartarsky20}*{Conjecture 7.1}, while Theorem \ref{th:finite} proves the ones of \cite{Hartarsky21a}*{Conjecture 6.1} (that conjecture was intentionally not as precise concerning the semi-directed case \ref{loglog}, which seemed quite mysterious at the time of writing of the conjecture).

What is more, the present work treats all lower bounds in a systematic and comprehensive way, by showing that the main obstacle for the dynamics in all cases corresponds to a combinatorial bottleneck similar to the one of the two-dimensional East model (see Section \ref{sec:sketch} for more details). Moreover, the core of our argument carries over to higher dimensions with no further difficulty, thus reducing such lower bounds for higher-dimensional KCM to estimates for their bootstrap percolation counterparts.

\subsection{Organisation of the paper}
The remainder of the paper is organised as follows. In Section \ref{sec:sketch} we provide an outline of the ideas of the proof. The core of the proof is Section \ref{subsec:general}, where we establish the general combinatorial bottleneck for KCM dynamics. In Section \ref{sec:proof} we deduce our main results, Theorems \ref{th:infinite} and \ref{th:finite}, very directly from Proposition \ref{prop:main}. It is not until Section \ref{sec:proof} that the distinction between the different classes of models becomes relevant. In addition to the central Proposition \ref{prop:main}, Section \ref{sec:proof} requires several estimates regarding bootstrap percolation, which are established in the appendices. Appendix \ref{app:spanning} is rather standard and provides bounds on the probability of ``spanning'' by uniting arguments of Bollob\'as, Duminil-Copin, Morris and Smith \cite{Bollobas14} and of Toninelli and the authors \cite{Hartarsky20} without significant new input. This appendix may still be of interest to bootstrap percolation specialists, as it gives a simplification of the most technical part of the proof of the main result of \cite{Bollobas14}.
Finally, in Appendix \ref{app:crossing} we establish bounds on the probability of a notion of ``crossing'' inspired from \cite{Bollobas14} but modified to suit our setting (see Section \ref{sec:sketch}), which are not particularly difficult, given Appendix \ref{app:spanning}.

\section{Outline of the proof}
\label{sec:sketch}
Let us outline the main highlights of the proof, emphasising new ideas with respect to our main sources of inspiration \cites{Mareche20combi,Bollobas14,Hartarsky20,Hartarsky20II}. The differences between the various universality classes will be much more apparent in \cite{Hartarsky20II}, where efficient mechanisms for the infection of the origin are implemented for each class. Since one of the main virtues of our work is that the core argument is independent of the universality class, we invite the reader willing to understand the origin of the choices of length scales appearing in Section \ref{sec:proof} to consult \cite{Hartarsky20II}, which is also our inspiration for choosing them. For the sake of concreteness, unless otherwise indicated, in this section we restrict our attention to the representative of the class \ref{log1} of rooted balanced models with finitely many stable directions given in Figure~\ref{fig:log1}. 

The proof relies on the notion of ``bottleneck'': we show that before the origin can be infected, the dynamics has to go through a set of configurations with a probability small enough so that this does not happen for a very long time.

Morally speaking, in this model the smallest mobile entity (``droplet'') is an infected square of size roughly $1/q$. Indeed, typically on its right and top sides one can find an infection, which allows it to infect the column of sites on its right and the row of sites above it. However, it is essentially impossible for the infection to grow down or left, as this requires two consecutive infections and those are typically only available at distance $1/q^2$ from the droplet. We will only work in a region $R$ of size $1/q^{7/4}$ around the droplet, so, morally, such couples of infections are not available. Thus, we can think of the droplet as performing the following simpler dynamics. If there is a droplet present at a certain position, it may create/destroy another one above it and to its right. Therefore, the droplets follow the dynamics of the two-dimensional supercritical rooted KCM called East model.

As mentioned in the introduction, for supercritical rooted models, including the (two-dimensional) East one, the second author established in \cite{Mareche20combi} the following bottleneck: in order for an infection (representing a droplet in our original model) to reach the center of a box $R$ of size $3^{n}$ initially fully healthy it is necessary to visit a configuration with at least $n$ infections simultaneously present in the box. This was achieved by an inductive argument that we describe next. It is enough to show that if $R$ is initially fully healthy and we can only have strictly less than $n$ infections at the same time in $R$, we cannot reach a configuration with all infections in the middle of $R$. Indeed, this implies that there are strictly less than $n-1$ infections at the same time in the middle of $R$, and one can make an induction on $n$. To show that we cannot reach a configuration with all infections in the middle of $R$, by reversibility we may instead prove that if we start with infections only in the middle of $R$, but we are never allowed to have $n$ infections simultaneously in $R$, we cannot reach a configuration fully healthy in $R$. The idea is to ensure that for any path of the dynamics starting with all infections in the middle of $R$ in which we never have $n$ infections at the same time in $R$, the following two conditions remain true at all times. Firstly, a buffer zone (see the shaded frame $B$ in Figure \ref{fig:main}) with no infections remains intact. Secondly, there is always an infection in the internal region encircled by the buffer, so the dynamics cannot reach a configuration completely healthy in $R$.

In order to achieve that, we use a second induction, on the step of the path. We know that so far an infection remains trapped in the internal region encircled by the buffer, so we only have $n-1$ infections available for disrupting the buffer from the outside, which is impossible by induction on $n$. Therefore, it suffices to show that we may not disrupt the buffer from the inside either. By projecting the two-dimensional East model on each axis it is clear that no infection can enter the left and bottom parts of the buffer from the inside, and the projections of the lowest and leftmost particles in the region inside the buffer need to remain where they were initially. The right part of the buffer (and similarly the top one) cannot be reached from the inside, because at least one infection needs to remain as far left as the leftmost initial one was, so we only have $n-1$ infections with which to reach the right part of the buffer, which is impossible by induction on $n$.

On a very high level we will proceed in the same way for critical models. However, there are several obvious problems in making the above reasoning rigorous for these models (in fact we do not think that a direct mapping to the two-dimensional East model can be made rigorous for our purposes). Firstly, we said above that the smallest mobile entities, ``droplets,'' were infected squares of size $1/q$, but the smallest mobile entities are actually more complicated. One needs to identify an event, which says whether or not something is a droplet and this event should be both deterministically necessary for infection to spread and sufficiently unlikely, so that having many droplets at the same time has probability small enough to be a good bottleneck. It turns out that the notion of ``spanning'' introduced in \cite{Bollobas14}, following \cite{Cerf99} is flexible enough for us. Roughly speaking (see Definition \ref{def:span}), a droplet is spanned if the infections present inside it are sufficient to infect a connected set touching all its sides. We call a droplet critical if it has size roughly $1/q$. It is known from \cite{Bollobas14} and obtained again in Appendix \ref{app:spanning} in a more adapted form that the probability of a specific critical droplet being spanned is roughly $\exp(-1/q)$. Unfortunately, given a configuration, spanned critical droplets may overlap, so in order to obtain good bounds on the probability of the configuration, one needs to consider disjointly occurring ones. We may then define the number of spanned critical droplets as the maximal number of disjointly occurring ones.

Having fixed these notions, we encounter a more significant issue---the (spanned critical) droplets may move a bit without creating another droplet, by changing their internal structure. Worse yet, they are not really forbidden to move left or down, but simply are not likely to be able to do so wherever they want: it depends on the dynamic environment. Indeed, being able to move by a single step down is allowed by the presence of a couple of infections on the side of the droplet, which has probability only as small as $q$ and is by far not something we can ensure never happens up to time $T=\exp(\log (1/q)/q)$.

In order to handle these problems we introduce the crucial notion of crossing (not to be confused e.g.\ with the one of \cite{Bollobas14}). Consider a vertical strip $S$ of width $1/q^{3/2}$ of our domain, $R$, which is a square of size $1/q^{7/4}$. Roughly speaking (see Definition \ref{def:crossing} for a more precise statement), we say that $S$ has a crossing if the following two events occur. Firstly, the infections in $S$ together with the entire half-plane to the right of $S$ are enough to infect a path from right to left in $S$ (this is essentially the notion of crossing in \cite{Bollobas14}). Secondly, $S$ does not contain a spanned critical droplet. Notice that these two events go opposite ways---the former is favoured by infections, while the latter is not. In Appendix \ref{app:crossing} we show that the probability of a crossing decays exponentially with the width of $S$ at our scales.

Having established such a bound on the probability of crossings, we may safely assume (it happens with high probability) that they never occur until time $T$ and this is the property we use to formalise the intuition that ``moving down or left is impossible.'' More precisely, assume that initially the only critical droplet is on the right of $S$ and $S$ never has a crossing. Then, simply because the KCM dynamics can never infect more than what bootstrap percolation can, starting from the same initial condition, the droplet will not be able to reach the left side of $S$. Indeed, if it could, there would be a ``trail'' of infectable sites from the right of $S$ to its left, which would imply a crossing.

\section{Supercritical rooted dynamics of droplets}
\label{subsec:general}
\subsection{Setting and preliminaries} \label{subsec:setting}Let $\mathcal{U}$ be an update family. Assuming they exist, we further fix two non-collinear rational stable directions $u_1$ and $u_2$. We set $u_3=u_1+\pi$, $u_4=u_2+\pi$ and $\cT=\{u_1,u_2,u_3,u_4\}$. We will simply call \emph{parallelogram} a set of the form 
\begin{align*}
R(a,b;c,d)&{}=\left\{x\in\bbR^2\,|\,\<x,u_3\>\in[a,c],\<x,u_4\>\in[b,d]\right\}\\
&{}=\bar\bbH_{u_1}(-a) \cap \bar\bbH_{u_2}(-b) \cap \bar\bbH_{u_3}(c) \cap \bar\bbH_{u_4}(d)
\end{align*}
for real numbers $a \leq b,c \leq d$ and denote by $\ring R(a,b;c,d)$ its topological interior. For parallelograms we will systematically extend definitions by translation and interchange of $u_1$ and $u_2$ (resp. $u_3$ and $u_4$).

Finally,
\[C_6\gg C_5\gg C_2'\gg C_1\gg r=\max\{\|s-s'\| \,|\, s,s' \in U\cup\{0\},U \in \mathcal{U}\}\]
are constants not depending on $q$, but only on $\cU$ and $\cT$, each one sufficiently large with respect to functions of the next.\footnote{We use $C_2'$ instead of $C_2$ and avoid $C_3$ and $C_4$ for coherence with the appendices.} In addition, $K=K(q)$ will be a function such that $K \to +\infty$ when $q \to 0$, so that $K\gg C_6$. The value of $K$ will depend on the class of models studied, hence will be specified in Section \ref{sec:proof}. Furthermore, we systematically assume that $q$ is small enough, as we are interested in $q\to0$.
\begin{defn}[\cite{Bollobas14}*{Definitions 2.3 and 2.4}]
\label{def:span}
A set $Z\subset \bbZ^2$ is \emph{strongly connected} if it is connected in the graph with vertex set $\bbZ^2$ defined by $x\sim y$ if $\|x-y\|\le C_2'$.

We say a parallelogram is \emph{critical} when its diameter is contained between $K/C_1$ and $K$.

A parallelogram $D$ is \emph{spanned} in $\eta$ if there exists a strongly connected set $X\subset[D\cap \eta]$ such that the smallest parallelogram containing $X$ is $D$.
\end{defn}
If $\eta,\eta'$ are two configurations, we say that $\eta\le\eta'$ when $\eta_s\le\eta_s'$ for all $s$. For instance, if a parallelogram $D$ is spanned for $\eta$, then is is also spanned for any $\eta'\le\eta$. This order should not be confused with the (inverted) one induced by inclusion when viewing $\eta$ as its set of infections.

Notice that the event that a given parallelogram $D$ is spanned depends only on $\eta_D$ and does not occur when $\eta_D$ contains no infections. We further state two immediate consequences of Definition \ref{def:span} for future reference.
\begin{obs}
\label{obs:obvious}
Let $R\subset\bbR^2$. Then every parallelogram $D$ spanned in $\eta_R$ intersects $R$.
\end{obs}
\begin{obs}\label{obs:connected:component}
Let $\eta$ be a configuration and $X$ be a strongly connected component of $[\eta]$. Then $X = [\eta \cap X]$.
\end{obs}
\begin{proof}By maximality of a strongly connected component, $[\eta\cap X]\subset X$ is at distance at least $C_2'>2r$ from other strongly connected components $X$ of $[\eta]$. Thus, 
\[[\eta]=\bigsqcup_Y [\eta\cap Y],\]where the union is on all strongly connected components of $[\eta]$.\end{proof}

Another standard fact is the following Aizenman-Lebowitz lemma originating from \cite{Aizenman88}, whose proof can be found in Appendix \ref{app:spanning} (Lemma \ref{lem:AL:spanned}).
\begin{lem}
\label{lem:AL:body}
Let $D$ be a spanned parallelogram with diameter $d \geq C_1C_2'$ and let $C_1d\ge k\ge C_1C_2'$. Then there exists a spanned parallelogram with diameter $d'$ such that $k/C_1\le d'\le k$. In particular, if $d\ge K/C_1$, then there exists a spanned critical parallelogram.
\end{lem}

We next import and adapt the notion of crossing from \cite{Bollobas14}*{Definition 8.16}.
\begin{defn}[Crossing]
\label{def:crossing}
We say that a parallelogram $R=R(a,b;0,d)$ is \emph{$u_1$-crossed} if there exists a strongly connected set in $[\bbH_{u_1}\cup (R\cap\eta)]$ intersecting both $\bbH_{u_1}$ and $\bar\bbH_{u_3}(a)$.

Let $C_R^{u_1}$ denote the event that there exists $\eta'\ge \eta$ such that $R$ is $u_1$-crossed for $\eta'$, but there is no spanned critical parallelogram for $\eta'_R$.

We say that a parallelogram $\Lambda=R(0,0;L,H)$ \emph{has no $(\ell,h)$-crossing} (or simply crossing) if the event $C_R^{u_1}$ does not occur for any $R\subset\Lambda$ of the form $R(a,0;a+\ell,H)$ and the event $C_R^{u_2}$ does not occur for any $R\subset\Lambda$ of the form $R(0,b;L,b+h)$.
\end{defn}

In words, a parallelogram is $u_1$-crossed if, given an infected boundary condition on its side opposite to direction $u_1$, the infections inside it are sufficient to infect a ``path'' reaching its side in direction $u_1$. Thus, a $u_1$-crossing corresponds to the propagation of infection across the parallelogram in the direction $u_1$. The event $C_R^{u_1}$ further demands that this crossing is achieved without the help of large spanned parallelograms (in view of Lemma \ref{lem:AL:body}). We should note that, while parallelograms smaller than critical are increasingly unlikely as their size grows, parallelograms larger than critical ones roughly become more likely with their size (hence the name of critical ones), so $C_R^{u_1}$ forces us to use only small unlikely parallelograms all along the crossing. Relying on this fact, it will be possible to establish very strong bounds on the probability of $C_R^{u_1}$. This leads to the notion of a $(\ell,h)$-crossing, which will one of the ``good'' properties we will require the dynamics to satisfy.

Finally, we say that a site $s \in \bbZ^2$ is \emph{locally infectable} in a configuration $\eta$ if $s \in [\eta \cap (s+R(-2K,-2K;2K,2K))]$. We also denote by $\eta^s$ the configuration equal to $\eta$ everywhere except at $s$, i.e.\ $\eta^s_{s}=1-\eta_s$ and $\eta^s_{s'}=\eta_{s'}$ for any $s' \neq s$. We then have the following property, originating from \cite{Hartarsky20FA}*{Section 2}.
\begin{lem}
\label{lem:locally:infectable}
Let $\eta\in\{0,1\}^{\bbZ^2}$, $s\in \bbZ^2$, $U\in\cU$ be such that $s+U\subset\eta$ and let $R=R(-2K,-2K;2K,2K)$. Assume that the origin is not locally infectable in $\eta$, but is locally infectable in $\eta^s$. Then there exists a critical parallelogram $D$ spanned in $\eta^s_{R}$ such that $D\subset R(-3K,-3K;3K,3K)$.
\end{lem}
\begin{proof}
By definition, $0 \in [\eta^s\cap R]\setminus[\eta \cap R]$. Therefore, $s \in R$ and $s+U\not\subset R$. In particular, $d(s,0)> K$. Let $X$ be the strongly connected component of $0$ in $[\eta^s\cap R]$. We claim that $s\in X$. Indeed, by Observation \ref{obs:connected:component} we have $0\in X=[\eta^s\cap R \cap X]$, so if $s\not\in X$, we have $0\in [\eta\cap R]$ which is not the case. 

Let $D_0$ be the smallest parallelogram containing $X$. Since $D_0$ contains $0$ and $s$, its diameter is at least $K$. Moreover, $X=[\eta^s\cap R \cap X]\subset[\eta_R^s \cap D_0]$, so $X$ is a strongly connected set in  $[\eta_R^s \cap D_0]$. Therefore, $D_0$ is spanned in $\eta_R^s$ and has diameter at least $K$. In particular, by Lemma \ref{lem:AL:body} there exists a critical parallelogram $D$ spanned in $\eta^s_{R}$. By Observation \ref{obs:obvious}, $D$ intersects $R$, and by Definition \ref{def:span} $\diam(D)\le K$. Thus, we have $D\subset R(-3K,-3K;3K,3K)$, which ends the proof of the lemma.
\end{proof}

\subsection{The combinatorial bottleneck}
With the notation above we are ready to prove a very general deterministic bottleneck (Lemma \ref{lem:rec} below), constituting the core of our work, which relatively straightforwardly translates into the following bound on $\bbE_{\mu}[\tau_0]$.

The idea behind it is that for the center of a parallelogram of size roughly $3^n$ to become infected, either an $(\ell,h)$-crossing should occur or we should witness $n+1$ spanned critical parallelograms simultaneously. Assuming upper bounds on the probabilities of these two events, we deduce a lower bound on $\bbE_\mu[\tau_0]$.

\begin{prop}
\label{prop:main}
Let $T,L,H,K,\ell,h$ be positive real numbers sufficiently large with respect to $C_2'$. Denote
\[\rho=\max_D\mu(D\text{ is spanned}),\]
where the max is over all critical parallelograms. Also set
\begin{align*}
\pl&{}=\max_{a,b\in\bbR}\mu\left(C_{R(a,b;a+\ell,b+H)}^{u_1}\right)\\
\pd&{}=\max_{a,b\in\bbR}\mu\left(C_{R(a,b;a+L,b+h)}^{u_2}\right).
\end{align*}
Assume that for some integer $n\ge0$ we have the following inequalities on geometry:
\begin{align}
L&{}\ge 3^{n}(11K+\ell)&H&{}\ge 3^n(11K+h),\label{eq:LHbounds}
\end{align}
and probability: 
\begin{align}
1/8&{}\ge \mu(0\text{ is locally infectable})\label{eq:hyp:loc:inf}\\
1&{}\ge T(LH)^2\max(\pd,\pl)\label{eq:hyp:cross}\\
1&{}\ge TLH(LHK^3\rho)^{n+1}\label{eq:hyp:span}.
\end{align}
Then the $\cU$-KCM on $\bbZ^2$ satisfies $\bbE_\mu[\tau_0]\ge T$.
\end{prop}
\begin{rem}
Although the boostrap percolation estimates needed to make use of this statement in higher dimensions are not yet available, let us mention that our argument is not dimension sensitive.
\end{rem}
The proof of Proposition \ref{prop:main} will occupy the rest of the present section. We start by fixing $T,L,H,K,\ell,h$ as in the statement and introducing the following definitions.

Recall that for any $R \subset \bbZ^2$, we identify configurations $\eta\in\{0,1\}^R$ with their zero set $\{x\in R,\eta_x=0\}$. Unless otherwise specified, configurations $\eta\in\{0,1\}^R$ are extended to $\{0,1\}^{\bbZ^2}$ by keeping the same zero set.

\begin{defn}[Good paths and configurations]
\label{def:legal}
For any parallelogram $R \subset \bbR^2$, configuration $\eta \in \{0,1\}^{R \cap \bbZ^2}$ and integer $n\ge0$, we say that $\eta$ is \emph{$n$-good} when the maximum number of critical parallelograms that are disjointly\footnote{As is standard \cite{BK85}, we say that the parallelograms $R_1,\dots,R_k$ are disjointly spanned in $\eta$ if one can find disjoint sets $X_1,\dots,X_k\subset\eta$ such that $\eta'_{X_i}=0$ implies that $R_i$ is spanned in $\eta'$ for all $1\le i\le k$.} spanned in $\eta$ is at most $n$ and $R$ has no crossing for $\eta$.

A \emph{legal path} in $R$ is a sequence $(\eta^{(j)})_{0 \leq j \leq m}$ of configurations in $\{0,1\}^{R \cap \bbZ^2}$ such that for every $j \in \{0,\dots,m-1\}$, there exists $s\in R\cap\bbZ^2$ such that $\eta^{(j+1)}=(\eta^{(j)})^s$ and $(s+U)\cap R\subset \eta^{(j)}$ for some $U \in \mathcal{U}$. For any integer $n\ge 0$, the path is \emph{$n$-good} if for every $j \in\{0,\dots,m\}$, $\eta^{(j)}$ is $n$-good. For any $A,B \subset \{0,1\}^{R \cap \bbZ^2}$, we say $(\eta^{(j)})_{0 \leq j \leq m}$ is a path from $A$ to $B$ when $\eta^{(0)} \in A$ and $\eta^{(m)} \in B$ (if $A$ or $B=\{\eta\}$, we will write $\eta$ to simplify).

We denote by $G(R)$ the set of configurations in $\{0,1\}^{R \cap \bbZ^2}$ that contain no spanned critical parallelogram and such that $R$ contains no crossing, i.e.\ the 0-good configurations. For any $n \in \mathbb{N}$, we define 
\[V(n,R) = \left\{\eta \in \{0,1\}^{R \cap \bbZ^2} \,|\text{ there is an $n$-good legal path from $G(R)$ to $\eta$}\right\}.\]
\end{defn}

Finally, we define our domain sizes for the induction to come:
\begin{align}
L_n &{}= \frac{3^n-1}{2}\left(9K+\ell\right)+3^nK&H_n &{}= \frac{3^n-1}{2}\left(9K+h\right)+3^nK,\label{eq:def:Ln}\end{align}
so that $L_n-L_{n-1}=2L_{n-1}+9K+\ell$ and $H_{n}-H_{n-1}=2H_{n-1}+9K+h$.

\begin{lem}\label{lem:rec}
For any non-negative integer $n$, for any parallelogram $R=R(a,b;c,d)$ such that $c-a \geq 2L_n$ and $d-b \geq 2H_n$, we have that for all $\eta \in V(n,R)$, there is no spanned critical parallelogram in $\eta$ intersecting $R(a+L_n,b+H_n;c-L_n,d-H_n)$.
\end{lem}
We first deduce Proposition \ref{prop:main} from Lemma \ref{lem:rec}.

\begin{proof}[Proof of Proposition \ref{prop:main}, assuming Lemma \ref{lem:rec}]

Clearly, it suffices to prove that $\bbP_{\mu}(\tau_0> 2T)\ge 1/2$. Let $\tau'=\inf\{t\ge 0,0\text{ is locally infectable in $\eta(t)$}\}$. Clearly, $\tau'\le \tau_0$. We denote $R=R(-L/2,-H/2;L/2,H/2)$ and define the following events.
\begin{itemize}
    \item[$E_1$:] $\tau'>0$, i.e.\ $0$ is not locally infectable in $\eta(0)$.
    \item[$E_2$:] There is no critical parallelogram spanned in $(\eta(0))_R$.
    \item[$E_3$:] For all $0\le t\le 2T$, no $n+1$ critical parallelograms are disjointly spanned in $(\eta(t))_R$.
    \item[$E_4$:] For all $0\le t\le 2T$, $R$ has no crossing for $(\eta(t))_R$.
\end{itemize}

We claim that $\bigcap_{i=1}^4E_i\subset \{\tau'>2T\}$. Indeed, let us assume $\bigcap_{i=1}^4E_i$ occurs. By Definition \ref{def:legal}, $E_2$ and $E_4$ imply $(\eta(0))_R\in G(R)$. Moreover, $E_3$ and $E_4$ give that $(\eta(t))_R$ is $n$-good for all $t\le 2T$. Thus, for all $t\le 2T$, the sequence of configurations $(\eta(t'))_R$ for $t'\in[0,t]$ yields a $n$-good legal path from $(\eta(0))_R\in G(R)$ to $(\eta(t))_R$, therefore $(\eta(t))_R\in V(n,R)$. Applying Lemma \ref{lem:rec}, this yields that for all $t\in[0,2T]$ there is no critical parallelogram spanned in $(\eta(t))_R$ intersecting \begin{multline*}R(-L/2+L_n,-H/2+H_n;L/2-L_n,H/2-H_n)\\\supset R\left(-\frac{9K+\ell}{2},-\frac{9K+h}{2};\frac{9K+\ell}{2},\frac{9K+h}{2}\right)\supset R(-3K,-3K;3K,3K),\end{multline*}
recalling \eqref{eq:LHbounds} and \eqref{eq:def:Ln}. Finally, notice that 0 being locally infectable depends only on the configuration in $R(-2K,-2K;2K,2K)$. Thus, if $E_1$ occurs, we can apply Lemma \ref{lem:locally:infectable} with $\eta^s:=(\eta(\tau'))_R$ and $\eta:=\lim_{t\to\tau',t<\tau'}(\eta(t))_R$, to deduce that there is a critical parallelogram spanned in $(\eta(\tau'))_R$ and contained in $R(-3K,-3K;3K,3K)$. Hence, $\tau'>2T$, as claimed.

We next brutally bound the probability of $E_1,\dots,E_4$. By \eqref{eq:hyp:loc:inf}, $1-\bbP_{\mu}(E_1)\le 1/8$. By the union bound on all (discrete) critical parallelograms intersecting $R$ (recall Observation \ref{obs:obvious}) and \eqref{eq:hyp:span} we have \[1-\bbP_{\mu}(E_2)\le O(LHK^2)\rho\le \frac{LHK^3\rho}{8}\le\frac{1}{8(TLH)^{1/(n+1)}}\le \frac 1 8.\]

In order to treat $E_3$ and $E_4$, recall from Section \ref{subsec:KCM} that the $\cU$-KCM may be constructed by associating to each site $x\in\bbZ^2$ a standard Possion process and attempting to update $\eta_x$ at the times given by its Poisson process. We will refer to these times as \emph{clock rings}. Let $N$ denote the number of clock rings in $R$ between $0$ and $2T$. Since $N$ is a Poisson random variable with parameter $2T|R|=\Theta(TLH)>1$, we have $\bbP_{\mu}(N\ge C_1TLH)\le 1/16$ (e.g.\ by the Bienaym\'e--Chebyshev inequality).

Let $\eta^{(j)}$ denote the restriction of the configuration to $R \cap\bbZ^2$ after the $j$-th clock ring in $R$. We next claim that $\eta^{(j)}$ is at equilibrium, the formal proof being postponed for the moment.
\begin{claim}
\label{obs:stationarity}
With the above notation, $\eta^{(j)}$ has the product Bernoulli distribution with parameter $1-q$ for all $j$.
\end{claim}
For any $\eta \in \{0,1\}^{R \cap \bbZ^2}$ we write $\cD_{n+1}(\eta) = \{$there are $n+1$ critical parallelograms disjointly spanned in $\eta\}$, so that $E_3$ does not occur iff $\bigcup_{j=0}^N\cD_{n+1}(\eta^{(j)})$ does not. Then the union bound and Claim \ref{obs:stationarity} give
\begin{align}1-\bbP_{\mu}(E_3)\le{}& \bbP_{\mu}(N\ge C_1TLH)+\sum_{j=0}^{C_1TLH}\bbP_{\mu}\left(\cD_{n+1}\left(\eta^{(j)}\right)\right)\nonumber\\
{}\le{}& \frac{1}{16}+2C_1TLH\mu(\cD_{n+1}(\eta_R)).\label{eq:E3:bound}
\end{align}
In order to bound $\mu(\cD_{n+1}(\eta_R))$, we use the union bound on all $(O(LHK^2))^{n+1}$ possible choices of $n+1$ critical parallelograms intersecting $R$ (by Observation \ref{obs:obvious}, a parallelogram spanned in $\eta_R$ intersects $R$) together with the BK inequality \cite{BK85} to get that \eqref{eq:E3:bound} is at most
\[\frac{1}{16}+2C_1TLH(O(LHK^2)\rho)^{n+1}\le \frac{1}{8},
\]
using \eqref{eq:hyp:span} in the last inequality.

Similarly, using \eqref{eq:hyp:cross}, we have
\[1-\bbP_{\mu}(E_4)\le \frac{1}{16}+2C_1TLH\cdot O(L+H)\max(\pd,\pl) \le \frac{1}{8}.\]
Putting the bounds $1-\bbP_\mu(E_i)\le 1/8$ for all $i$ together and recalling that $\bigcap_{i=1}^4E_i\subset \{\tau'>2T\}$, we conclude that
\[\bbP_{\mu}(\tau_0> 2T)\ge \bbP_\mu(\tau'>2T)\ge \bbP_\mu\left(\bigcap_{i=1}^4E_i\right)\ge 1- \sum_{i=1}^4(1-\bbP_{\mu}(E_i))\ge \frac{1}{2}.\qedhere\]
\end{proof}
\begin{proof}[Proof of Claim \ref{obs:stationarity}]
We fix $j$, and denote by $t$ the time of the $j$-th clock ring in $R$. For all $s\in[0,t]$, we will construct a (random) set of sites $X_s$ such that the configuration in $R$ at time $t$ can be reconstructed from the configuration in $X_s$ at time $s$ and the updates since time $s$. The construction is as follows. For all $s\in[0,t]$, $X_s$ contains all sites in $R$. Moreover, for any $x_0\in R$ which had a clock ring before time $t$, for any $s$ before this clock ring we add to $X_s$ the sites $x_1$ such that $x_1-x_0\in\bbU$, where $\bbU=\bigcup_{U\in\mathcal{U}}U$. Now, for any of these $x_1$ that had a clock ring before the clock ring at $x_0$, for any $s$ before this clock ring at $x_1$ we add the sites $x_2$ such that $x_2-x_1\in\bbU$, etc.
It is classical to see that the $X_s$ are a.s.\ finite. For example, one can see there are exactly $j$ clock rings in $R$ before time $t$, so clock rings in $R$ may add at most $j|\bbU|$ sites to the $X_s$. Now, for any of these sites that is not in $R$, the number of sites it brings can be bounded from above by a continuous time branching process with reproduction law $\delta_{|\bbU|+1}$, so stays finite.

We now consider all the clock rings such that there exists $s\in[0,t]$ such that the clock ring occurs in $X_s$ at time at most $s$, and order them chronologically. If $s$ is the time of the $j'$-th such clock ring, $\bar\eta^{(j')}$ will denote the configuration in $X_s$ at time $s$. Let us denote by $\mathcal{F}$ the sigma-algebra generated by all clock rings in $\bbZ^2$. In particular, for any $s\in[0,t]$, $X_s$ is $\cF$-measurable. By induction on $j'$, one can prove that conditionally to $\cF$, $\bar\eta^{(j')}$ has law $\mu$. Now, the $j'$ corresponding to $j$ depends only on $\cF$, hence $\eta^{(j)}$ has law $\mu$, which proves the claim.
\end{proof}

\begin{proof}[Proof of Lemma \ref{lem:rec}.]
We will prove the lemma by induction on $n$. For any $n$ let us call $\mathcal{H}_n$ the statement of the lemma for $n$. $\mathcal{H}_0$ holds by definition. Let $n \geq 1$ and assume that $\cH_{n-1}$ holds. Let $R=R(a,b;c,d)$ be such that $c-a \geq 2L_n$ and $d-b \geq 2H_n$. We define a smaller parallelogram $R'$ in the middle of $R$:
\[
R'=R(a+L_n-L_{n-1},b+H_n-H_{n-1};c-(L_n-L_{n-1}),d-(H_n-H_{n-1}))
\]
(see Figure \ref{fig:main}). We will prove Lemma \ref{lem:rec} by showing $\cH_{n}$, using the following result, whose proof we postpone for the moment.
\begin{lem}\label{lem:rec:main}
For all $\eta \in V(n,R) \setminus G(R)$ (recall Definition \ref{def:legal}), there exists a critical parallelogram not intersecting $R'$ that is spanned in $\eta$.
\end{lem}
Let $\eta \in V(n,R)$. The idea is that since $\eta$ is $n$-good, there are at most $n$ spanned critical parallelograms in $\eta$, and Lemma \ref{lem:rec:main} implies one of them has to be outside $R'$. Thus, there can only be $n-1$ such parallelograms in $R'$. This will allow us to use $\mathcal{H}_{n-1}$ in $R'$, which will prevent the existence of a spanned critical parallelogram intersecting $R(a+L_n,b+H_n;c-L_n,d-H_n)$. 

We now give the rigorous argument. Since $\eta \in V(n,R)$, there exists an $n$-good legal path from $G(R)$ to $\eta$, denoted by $(\eta^{(j)})_{0 \leq j \leq m}$. For any $j\in\{0,\dots,m\}$, we have $\eta^{(j)}\in V(n,R)$. Let us prove that there are at most $n-1$ critical parallelograms disjointly spanned in $\eta^{(j)}_{R'}$. If $\eta^{(j)}\in G(R)$, it is clear. If $\eta^{(j)}\not\in G(R)$, Lemma \ref{lem:rec:main} guarantees the existence of a critical parallelogram $D$ spanned in $\eta^{(j)}$ that does not intersect $R'$. Now, let $k$ be the maximal number of critical parallelograms disjointly spanned in $\eta^{(j)}_{R'}$, and $D_1,\dots,D_k$ be such parallelograms. Since $D$ does not intersect $R'$, we have that $D,D_1,\dots,D_k$ are critical parallelograms disjointly spanned in $\eta^{(j)}$. Since $\eta^{(j)}$ is $n$-good, it contains at most $n$ disjointly spanned critical parallelograms, so $k\leq n-1$. Furthermore, if $R'$ had a crossing for $\eta^{(j)}_{R'}$, $R$ would have a crossing for $\eta^{(j)}$, which is not the case, as $\eta^{(j)}$ is $n$-good. Hence, $R'$ has no crossing for $\eta^{(j)}_{R'}$. We deduce that $\eta^{(j)}_{R'}$ is $(n-1)$-good in the parallelogram $R'$. Thus, if we consider the path $(\eta^{(j)}_{R'})_{0 \leq j \leq m}$ and keep only the steps $\eta^{(j)}_{R'}$ that differ from $\eta^{(j-1)}_{R'}$, we obtain a $(n-1)$-good legal path, going from $G(R')$ to $\eta_{R'}$. This implies $\eta_{R'}\in V(n-1,R')$. Therefore, we can apply $\mathcal{H}_{n-1}$ to $\eta_{R'}$, which yields that there is no spanned critical parallelogram in $\eta_{R'}$ intersecting $R(a+L_n,b+H_n;c-L_n,d-H_n)$. This implies $\mathcal{H}_n$ and concludes the proof of Lemma \ref{lem:rec}.
\end{proof}

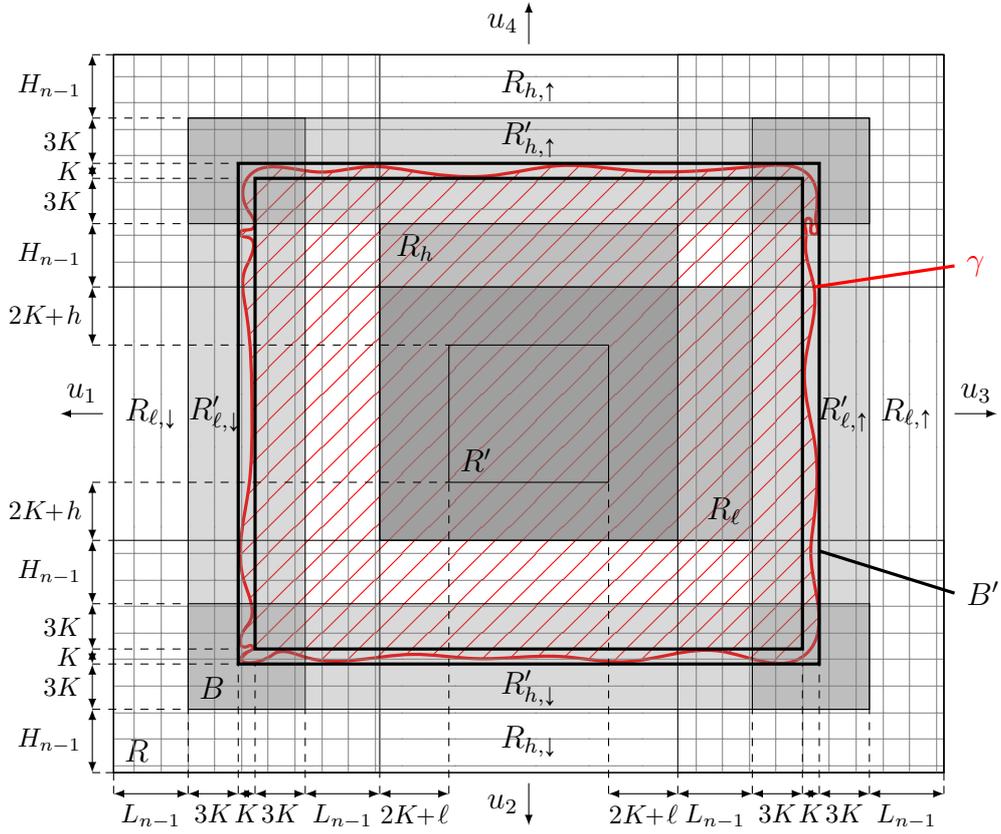
\begin{figure}
    \centering
    \begin{tikzpicture}[scale=0.14]
    \draw (-39,-34)--(39,-34)--(39,34)--(-39,34)--cycle ;
    \draw (-39,-34) node [above right] {$R$} ;
    \draw[fill opacity=1,pattern=my north east lines, pattern color=red, draw=red,very thick] (-26,-22) to [closed, curve through={(-26.8,-22.8)..(-24,-22.8)..(-21.2,-23)..(-11,-22.8)..(-5,-23.2)..(3,-23)..(9,-23.5)..(18,-22.5)..(26,-22.8)..(27,-17)..(26.5,-13)..(26.8,-10)..(26.9,-3)..(25.9,4)..(26.8,12)..(26,16)..(26,17.5)..(26.3,18.5)..(26.9,17)..(27,18)..(27,19)..(26.3,22.6)..(22.4,23.5)..(16,23)..(10,23)..(3,23.5)..(-4,22.5)..(-11,23)..(-15,23.4)..(-20,22.9)..(-26.5,22.2)..(-26.2,19)..(-26,17.5)..(-27.1,17)..(-26.4,16.7)..(-26.5,14)..(-26.6,10)..(-26,0)..(-26.3,-7)..(-26.7,-16)..(-26.4,-20)}] (-26.6,-22.1);
    \draw[black,fill=gray,fill opacity=0.5] (-14,-12)--(14,-12)--(14,18)--(-14,18)--cycle;
    \draw (-13.5,18) node[below right] {$R_h$};
    \draw[black,fill=gray,fill opacity=0.5] (-14,-12)--(21,-12)--(21,12)--(-14,12)--cycle;
    \draw (21,-11.5) node [above left] {$R_\ell$};
    \draw (7.5,6.5)--(7.5,-6.5)--(-7.5,-6.5)--(-7.5,6.5)--cycle ;
    \draw (-7.5,-6.5) node [above right] {$R'$} ;
    \draw[pattern=my vertical lines, pattern color=gray] (-39,-34)--(-14,-34)--(-14,34)--(-39,34)--cycle;
    \draw (-35.5,0) node {$R_{\ell,\downarrow}$};
    \draw[pattern=my horizontal lines, pattern color=gray] (-39,-34)--(39,-34)--(39,-12)--(-39,-12)--cycle;
    \draw (0,-31) node {$R_{h,\downarrow}$};
    \draw[pattern=my vertical lines, pattern color=gray] (14,-34)--(39,-34)--(39,34)--(14,34)--cycle;
    \draw (35.5,0) node {$R_{\ell,\uparrow}$};
    \draw[pattern=my horizontal lines, pattern color=gray] (-39,12)--(39,12)--(39,34)--(-39,34)--cycle;
    \draw (0,31) node {$R_{h,\uparrow}$};
    \draw[black,fill=gray,fill opacity=0.3](-32,-28)--(-21,-28)--(-21,28)--(-32,28)--cycle;
    \draw[black,fill=gray,fill opacity=0.3] (-32,-28)--(32,-28)--(32,-18)--(-32,-18)--cycle;
    \draw[black,fill=gray,fill opacity=0.3] (32,-28)--(21,-28)--(21,28)--(32,28)--cycle;
    \draw[black,fill=gray,fill opacity=0.3] (-32,28)--(32,28)--(32,18)--(-32,18)--cycle;
    \draw[very thick] (-27.29,-23.71)--(27.29,-23.71)--(27.29,23.71)--(-27.29,23.71)--cycle ;
    \draw[very thick] (-25.71,-22.29)--(25.71,-22.29)--(25.71,22.29)--(-25.71,22.29)--cycle ;
    \draw[red,very thick] (26.8,12)--(40,14) node [right, red] {$\gamma$};
    \draw (-29.5,0) node{$R_{\ell,\downarrow}'$};
    \draw (0,-26) node {$R_{h,\downarrow}'$};
    \draw (29.5,0) node{$R_{\ell,\uparrow}'$}; 
    \draw (0,26) node {$R_{h,\uparrow}'$};
    \draw[very thick] (27.29,-13)--(40,-17) node [right] {$B'$};
    \draw (-32,-28) node [above right] {$B$} ; 
    \draw[->,>=latex] (-40,0)--(-44,0) node [midway,above] {$u_1$} ;
    \draw[->,>=latex] (0,-35)--(0,-39) node [midway,left] {$u_2$}; 
    \draw[->,>=latex] (40,0)--(44,0) node [midway,above] {$u_3$} ;
    \draw[->,>=latex] (0,35)--(0,39) node [midway,left] {$u_4$}; 
    \draw[<->,>=latex] (-39,-36)--(-32,-36) node [midway,below]{\footnotesize{$L_{n-1}$}};
    \draw[densely dashed] (-39,-36)--(-39,-34);
    \draw[dashed] (-32,-36)--(-32,-28) ;
    \draw[<->,>=latex] (-32,-36)--(-27.29,-36) node [midway,below]{\footnotesize{$3K$}};
    \draw[dashed] (-25.71,-36)--(-25.71,-23.71);
    \draw[<->,>=latex] (-25.71,-36)--(-27.29,-36) node [midway,below]{\footnotesize{$K$}};
    \draw[dashed] (-27.29,-36)--(-27.29,-22.29);
    \draw[<->,>=latex] (-25.71,-36)--(-21,-36) node [midway,below]{\footnotesize{$3K$}};
    \draw[dashed] (-21,-36)--(-21,-28);
    \draw[<->,>=latex] (-21,-36)--(-14,-36) node [midway,below]{\footnotesize{$L_{n-1}$}};
    \draw[dashed] (-14,-36)--(-14,-34);
    \draw[<->,>=latex] (-14,-36)--(-7.5,-36) node [midway,below]{\footnotesize{$2K\!\!+\!\ell$}};
    \draw[dashed] (-7.5,-36)--(-7.5,-6.5);
    \draw[<->,>=latex] (39,-36)--(32,-36) node [midway,below]{\footnotesize{$L_{n-1}$}};
    \draw[densely dashed] (39,-36)--(39,-34);
    \draw[dashed] (32,-36)--(32,-28) ;
    \draw[<->,>=latex] (32,-36)--(27.29,-36) node [midway,below]{\footnotesize{$3K$}};
    \draw[dashed] (25.71,-36)--(25.71,-23.71);
    \draw[<->,>=latex] (25.71,-36)--(27.29,-36) node [midway,below]{\footnotesize{$K$}};
    \draw[dashed] (27.29,-36)--(27.29,-22.29);
    \draw[<->,>=latex] (25.71,-36)--(21,-36) node [midway,below]{\footnotesize{$3K$}};
    \draw[dashed] (21,-36)--(21,-28);
    \draw[<->,>=latex] (21,-36)--(14,-36) node [midway,below]{\footnotesize{$L_{n-1}$}};
    \draw[dashed] (14,-36)--(14,-34);
    \draw[<->,>=latex] (14,-36)--(7.5,-36) node [midway,below]{\footnotesize{$2K\!\!+\!\ell$}};
    \draw[dashed] (7.5,-36)--(7.5,-6.5);
    \draw[<->,>=latex] (-41,-34)--(-41,-28) node [midway,left] {\footnotesize{$H_{n-1}$}};
    \draw[densely dashed] (-41,-34)--(-39,-34);
    \draw[dashed] (-41,-28)--(-32,-28);
    \draw[<->,>=latex] (-41,-28)--(-41,-23.71) node [midway,left] {\footnotesize{$3K$}};
    \draw[dashed] (-41,-23.71)--(-25.71,-23.71);
    \draw[<->,>=latex] (-41,-23.71)--(-41,-22.29) node [midway,left] {\footnotesize{$K$}};
    \draw[dashed] (-41,-22.29)--(-27.29,-22.29);
    \draw[<->,>=latex] (-41,-22.29)--(-41,-18) node [midway,left] {\footnotesize{$3K$}};
    \draw[dashed] (-41,-18)--(-32,-18);
    \draw[<->,>=latex] (-41,-18)--(-41,-12) node [midway,left] {\footnotesize{$H_{n-1}$}};
    \draw[dashed] (-41,-12)--(-39,-12);
    \draw[<->,>=latex] (-41,-12)--(-41,-6.5) node [midway,left] {\footnotesize{$2K\!\!+\!h$}};
    \draw[dashed] (-41,-6.5)--(-7.5,-6.5);
    \draw[<->,>=latex] (-41,34)--(-41,28) node [midway,left] {\footnotesize{$H_{n-1}$}};
    \draw[densely dashed] (-41,34)--(-39,34);
    \draw[dashed] (-41,28)--(-32,28);
    \draw[<->,>=latex] (-41,28)--(-41,23.71) node [midway,left] {\footnotesize{$3K$}};
    \draw[dashed] (-41,23.71)--(-25.71,23.71);
    \draw[<->,>=latex] (-41,23.71)--(-41,22.29) node [midway,left] {\footnotesize{$K$}};
    \draw[dashed] (-41,22.29)--(-27.29,22.29);
    \draw[<->,>=latex] (-41,22.29)--(-41,18) node [midway,left] {\footnotesize{$3K$}};
    \draw[dashed] (-41,18)--(-32,18);
    \draw[<->,>=latex] (-41,18)--(-41,12) node [midway,left] {\footnotesize{$H_{n-1}$}};
    \draw[dashed] (-41,12)--(-39,12);
    \draw[<->,>=latex] (-41,12)--(-41,6.5) node [midway,left] {\footnotesize{$2K\!\!+\!h$}};
    \draw[dashed] (-41,6.5)--(-7.5,6.5);
    \end{tikzpicture}  
    \caption{The setting of the proof of Lemma \ref{lem:rec:main}. For the figure we assume that $u_3=0$ and $u_4=\pi/2$. $B'$ is the frame with thickened boundary, $R_\ell$ and $R_h$ are the overlapping regions in dark gray. The regions $R_{\ell,\downarrow}'$, $R_{h,\downarrow}'$, $R_{\ell,\uparrow}'$ and $R_{h,\uparrow}'$ are in lighter gray and the frame formed by their union is $B$. The horizontally (resp. vertically) hatched regions are $R_{h,\downarrow}$ and $R_{h,\uparrow}$ (resp. $R_{\ell,\downarrow}$ and $R_{\ell,\uparrow}$). The contour inside $B'$ is $\gamma$ and its diagonally hatched interior is $\ring\gamma$. All the regions drawn are closed subsets of $\bbR^2$ with the exception of $R_\ell$, $R_h$ and $\ring\gamma$, which are open. The thicker version, $\bar\gamma$, of $\gamma$ and the set $F\subset\bbZ^2$ defined in \eqref{eq:def:F} are not drawn.}
    \label{fig:main}
\end{figure}

Consequently, it remains to prove Lemma \ref{lem:rec:main}.
\begin{proof}[Proof of Lemma \ref{lem:rec:main}]
Another way to state Lemma \ref{lem:rec:main} is to say that if we start a $n$-good legal path from $G(R)$, we cannot reach a configuration in which all spanned critical parallelograms intersect $R'$. As legal paths are reversible, we will prove that if we start an $n$-good legal path from a configuration in which all spanned critical parallelograms intersect $R'$, it is impossible to reach $G(R)$, because there will always be a spanned critical parallelogram ``near'' $R'$. 

In order to do that, we start by introducing the following geometric regions, represented in Figure~\ref{fig:main}.
\begin{align*}
R_\ell={}&\ring R(a+2L_{n-1}+7K,b+2H_{n-1}+7K;c-L_{n-1}-7K,d-2H_{n-1}-7K)\\
R_{\ell,\downarrow}={}&R(a,b;a+2L_{n-1}+7K,d)\\
R_{\ell,\uparrow}={}&R(c-2L_{n-1}-7K,b;c,d)\\
R'_{\ell,\downarrow}={}&R(a+L_{n-1},b+H_{n-1},a+L_{n-1}+7K,d-H_{n-1})\\
R'_{\ell,\uparrow}={}&R(c-L_{n-1}-7K,b+H_{n-1},c-L_{n-1},d-H_{n-1})
\end{align*}
$R_\ell$ will contain the aforementioned spanned critical parallelogram ``near'' $R'$. In the parallelograms $R_{\ell,\downarrow}$ and $R_{\ell,\uparrow}$, we will use $\mathcal{H}_{n-1}$, and $R'_{\ell,\downarrow}$, $R'_{\ell,\uparrow}$ are the respective ``central'' parallelograms inside that will not be intersected by spanned critical parallelograms. We also define similar regions with index $h$ instead of $\ell$.

We further define the two frames (see Figure \ref{fig:main})
\begin{align*}B={}&R_{\ell,\downarrow}' \cup R_{\ell,\uparrow}' \cup R_{h,\downarrow}' \cup R_{h,\uparrow}'\\
B'={}&\begin{multlined}[t]
R(a+L_{n-1}+3K,b+H_{n-1}+3K;c-L_{n-1}-3K,d-H_{n-1}-3K)\setminus \\
\ring R(a+L_{n-1}+4K,b+H_{n-1}+4K;c-L_{n-1}-4K,d-H_{n-1}-4K).
\end{multlined}
\end{align*}
As the union of $R'_{\ell,\downarrow}$, $R'_{\ell,\uparrow}$, $R'_{h,\downarrow}$, $R'_{h,\uparrow}$, the frame $B$ will not be intersected by spanned critical parallelograms. $B$ will be a ``buffer'' between the inner and outer parts of $R$, and $B'$ its central part. We will be able to find a contour $\gamma$ contained in $B'$ such that the dynamics in the interior of the contour is ``isolated'' from the dynamics outside in a specific way. This is the goal of the following claim.

\begin{claim}
\label{claim:gamma}
Let $\eta\in\{0,1\}^{R \cap \bbZ^2}$ be such that every critical parallelogram spanned in $\eta$ intersects $R'$. Then there exists a closed contour $\gamma\subset \bbR^2$ (that is, a self-avoiding and closed path obtained by connecting sites of $\bbZ^2$ by straight lines linking a site to its left, top, right or bottom neighbour) satisfying the following properties:
\begin{itemize}
    \item $\gamma\subset B'$.
    \item $d(\gamma,R\setminus B')\ge C_1^2$.
    \item Every $s\in\bar\gamma$ is not locally infectable in $\eta$, where
    \[\bar\gamma=\{s\in\bbZ^2|d(s,\gamma)\le C_1\}.\]
    \item The (topologically open) interior, $\ring\gamma\subset \bbR^2$, defined by $\gamma$ contains $R'$.
\end{itemize}
\end{claim}
\begin{proof}
We proceed by renormalization. Let $H$ denote the regular hexagon centered at the origin with diameter $C_1^3$ and having two horizontal sides. Consider the tiling of the plane with translates of $H$ and denote by $\bbT$ the triangular lattice formed by their centers. Let $T=\{t\in\bbT|H+t\subset B'\}$ be the sites of $\bbT$ corresponding to $B'$. We say that a site $t\in T$ is \emph{open} if no site in $(t+H)\cap\bbZ^2$ is locally infectable in $\eta$.

If there exists a contour of open sites in $T$ surrounding $R'$ (where a contour in $\bbT$ is a self-avoiding and closed path in the graph $(\bbT,\{(t,t')\in\bbT^2|t+H$ and $t'+H$ share a side$\}$), we may choose $\gamma$ approximating this contour, which clearly satisfies the conditions of the claim. Assume for a contradiction that such a contour does not exist. In this case, there is a path of closed sites in $T$ from the inner to the outer boundary of $T$. In particular, this path yields a strongly connected (recall Definition \ref{def:span}) set $X$ of sites of $\bbZ^2$ that are locally infectable in $\eta$, with diameter at least $K-4C_1^3$, contained in either the ``left part'' of the frame $B'$, defined as \[R''_{\ell,\downarrow} = R(a+L_{n-1}+3K,b+H_{n-1}+3K;a+L_{n-1}+4K,d-H_{n-1}-3K),\]
(see Figure \ref{fig:main}) or in the top, right or bottom part of $B'$, defined similarly. Without loss of generality, assume that $X$ is contained in $R''_{\ell,\downarrow}$. Since the sites of $X$ are locally infectable in $\eta$, they are infectable in $\eta_{R'''_{\ell,\downarrow}}$, where $R'''_{\ell,\downarrow}$ is ``$R''_{\ell,\downarrow}$ enlarged by $2K$ on each side,'' that is \[R'''_{\ell,\downarrow} = R(a+L_{n-1}+K,b+H_{n-1}+K;a+L_{n-1}+6K,d-H_{n-1}-K).\]
$X$ is then a strongly connected set contained in $[\eta_{R'''_{\ell,\downarrow}}]$. 

We denote by $X'$ the strongly connected component of $[\eta_{R'''_{\ell,\downarrow}}]$ containing $X$, and we consider the smallest parallelogram $D$ containing $X'$. We claim that $D$ is spanned in $\eta_{R'''_{\ell,\downarrow}}$. Indeed, Observation \ref{obs:connected:component} yields $X'=[\eta_{R'''_{\ell,\downarrow}} \cap X'] \subset [\eta_{R'''_{\ell,\downarrow}} \cap D]$, so $X'$ is a strongly connected set in $[\eta_{R'''_{\ell,\downarrow}} \cap D]$ such that the smallest parallelogram containing $X'$ is $D$, which means that $D$ is spanned in $\eta_{R'''_{\ell,\downarrow}}$. Furthermore, $D$ contains $X$, which has diameter at least $K-4C_1^3$. Therefore, by Lemma \ref{lem:AL:body} there exists a critical parallelogram $D'$ spanned in $\eta_{R'''_{\ell,\downarrow}}$. $D'$ is then spanned in $\eta$. Moreover, by Observation \ref{obs:obvious}, $D'$ intersects $R'''_{\ell,\downarrow}$, and, since $D'$ is critical, it has diameter at most $K$. Hence, $D'$ is contained in $R'_{\ell,\downarrow}$. We deduce the existence of a critical parallelogram spanned in $\eta$ not intersecting $R'$, hence a contradiction.
\end{proof}

Now that most of the needed geometric regions are defined, we may start the proof of Lemma \ref{lem:rec:main} itself. We fix $\eta \in \{0,1\}^{R \cap \bbZ^2} \setminus G(R)$ such that every critical parallelogram spanned in $\eta$ intersects $R'$. We will prove that there is no $n$-good legal path from $\eta$ to $G(R)$. Since legal paths can be reversed, this implies $\eta \not\in V(n,R)$, which proves Lemma \ref{lem:rec:main}.

We fix a contour $\gamma$ as provided by Claim \ref{claim:gamma} for the configuration $\eta$ (see Figure \ref{fig:main}), as well as its thickened version $\bar\gamma$ and its interior $\ring\gamma$. In order to isolate the dynamics in $\ring\gamma$ from the dynamics outside, we need to define 
\begin{equation}
\label{eq:def:F}
F=\{s\in B'\cap\bbZ^2|s\text{ is not locally infectable in }\eta\}.
\end{equation}
We then have $\bar\gamma\subset F$. $F$ will ``shield $\ring\gamma$ from outside interference.''

Let $(\eta^{(j)})_{0 \leq j \leq m}$ be an $n$-good legal path with $\eta^{(0)}=\eta$. We will use an induction on $j \in \{0,\dots,m\}$ to prove that $\eta^{(m)} \not\in G(R)$. More precisely, we will prove by induction on $j$ that the following properties hold for $j \in \{0,\dots,m\}$.
\begin{itemize}
    \item[$\cP_j^1$] For every $\zeta \in \{\ell,h\}$, there exists a critical parallelogram contained in $R_\zeta$ spanned in $\eta^{(j)}$.
    \item[$\cP_j^2$] The sites of $F$ are not locally infectable in $\eta^{(j)}$.
    \item[$\cP_j^3$] For every $(\zeta,\xi) \in \{\ell,h\}\times\{\downarrow,\uparrow\}$, $\eta^{(j)}_{R_{\zeta,\xi}} \in V(n-1,R_{\zeta,\xi})$.
    \item[$\cP_j^4$] $[\eta^{(j)}_{\ring\gamma}] = [\eta^{(0)}_{\ring\gamma}]$.
\end{itemize}
$\cP_j^4$ is what we mean by ``the dynamics inside $\ring\gamma$ is isolated from the outside'': from time 0 to time $j$, the sites infectable by the configuration in $\ring\gamma$ have not changed. $\cP_j^2$ is necessary for $F$ to play its role as a shield throughout the path. $\cP_j^3$, along with $\mathcal{H}_{n-1}$, will ensure that there are no spanned critical parallelograms at the center of the $R_{\zeta,\xi}$, which will help to preserve the shield. Finally, if $\cP^1_m$ is satisfied, then there exists a critical parallelogram spanned in $\eta^{(m)}$, so $\eta^{(m)} \not \in G(R)$, which proves Lemma~\ref{lem:rec:main}, so it suffices to establish the induction.

We begin with a quick claim.
\begin{claim}\label{claim:rec:easy}
There exists a critical parallelogram spanned in $\eta$ intersecting $R'$.
\end{claim}
\begin{proof}
Since $\eta$ belongs to an $n$-good legal path, $R$ has no crossing for $\eta$, but by the definition of $\eta$, $\eta\not\in G(R)$, hence there exists a critical parallelogram spanned in $\eta$. Moreover, by assumption, every critical parallelogram spanned in $\eta$ intersects $R'$, hence the claim.
\end{proof}
With this, we can start work on the induction.

\textbf{Base: j=0.}
$\cP_0^4$ is trivial, and the definition \eqref{eq:def:F} of $F$ implies $\cP_0^2$. For $\cP_0^1$, Claim \ref{claim:rec:easy} gives the existence of a critical parallelogram spanned in $\eta$ intersecting $R'$, hence contained in $R_\ell\cap R_h$, which yields $\cP_0^1$. 

We now show $\cP_0^3$. Let $(\zeta,\xi) \in \{\ell,h\}\times\{\downarrow,\uparrow\}$. We claim that there is no critical parallelogram spanned in $\eta_{R_{\zeta,\xi}}$. Indeed, by Observation \ref{obs:obvious} such a parallelogram would have to intersect $R_{\zeta,\xi}$ so could not intersect $R'$, which would contradict the definition of $\eta$. Moreover, $R$ has no crossing for $\eta$ since $\eta$ is $n$-good, so $R_{\zeta,\xi}$ has no crossing for $\eta_{R_{\zeta,\xi}}$. Thus, $\eta_{R_{\zeta,\xi}} \in G(R_{\zeta,\xi})$, hence $\eta_{R_{\zeta,\xi}} \in V(n-1,R_{\zeta,\xi})$, so $\cP_0^3$ holds.

\textbf{Induction step.} Let $j \in \{0,\dots,m-1\}$, and suppose that $\cP^1_{j}$, $\cP^2_{j}$, $\cP^3_{j}$ and $\cP_j^4$ hold. Since $(\eta^{(k)})_{0 \leq k \leq m}$ is a legal path, we have $\eta^{(j+1)}=(\eta^{(j)})^s$ and $(s+U)\cap R\subset \eta^{(j)}$ for some $s\in R \cap \bbZ^2$ and $U\in\cU$.

We first prove $\cP^4_{j+1}$, that is $[\eta^{(j+1)}_{\ring\gamma}] = [\eta^{(0)}_{\ring\gamma}]$, using that ``$F$ shields $\ring\gamma$ from the influence of the exterior.''
\begin{claim} \label{claim:P4}
$\cP^4_{j+1}$ holds.
\end{claim}
\begin{proof}
If $s \not\in \ring\gamma$, then $\eta^{(j+1)}_{\ring \gamma} = \eta^{(j)}_{\ring \gamma}$, so $[\eta^{(j+1)}_{\ring \gamma}]=[\eta^{(j)}_{\ring \gamma}] = [\eta^{(0)}_{\ring \gamma}]$ by $\cP_j^4$. Moreover, if $s \in \ring \gamma$, then $s+U \subset \ring\gamma \cup \bar \gamma$. Furthermore, the sites of $\bar \gamma$ are in $F$, so by $\cP_{j}^2$ they are not locally infectable in $\eta^{(j)}$ and, in particular $s'\not\in \eta^{(j)}$ for all $s'\in\bar\gamma$. Since $(s+U)\cap R\subset \eta^{(j)}$, this implies $s+U \subset \ring\gamma$ and so $[\eta^{(j+1)}_{\ring \gamma}]=[\eta^{(j)}_{\ring \gamma}] = [\eta^{(0)}_{\ring \gamma}]$ by $\cP_j^4$.
\end{proof}

We cannot prove $\cP^1_{j+1}$ yet, which would be that $R_\ell$ and $R_h$ contain a critical parallelogram spanned in $\eta^{(j+1)}$. Instead, we establish a weaker result, that there exists at least one spanned critical parallelogram of $\eta^{(j+1)}_{\ring \gamma}$ ``to the left'' of $R_{\ell,\uparrow}$ (see Figure \ref{fig:main}), as well as one ``below'' $R_{h,\uparrow}$ (these two parallelograms may be the same). The idea of the proof is that since at the beginning of the path all spanned critical parallelograms intersect $R'$, if at step $j+1$ they have moved too much to the right, then they must have left a trail of infectable sites behind them, which constitutes a crossing.
\begin{claim}\label{claim:restriction1}
There exists a critical parallelogram contained in $\bbH_{u_3}(c-2L_{n-1}-7K)$ that is spanned in $\eta^{(j+1)}_{\ring\gamma}$ and similarly for $\bbH_{u_4}(d-2H_{n-1}-7K)$ (recall from Section \ref{subsec:bootstrap} that $\bbH_u(x)$ is the open half-plane directed by $u$ translated by a distance $x$ and $\bar\bbH_u(x)$ is its closure).
\end{claim}

\begin{proof}
We will only treat $H=\bbH_{u_3}(c-2L_{n-1}-7K)$ as the argument for the other half-plane is the same. Assume for a contradiction that there is no critical parallelogram contained in $H$ that is spanned in $\eta^{(j+1)}_{\ring\gamma}$. We will construct a crossing for $\eta^{(j+1)}$, which contradicts the fact that $\eta^{(j+1)}$ is $n$-good.

By Claim \ref{claim:rec:easy}, there exists a critical parallelogram $D$ spanned in $\eta^{(0)}$ intersecting $R'$. $D$ is then contained in $\ring\gamma$ (see Figure \ref{fig:main}), hence spanned in $\eta^{(0)}_{\ring\gamma}$. Let $X$ be a strongly connected set of $[D \cap \eta^{(0)}_{\ring\gamma}]$ such that the smallest parallelogram containing $X$ is $D$. Then, since $D$ is ``sufficiently to the left'' to intersect $R'$, $X$ intersects $\bar\bbH_{u_3}(c-2L_{n-1}-9K-\ell)$ (see Figure \ref{fig:main}). By Claim \ref{claim:P4} we have $[\eta_{\ring\gamma}^{(j+1)}]=[\eta^{(0)}_{\ring\gamma}]$, and $X\subset[D \cap \eta^{(0)}_{\ring\gamma}]\subset[\eta^{(0)}_{\ring\gamma}]$, so we can consider $X'$ the strongly connected component of $[\eta^{(j+1)}_{\ring\gamma}]$ containing $X$. Then $X'$ intersects $\bar\bbH_{u_3}(c-2L_{n-1}-9K-\ell)$.

We will now prove that $X'$ also intersects $\bar\bbH_{u_1}(-(c-2L_{n-1}-8K))$. The smallest parallelogram $D'$ containing $X'$ contains $X$, thus it contains $D$, so it has diameter at least $K/C_1$, since $D$ is critical (recall Definition \ref{def:span}). By Observation \ref{obs:connected:component}, $X'=[\eta^{(j+1)}_{\ring\gamma} \cap X'] \subset [\eta^{(j+1)}_{\ring\gamma \cap X'} \cap D']$, so $D'$ is spanned in $\eta^{(j+1)}_{\ring\gamma \cap X'}$. Hence, by Lemma \ref{lem:AL:body}, there exists a critical parallelogram $D''$ spanned for $\eta^{(j+1)}_{\ring\gamma \cap X'}$. By the assumption made at the beginning of the proof of the claim, $D''$ cannot be contained in $H$, so $D''$ intersects $\bar\bbH_{u_1}(-(c-2L_{n-1}-7K))$. Furthermore, $D''$ is critical, so its diameter is at most $K$. Thus, $D''$ is contained in $\bar\bbH_{u_1}(-(c-2L_{n-1}-8K))$. In addition, since $D''$ is spanned for $\eta^{(j+1)}_{\ring\gamma \cap X'}$, by Observation \ref{obs:obvious} $D''$ intersects $X'$, so $X'$ intersects $\bar\bbH_{u_1}(-(c-2L_{n-1}-8K))$.

We now construct the rectangle in which the crossing will take place. We denote $a_0=\max\{a'\,|\,X' \subset \bar\bbH_{u_1}(-a')\}$ (the ``left end'' of $X'$
). We claim $a_0 \geq a$. Indeed, $X'\subset[\eta^{(j+1)}_{\ring\gamma}]$, and since we have $\ring\gamma \subset \bbH_{u_1}(-a)$ and $u_1$ is a stable direction, we have $[\eta^{(j+1)}_{\ring\gamma}] \subset \bbH_{u_1}(-a)$, so $X' \subset \bbH_{u_1}(-a)$. Moreover, $a_0 \leq c-2L_{n-1}-9K-\ell$, since we proved that $X'$ intersects $\bar\bbH_{u_3}(c-2L_{n-1}-9K-\ell)$. Furthermore, we saw that by Observation \ref{obs:connected:component}, $X'=[\eta^{(j+1)}_{\ring\gamma} \cap X']$, so if we denote $R_{X'}=R(a_0,b;a_0+\ell,d)$, then $X'\subset[(\eta^{(j+1)}_{\ring\gamma} \cap R_{X'})\cup\bbH_{u_1}(-(a_0+\ell))]$. Together with the fact that $X'$ intersects $\bar\bbH_{u_1}(-(c-2L_{n-1}-8K)) \subset \bbH_{u_1}(-(a_0+\ell))$, this yields that $R_{X'}$ is $u_1$-crossed for $\eta^{(j+1)}_{\ring \gamma}$ (recall Definition \ref{def:crossing}). 

Moreover, since critical parallelograms have diameter at most $K$ and $a_0 \leq c-2L_{n-1}-9K-\ell$, any critical parallelogram intersecting $R_{X'}$ is contained in $H$. This will imply that there is no critical parallelogram spanned for $\eta^{(j+1)}_{\ring \gamma} \cap R_{X'}$, as by Observation \ref{obs:obvious} such a parallelogram would intersect $R_{X'}$, thus would be contained in $H$, which is impossible by the assumption made at the beginning of the proof of the claim. Since $R_{X'}$ is $u_1$-crossed for $\eta^{(j+1)}_{\ring \gamma}$ and there is no critical parallelogram spanned for $\eta^{(j+1)}_{\ring \gamma} \cap R_{X'}$, the event $C_{R_{X'}}^{u_1}$ occurs for $\eta^{(j+1)}$ (recall Definition \ref{def:crossing}). But this which is a contradiction with the fact that $R$ has no crossing for $\eta^{(j+1)}$, as it is a configuration in a $n$-good legal path. This contradiction concludes the proof of the claim.
\end{proof}

Claim \ref{claim:restriction1} will allow us to prove half of $\cP_{j+1}^3$, more precisely $\eta^{(j+1)}_{R_{\ell,\uparrow}} \in V(n-1,R_{\ell,\uparrow})$ and $\eta^{(j+1)}_{R_{h,\uparrow}} \in V(n-1,R_{h,\uparrow})$. The idea of the proof is that since Claim \ref{claim:restriction1} yields that at least one spanned critical parallelogram is to the left of $R_{\ell,\uparrow}$, and as the configurations are $n$-good, there are at most $n$ spanned critical parallelograms in total, this implies there are at most $n-1$ spanned critical parallelograms inside $R_{\ell,\uparrow}$ (and similarly for $R_{h,\uparrow}$).
\begin{claim}\label{claim:P3:1}
$\eta^{(j+1)}_{R_{\ell,\uparrow}} \in V(n-1,R_{\ell,\uparrow})$ and $\eta^{(j+1)}_{R_{h,\uparrow}} \in V(n-1,R_{h,\uparrow})$.
\end{claim}

\begin{proof}
We will only prove $\eta^{(j+1)}_{R_{\ell,\uparrow}} \in V(n-1,R_{\ell,\uparrow})$, as the other proof is similar. It will suffice to prove that $\eta^{(j+1)}_{R_{\ell,\uparrow}}$ is $(n-1)$-good. Indeed, by $\cP^3_j$ we have $\eta^{(j)}_{R_{\ell,\uparrow}} \in V(n-1,R_{\ell,\uparrow})$, hence in the case $\eta^{(j+1)}_{R_{\ell,\uparrow}}\neq\eta^{(j)}_{R_{\ell,\uparrow}}$ we can say that there exists an $(n-1)$-good legal path from $G(R_{\ell,\uparrow})$ to $\eta^{(j)}_{R_{\ell,\uparrow}}$. If we add $\eta^{(j+1)}_{R_{\ell,\uparrow}}$ to this path, we then obtain an $(n-1)$-good legal path from $G(R_{\ell,\uparrow})$ to $\eta^{(j+1)}_{R_{\ell,\uparrow}}$, hence $\eta^{(j+1)}_{R_{\ell,\uparrow}} \in V(n-1,R_{\ell,\uparrow})$.

We now prove that $\eta^{(j+1)}_{R_{\ell,\uparrow}}$ is $(n-1)$-good. Firstly, $R_{\ell,\uparrow}$ has no crossing for $\eta^{(j+1)}_{R_{\ell,\uparrow}}$ because $R$ has no crossing for $\eta^{(j+1)}$. It remains only to show that the maximal number $k$ of critical parallelograms that are disjointly spanned in $\eta^{(j+1)}_{R_{\ell,\uparrow}}$ is at most $n-1$. Let $D_1,\dots,D_k$ be critical parallelograms that are disjointly spanned in $\eta^{(j+1)}_{R_{\ell,\uparrow}}$. By Claim \ref{claim:restriction1}, there exists a critical parallelogram $D\subset \bbH_{u_3}(c-2L_{n-1}-7K)$ that is spanned in $\eta^{(j+1)}_{\ring \gamma}$ and, therefore, also in $\eta^{(j+1)}_D$. 
$D$ is then disjoint from $R_{\ell,\uparrow}$ (see Figure \ref{fig:main}), so we deduce that $D_1,\dots,D_k,D$ are disjointly spanned in $\eta^{(j+1)}$, so $\eta^{(j+1)}$ contains $k+1$ disjointly spanned critical parallelograms. Since $\eta^{(j+1)}$ is $n$-good, we get $k \leq n-1$, which ends the proof.
\end{proof}

We are now ready to prove $\cP_{j+1}^1$, that is, $R_\ell$ and $R_h$ contain a critical parallelogram spanned in $\eta^{(j+1)}$. To do that, we will prove that a spanned critical parallelogram in $\bbH_{u_4}(d-2H_{n-1}-7K)$ provided by Claim \ref{claim:restriction1} is in fact in $R_\ell$ (and similarly for $R_h$). The idea of the proof is that Claim \ref{claim:restriction1} prevents the parallelogram from being ``too far up,'' being ``too far left or down'' would induce a crossing, so will be impossible, and being ``too far right'' will be prevented because Claim \ref{claim:P3:1} will allow us to use $\cH_{n-1}$ in $R_{\ell,\uparrow}$.

\begin{claim}\label{claim:betterP1}
$\cP_{j+1}^1$ holds.
\end{claim}

\begin{proof}
We only treat $R_\ell$, as $R_h$ is similar. Assume for a contradiction that there is no critical parallelogram contained in $R_\ell$ spanned in $\eta^{(j+1)}$. 

By Claim \ref{claim:restriction1}, there exists a critical parallelogram $D$ contained in $\bbH_{u_4}(d-2H_{n-1}-7K)$ that is spanned in $\eta^{(j+1)}_{\ring\gamma}$. By assumption $D\not\subset R_\ell$. There are three possibilities (see Figure \ref{fig:main}):
\begin{enumerate}[label=\alph*)]
    \item\label{item:b} $D\cap R_{\ell,\downarrow}\neq\varnothing$, i.e.\ $D$ is ``too far left;''
    \item\label{item:c} $D\cap R_{h,\downarrow}\neq \varnothing$, i.e.\ $D$ is ``too far down;''
    \item\label{item:a} $D \cap \bar\bbH_{u_1}(-(c-L_{n-1}-7K))\neq\varnothing$, i.e.\ $D$ is ``too far right.''
\end{enumerate}

We first assume case \ref{item:a} occurs. Since $D$ is spanned in $\eta^{(j+1)}_{\ring\gamma}$, by Observation \ref{obs:obvious}, $D$ intersects $\ring \gamma$. In addition, since $D$ is critical it has diameter at most $K$. This yields that the intersection of $D$ and $\bar\bbH_{u_1}(-(c-L_{n-1}-7K))$ is in $R'_{\ell,\uparrow}$ (see Figure \ref{fig:main}).  Since $D$ intersects $R'_{\ell,\uparrow}$ and has diameter at most $K$, it is contained in $R_{\ell,\uparrow}$, hence spanned in $\eta_{R_{\ell,\uparrow}}^{(j+1)}$. However, by Claim \ref{claim:P3:1} we have $\eta^{(j+1)}_{R_{\ell,\uparrow}}\in V(n-1,R_{\ell,\uparrow})$, so $\cH_{n-1}$ implies there is no critical parallelogram intersecting $R'_{\ell,\uparrow}$ spanned in $\eta_{R_{\ell,\uparrow}}^{(j+1)}$, so we get a contradiction.

Cases \ref{item:b} and \ref{item:c} being analogous, we only treat case \ref{item:b}. Assume case \ref{item:b} occurs. The argument will resemble the one in the proof of Claim \ref{claim:restriction1}. Since $D$ is spanned in $\eta^{(j+1)}_{\ring\gamma}$, there exists a strongly connected set $X \subset [D \cap \eta^{(j+1)}_{\ring\gamma}]$ such that $D$ is the smallest parallelogram containing $X$. Let $X'$ be the strongly connected component of $[\eta^{(j+1)}_{\ring\gamma}]$ containing $X$, and $D'$ be the smallest parallelogram containing $X'$. Then $D'$ contains $X$, hence it contains $D$ and, since $D$ is critical, $\diam(D')\ge\diam(D)\ge K/C_1$. Furthermore, $X'$ is a strongly connected component of $[\eta^{(j+1)}_{\ring\gamma}]= [\eta^{(0)}_{\ring\gamma}]$ by Claim \ref{claim:P4}, so $X' = [\eta^{(0)}_{\ring\gamma} \cap X']$ by Observation \ref{obs:connected:component}. This implies $X' \subset [\eta^{(0)}_{\ring\gamma \cap X'} \cap D']$, hence $D'$ is spanned in $\eta^{(0)}_{\ring\gamma\cap X'}$.

Since $\diam(D')\ge K/C_1$, Lemma \ref{lem:AL:body} implies that there exists a critical parallelogram $D''$ that is spanned in $\eta^{(0)}_{\ring\gamma\cap X'}$. Then $D''$ is spanned in $\eta^{(0)}=\eta$ and therefore intersects $R'$ by the definition of $\eta$. Since $D''$ is critical, its diameter is at most $K$, so, since $D''$ intersects $R'$, it is contained in $\bar\bbH_{u_1}(-a-2L_{n-1}-8K-\ell)$ (see Figure \ref{fig:main}). Moreover, since $D''$ is spanned in $\eta^{(0)}_{\ring\gamma\cap X'}$, by Observation \ref{obs:obvious}, $X'$ intersects $D''$ hence $X'$ intersects $\bar\bbH_{u_1}(-a-2L_{n-1}-8K-\ell)$. In addition, since $D$ intersects $R_{\ell,\downarrow}$ by assumption \ref{item:b} and is the smallest parallelogram containing $X$, then $X$ intersects $\bar\bbH_{u_3}(a+2L_{n-1}+7K)$, thus $X'$ intersects this half-plane as well.

Denote $a_0=\max\{a'\,|\,X' \subset \bar\bbH_{u_1}(-a')\}$ (the ``left end'' of $X'$). Then $a_0 \geq a$, since $X' \subset [\eta^{(0)}_{\ring\gamma}]\subset[\ring\gamma]\subset[\bbH_{u_1}(-a)]=\bbH_{u_1}(-a)$ as $u_1$ is a stable direction, and $a_0 \leq a+2L_{n-1}+7K$, as $X'$ intersects $\bar\bbH_{u_3}(a+2L_{n-1}+7K)$. As in Claim \ref{claim:restriction1} (recalling that $X'$ intersects $\bar\bbH_{u_1}(-a-2L_{n-1}-8K-\ell)$), this entails that $R_{X'}=R(a_0,b;a_0+\ell,d)$ is $u_1$-crossed for $\eta^{(0)}$ (recall Definition \ref{def:crossing}). However, $R$ has no crossing for $\eta^{(0)}$ as $\eta^{(0)}$ is $n$-good, so $C_{R_{X'}}^{u_1}$ does not occur for $\eta^{(0)}$. Consequently, there exists a critical parallelogram spanned for $\eta^{(0)}_{R_{X'}}$, hence for $\eta^{(0)}$. By Observation \ref{obs:obvious} this parallelogram intersects $R_{X'}$, so since critical parallelograms have diameter at most $K$ and $a_0 \leq a+2L_{n-1}+7K$ it does not intersect $R'$, which contradicts the fact that all critical parallelograms spanned in $\eta^{(0)}=\eta$ intersect $R'$.
\end{proof}

The previous claim will allow us to complete the proof of $\cP_{j+1}^3$. We already proved in Claim \ref{claim:P3:1} that $\eta^{(j+1)}_{R_{\ell,\uparrow}} \in V(n-1,R_{\ell,\uparrow})$ and $\eta^{(j+1)}_{R_{h,\uparrow}} \in V(n-1,R_{h,\uparrow})$, so we need $\eta^{(j+1)}_{R_{\ell,\downarrow}} \in V(n-1,R_{\ell,\downarrow})$ and $\eta^{(j+1)}_{R_{h,\downarrow}} \in V(n-1,R_{h,\downarrow})$. The argument is identical to that of Claim \ref{claim:P3:1}: by Claim \ref{claim:betterP1} we have a spanned critical parallelogram in $R_\ell$, which entails that there are at most $n-1$ critical spanned parallelograms in $R_{\ell,\downarrow}$ (and similarly for $R_{h,\downarrow}$), hence the result.
\begin{claim}\label{claim:P3:2}
$\cP_{j+1}^3$ holds.
\end{claim}

\begin{proof}
The proof is actually the same as in Claim \ref{claim:P3:1}, replacing $\bbH_{u_3}(c-2L_{n-1}-7K)$ by $R_\ell$ or $R_h$ and Claim \ref{claim:restriction1} by Claim \ref{claim:betterP1}.
\end{proof}

It remains only to prove $\cP_{j+1}^2$, i.e.\ that the sites of $F$ are not locally infectable in $\eta^{(j+1)}$.
\begin{claim}\label{claim:P2}
$\cP_{j+1}^2$ holds.
\end{claim}

\begin{proof}
Assume for a contradiction that there exists $s' \in F\subset B'$ that is locally infectable in $\eta^{(j+1)}$. By $\cP_j^2$, $s'$ is not locally infectable in $\eta^{(j)}$. Therefore, by Lemma \ref{lem:locally:infectable}, there exist $\zeta\in\{\ell,h\}$, $\xi\in\{\downarrow,\uparrow\}$ and a critical parallelogram $D\subset R'_{\zeta,\xi}$ spanned in $\eta^{(j+1)}_{R_{\zeta,\xi}}$. However, by Claim \ref{claim:P3:2} $\eta^{(j+1)}_{R_{\zeta,\xi}} \in V(n-1,R_{\zeta,\xi})$, so $\cH_{n-1}$ yields the desired contradiction.
\end{proof}

Claims \ref{claim:P4}, \ref{claim:betterP1}, \ref{claim:P3:2} and \ref{claim:P2} together establish the induction step, which completes the proof of Lemma \ref{lem:rec:main}.
\end{proof}

\section{Application of Proposition~\ref{prop:main}}
\label{sec:proof}
In this section we derive Theorems \ref{th:infinite} and \ref{th:finite} from Proposition \ref{prop:main}. For that purpose we require some estimates on the probabilities appearing in the statement of the proposition, which are mostly proved in the appendices. We restate those results below as needed. Throughout the section $\cU$ is a critical update family with difficulty $\alpha$ subject to further assumptions recalled in each subsection. Such a family admits two non-collinear rational stable directions. We set $u_1$ and $u_2$ to be two arbitrary such directions, which will be chosen differently for each class of update families. We will use the definitions of Section \ref{subsec:general} with $u_1$ and $u_2$.

Let us start with the easiest estimate.
\begin{lem}
\label{lem:loc:inf:bound}
Let $K\le q^{-2\alpha}$. Then,
\[\mu(0\text{ is locally infectable})\le 1/8.\]
\end{lem}
\begin{proof}
Let $R=R(-2K,-2K;2K,2K)$. Since $\cU$ is critical, $\diam([\eta_{R}])\le C_1K$, so, starting the bootstrap percolation dynamics with $\eta_R$, the origin is either infected in time at most $C_1^3K^2$ or not at all. We conclude using e.g.\ \cite{Bollobas15}*{Theorem 1.4}, which gives that with probability tending to 1 as $q\to0$, the infection time is $\exp(q^{-\Theta(1)})$.
\end{proof}

Turning to the probability of spanning there are two cases to consider. For unbalanced models the following is essentially a reformulation of the most difficult result of \cite{Bollobas14}.
\begin{lem}
\label{lem:unbalanced:crit:proba}
Assume that $\cU$ is unbalanced and $K=q^{-\alpha-1/4}$. Then, for any critical parallelogram $D$ we have
\[\mu(D\text{ is spanned})\le\exp\left(-\frac{(\log(1/q))^2}{C_5q^\alpha}\right).\]
\end{lem}
\begin{proof}
By definition if $D$ is a spanned critical parallelogram, then there exists a strongly connected set $X \subset [\eta \cap D]$ with diameter at least $\diam(D)/C_1$. If $X'$ is the strongly connected component of $X$ in $[\eta \cap D]$, then by Observation \ref{obs:connected:component}, $X' = [\eta \cap D \cap X']$. Therefore, if $D'$ is the smallest $S_U$-droplet in the sense of \cite{Bollobas14}*{Definition 2.1} with the $S_U$ defined in \cite{Bollobas14}*{Lemma 6.2} such that $D'$ contains $X'$, then $D'$ is internally spanned by $\eta \cap D$ (in the sense of \cite{Bollobas14}*{Definition 2.4}\footnote{\label{foot:constants}Technically, it is not exactly the case, as \cite{Bollobas14} uses a different choice of constants. However, their results that we use still hold in our setting.}), and $\diam(D')\geq\diam(D)/C_1\ge K/C_1^2$. Repeating the proof of \cite{Bollobas14}*{Lemma 8.37}, we get that there is a critical droplet in the sense of \cite{Bollobas14}*{Definition 2.5} internally spanned by $\eta\cap D$. Then the union bound over such droplets and \cite{Bollobas14}*{Lemma 8.36} yield the desired result.
\end{proof}

Concerning balanced models, in Appendix \ref{app:spanning} (Corollary \ref{cor:bounds:app}) we establish the following, by combining the techniques of \cites{Hartarsky20,Bollobas14}.
\begin{lem}
\label{lem:bound:crit:balanced}Assume that $\cU$ is balanced and $q^{-\alpha}/C_5\le K\le q^{-2\alpha}$. Then for any critical parallelogram $D$ we have
\[\mu(D\text{ is spanned})\le \exp\left(-\frac{1}{C_5q^{\alpha}}\right).\]
\end{lem}

For the remaining conditions of Proposition \ref{prop:main} we need to distinguish the different classes of models.

\subsection{Proof of Theorem \ref{th:infinite}}
\label{subsec:infinite}
In this section we assume that $\cU$ has an infinite number of stable directions. We then choose two rational directions $u_1<u_2<u_1+\pi$ sufficiently close to each other, such that all directions in $[2u_1-u_2,2u_2-u_1]$ are stable and $u_1$, $u_2$ satisfy a technical condition which the reader is advised to ignore, namely that $u_1$ and $u_2$ are constructed like the eponymous directions in the proof of \cite{Hartarsky20}*{Lemma 4.1}.

\subsubsection{Proof of Theorem \ref{th:infinite}\ref{log4}}
\label{subsubsec:infinite:log4}
For this section we further assume that $\cU$ is unbalanced. We fix the values of the parameters of Proposition \ref{prop:main} as follows.
\begin{align*}
K&{}=q^{-\alpha-1/4}&
\ell&{}=q^{-4\alpha}&
L&{}=\exp\left(\frac{(\log(1/q))^2}{C_6q^{\alpha}}\right)\\
T&{}=\exp\left(\frac{(\log (1/q))^4}{C_6^2q^{2\alpha}}\right)&
h&{}=q^{-4\alpha}&
H&{}=\exp\left(\frac{(\log(1/q))^2}{C_6q^{\alpha}}\right).
\end{align*}

\begin{proof}[Proof of Theorem \ref{th:infinite}\ref{log4}]
The upper bound was proved in \cite{Martinelli19a}*{Theorem 2(a)}, so we focus on the lower one. We will apply Proposition \ref{prop:main} with the above choice of parameters, so that we obtain the desired conclusion: $\bbE_\mu[\tau_0]\ge T$. Hence, it suffices to verify the hypotheses of the proposition. Indeed, setting $n=(\log(1/q))^2/(2C_6q^\alpha)$, we have
\[L=\exp\left(\frac{(\log(1/q))^2}{C_6q^{\alpha}}\right)\ge 3^n\cdot2q^{-4\alpha}\ge 3^n(11K+\ell)\]
and similarly for $H$. By Lemma \ref{lem:loc:inf:bound} we do have $\mu(0\text{ is locally infectable})\le 1/8$. Moreover, recalling that $\rho\le \exp(-(\log(1/q))^2/(C_5q^\alpha))$ by Lemma \ref{lem:unbalanced:crit:proba}, we obtain
\begin{align*}TLH(LHK^3\rho)^{n+1}\le{}& \exp\left(\frac{2(\log(1/q))^4}{C_6^2q^{2\alpha}}+n\left(\frac{3(\log(1/q))^2}{C_6q^\alpha}-\frac{(\log(1/q))^2}{C_5q^\alpha}\right)\right)\\
\le{}&\exp\left(\frac{(\log(1/q))^4}{C_6q^{2\alpha}}\left(\frac{2}{C_6}+\frac{3}{2C_6}-\frac{1}{2C_5}\right)\right)\le 1.\end{align*}
Finally,
\[T(LH)^2\le \exp\left(\frac{2(\log(1/q))^4}{C_6^2q^{2\alpha}}\right)\le \exp\left(q^{-3\alpha}\right)\]
and Lemma \ref{lem:crossing:body:infinite} below yields $\max(\pl,\pd)\le \exp(-q^{-3\alpha})$. Therefore, we obtain $T(LH)^2\max(\pl,\pd)\le 1$ and all the hypotheses of Proposition \ref{prop:main} are verified once we establish Lemma \ref{lem:crossing:body:infinite}.
\end{proof}
\begin{lem}
\label{lem:crossing:body:infinite}
With the notation and assumptions above we have
\[\max(\pl,\pd)\le \exp\left(-q^{-3\alpha}\right).\]
\end{lem}
This bound is proved in Appendix \ref{app:crossing} (Lemma \ref{lem:crossing:bounds:boundary}), but let us provide a rough sketch for the reader's convenience.
\begin{proof}[Sketch of the proof of Lemma \ref{lem:crossing:body:infinite}]
The proof is based on the concept of \emph{hierarchies} introduced by Holroyd in \cite{Holroyd03} and extensively studied since then. 

Consider the rectangle $R=R(-\ell,0;0,H)$ and assume that $C_R^{u_1}$ occurs with $\eta'=\eta$ in Definition \ref{def:crossing}. Fix some $u_1$-crossing $X\subset[\bbH_{u_1}\cup(R\cap\eta)]$ of $R$. One can retrace how $X$ becomes infected by the bootstrap percolation process by a method known as a \emph{spanning algorithm} as follows. Starting from single infections, we lump them into groups of infections by progressively merging two groups if their closures are close to each other. Since $X$ has diameter at least $\ell$, this process will eventually produce a set of infections of diameter roughly $\ell$.
We associate to each set of infections a \emph{droplet} (appropriately shaped polygon) containing its closure. 

We can then view the spanning algorithm run backwards as a progressive shattering of the initial droplet of size $\ell$ and record its history. However, we do not wish to register the entire history of all splittings. When a droplet splits into two large droplets, we study the subsequent splittings of both droplets, but when a large droplet splits into a large one plus a small one, we ignore the small one. Still, if there are many successive such small splittings, we occasionally write down one of the resulting large droplets. We stop studying the splitting when the droplets are of size roughly $K$. We then call these droplets of size $K$ \emph{seeds}. In total, this gives us a tree of large droplets called \emph{hierarchy} recording how the droplet of size $\ell$ is produced.

Recall that the event $C_R^{u_1}$ requires the absence of spanned critical parallelograms, which prevents the existence of clusters of infectable sites of size at least $K$, hence of droplets of size at least $K$. This seemingly forbids the hierarchy to exist. However, in the above we overlooked the presence of the infected boundary $\bbH_{u_1}$. To take it into account, we actually consider hierarchies of \emph{cut droplets} which are ``very flat triangles sticking out of the boundary.'' 

To bound the probability of $C_R^{u_1}$, we bound the probability that such a hierarchy occurs. There are two possibilities: either there are many splittings into two large droplets, or much of the splitting is done into a large and a small droplet. In the first case, there will be many seeds. We establish a preliminary upper bound on the probability of occurrence of a seed of order $\exp(-q^{-\alpha})$. Since each seed should be contained in the original droplet of size $\ell$, there are few possible choices for them. We deduce an appropriate bound on the probability of such a hierarchy.

Assuming that most of the size $\ell$ of the initial droplet is gained by splittings into a large and a small droplet, let us consider one such step. This means that in the spanning algorithm, when we add to the infections of a large (cut) droplet the few infections contained in a small droplet, we get to infect a slightly larger droplet. This essentially implies the existence of an infectable set going from the boundary of the large droplet to that of the larger one. This in turn yields the occurrence of another cut droplet of size at least $K$. Hence, the preliminary bound on the probability of occurrence of a seed allows us to bound the probability of such a splitting, concluding the proof.
\end{proof}

\subsubsection{Proof of Theorem \ref{th:infinite}\ref{log0}}
\label{subsubsec:infinite:log0}
For this section we further assume that $\cU$ is balanced. We fix the values of the parameters of Proposition \ref{prop:main} as follows.
\begin{align*}
K&{}=q^{-\alpha}&
\ell&{}=q^{-4\alpha}&
L&{}=\exp\left(\frac{1}{C_6q^{\alpha}}\right)\\
T&{}=\exp\left(\frac{1}{C_6^2q^{2\alpha}}\right)&
h&{}=q^{-4\alpha}&
H&{}=\exp\left(\frac{1}{C_6q^{\alpha}}\right).
\end{align*}

Then Theorem \ref{th:infinite}\ref{log0} follows directly from Proposition \ref{prop:main} and the upper bound from \cite{Hartarsky20II}*{Theorem 1(b)}. Setting $n=1/(2C_6q^\alpha)$, the hypotheses of the proposition follow like in Section \ref{subsubsec:infinite:log4} from the choice of parameters, Lemmas \ref{lem:loc:inf:bound} and \ref{lem:bound:crit:balanced}, and Lemma \ref{lem:crossing:body:infinite}, which still applies.

\subsection{Proof of Theorem \ref{th:finite}}
\label{subsec:proof:finite}
\subsubsection{Proof of Theorem~\ref{th:finite}\ref{log3}}
\label{subsubsec:infinite:log3}
In this section we assume that $\cU$ is unbalanced, rooted and has a finite number of stable directions. Therefore, we can find rational directions $u_1,u_2,u_3$ such that $u_1+\pi=u_3$, $u_2\in(u_1,u_3)$ and $\alpha(u_i)\ge\alpha+1$ for all $i\in\{1,2,3\}$. 

We fix the values of the parameters of Proposition \ref{prop:main} as follows.
\begin{align*}
K&{}=q^{-\alpha-1/4}&
\ell&{}=q^{-\alpha-5/8}&
L&{}=q^{-\alpha-3/4}\\
T&{}=\exp\left(\frac{(\log (1/q))^3}{C_6q^{\alpha}}\right)&
h&{}=q^{-\alpha-5/8}&
H&{}=q^{-\alpha-3/4}.
\end{align*}

\begin{proof}[Proof of Theorem \ref{th:infinite}\ref{log3}]
The upper bound was proved in \cite{Hartarsky21a}*{Theorem 1}, so we focus on the lower one. As previously, it suffices to verify the hypotheses of Proposition \ref{prop:main} with the above choice of parameters. Indeed, setting $n=\log (1/q)/C_1$, we have
\[L=q^{-\alpha-3/4}\ge 3^n\cdot2q^{-\alpha-5/8}\ge 3^n(11K+\ell)\]
and similarly for $H$. By Lemma \ref{lem:loc:inf:bound} we do have $\mu(0\text{ is locally infectable})\le 1/8$. Moreover, recalling that $\rho\le \exp(-(\log(1/q))^2/(C_5q^\alpha))$ by Lemma \ref{lem:unbalanced:crit:proba}, we obtain
\begin{align*}TLH(LHK^3\rho)^{n+1}\le{}&\exp\left(\frac{2(\log(1/q))^3}{C_6q^{\alpha}}-\frac{n(\log(1/q))^2}{2C_5q^\alpha}\right)\\
\le{}&\exp\left(\frac{(\log(1/q))^3}{q^{\alpha}}\left(\frac{2}{C_6}-\frac{1}{2C_1C_5}\right)\right)\le 1.\end{align*}
Finally,
\[T(LH)^2\le \exp\left(\frac{2(\log(1/q))^3}{C_6q^{\alpha}}\right)\le \exp\left(q^{-\alpha-1/4}\right).\]
Thus, once we establish Lemma \ref{lem:crossing:body:finite} below, all the hypotheses of Proposition \ref{prop:main} are verified.
\end{proof}

\begin{lem}
\label{lem:crossing:body:finite}
With the notation and assumptions above we have
\[\max(\pl,\pd)\le \exp\left(-q^{-\alpha-1/4}\right).\]
\end{lem}
This bound is proved in Appendix \ref{app:crossing} (Lemma \ref{lem:crossing:bounds}). Since the proof is quite different from the one of Lemma \ref{lem:crossing:body:infinite}, we also provide a sketch.
\begin{proof}[Sketch of the proof of Lemma \ref{lem:crossing:body:finite}]
Consider $R=R(-\ell,0;0,H)$. Our goal is essentially to prove that $\mu(C_R^{u_1})\le e^{-\ell}$. To do this, we use an improvement of the \emph{partition} method of \cite{Bollobas14}*{Section 8.3}. We cut the rectangle into strips of a large constant width and lump the strips together into groups crossed by a spanned parallelogram. We establish that it is deterministically necessary for all strips to either be intersected by a spanned parallelogram as above or to contain a set of $\alpha(u_1)\ge \alpha+1$ infections. Yet, $H$ is much smaller than $q^{-\alpha-1}$, so it is unlikely that a strip contains $\alpha+1$ infections. As for the strips intersected by spanned parallelograms, since $C_R^{u_1}$ assumes that no critical parallelogram is spanned, we only need to care about smaller spanned parallelograms, for which one may prove a preliminary probability bound exponentially small in their size. Combining these facts and taking into account that the size of $R$ is only polynomial in $1/q$, we can control the entropy and conclude that the probability of a crossing is indeed exponentially small in its width.\end{proof}

\subsubsection{Proof of Theorem~\ref{th:finite}\ref{log1}}
In this section we assume that $\cU$ is balanced, rooted and with a finite number of stable directions. Therefore, we can find non-opposite rational directions $u_1,u_2$ such that $\alpha(u_1)\ge\alpha+1$ and $\alpha(u_2)\ge \alpha+1$. We fix the values of the parameters of Proposition \ref{prop:main} as follows.
\begin{align*}
K&{}=1/(C_5q^{\alpha})&
\ell&{}=q^{-\alpha-1/2}&
L&{}=q^{-\alpha-3/4}\\
T&{}=\exp\left(\frac{\log (1/q)}{C_6q^{\alpha}}\right)&
h&{}=q^{-\alpha-1/2}&
H&{}=q^{-\alpha-3/4}.
\end{align*}

Then Theorem \ref{th:finite}\ref{log1} follows directly from Proposition \ref{prop:main} and the upper bound in \cite{Hartarsky20II}*{Theorem 1(e)}. Setting $n=\log (1/q)/C_1$, the hypotheses of the proposition follow like in Section \ref{subsubsec:infinite:log3} from the choice of parameters, Lemmas \ref{lem:loc:inf:bound} and \ref{lem:bound:crit:balanced}, and Lemma \ref{lem:crossing:body:finite}, which still holds.

\subsubsection{Proof of Theorem~\ref{th:finite}\ref{loglog}}
In this section we assume that $\cU$ is semi-directed. Therefore, we can find non-opposite rational directions $u_1,u_2$ such that $\alpha(u_1)=\alpha$ and $\alpha(u_2)\ge \alpha+1$. We fix the values of the parameters of Proposition \ref{prop:main} as follows.
\begin{align*}
K&{}=1/(C_5q^{\alpha})&\ell&{}=q^{-\alpha-1/2}&L&{}=q^{-\alpha-3/4}\\
T&{}=\exp\left(\frac{\log\log(1/q)}{C_6^3q^{\alpha}}\right)&
h&{}=\frac{\log\log (1/q)}{q^{\alpha}}&
H&{}=\frac{(\log (1/q))^{1/4}}{q^{\alpha}}.
\end{align*}

\begin{proof}[Proof of Theorem \ref{th:infinite}\ref{loglog}]
The upper bound is proved in \cite{Hartarsky20II}*{Theorem 1(f)}, so we focus on the lower one. As previously, it suffices to verify the hypotheses of Proposition \ref{prop:main} with the above choice of parameters. Indeed, setting $n=\log\log(1/q)/C_1$, we have
\begin{align*}
L={}&q^{-\alpha-3/4}\ge 3^n\cdot 2q^{-\alpha-1/2}\ge 3^n(11K+\ell),\\
H={}&\frac{(\log(1/q))^{1/4}}{q^\alpha}\ge 3^n\frac{2\log\log(1/q)}{q^\alpha}\ge 3^n(11K+h).\end{align*}
By Lemma \ref{lem:loc:inf:bound} we do have $\mu(0\text{ is locally infectable})\le 1/8$. Moreover, recalling that $\rho\le \exp(-1/(C_5q^\alpha))$ by Lemma \ref{lem:bound:crit:balanced}, we obtain
\begin{align*}TLH(LHK^3\rho)^{n+1}\le{}&\exp\left(\frac{2\log\log(1/q)}{C_6^3q^{\alpha}}-\frac{n}{2C_5q^\alpha}\right)\\
\le{}&\exp\left(\frac{\log\log(1/q)}{q^{\alpha}}\left(\frac{2}{C_6^3}-\frac{1}{2C_1C_5}\right)\right)\le 1.\end{align*}
Finally,
\[T(LH)^2\le \exp\left(\frac{2\log\log(1/q)}{C_6^3q^{\alpha}}\right)\le \exp\left(\frac{\log\log(1/q)}{2C_6^2q^\alpha}\right)\le \exp(q^{-\alpha-1/4}).\]
Thus, once we establish Lemma \ref{lem:crossing:body:loglog} below, all the hypotheses of Proposition \ref{prop:main} are verified.
\end{proof}

\begin{lem}
\label{lem:crossing:body:loglog}
With the notation and assumptions above we have
\begin{align*}
    \pl&{}\le\exp\left(-q^{-\alpha-1/4}\right)&
    \pd&{}\le\exp\left(-\frac{\log\log(1/q)}{2C_6^2q^\alpha}\right).
\end{align*}
\end{lem}
This bound is proved in Appendix \ref{app:crossing} (Lemma \ref{lem:crossing:bounds}) like Lemma \ref{lem:crossing:body:finite}.

\section*{Acknowledgements}
This work has been supported by ERC Starting Grant 680275 MALIG. We wish to thank Cristina Toninelli and Fabio Martinelli for enlightening discussions on the universality of $\cU$-KCM. Most of this work was done when the second author was affiliated to the École Polytechnique Fédérale de Lausanne, Switzerland. The first author would also like to thank ÉPFL for the hospitality during his visit. We thank the anonymous referee for helpful comments on the presentation of the paper.
\appendix
\section{Bounds on spanning}
\label{app:spanning}
\paragraph{Relation to previous works}Let us start by explaining why additional arguments are needed, as specialists would probably expect such bounds to be automatic. In \cite{Bollobas14} two main algorithms were used---the covering and the spanning ones. The former provides bounds of the type we need but for a notion of covered droplet invoking only the initial configuration. On the other hand, the spanning algorithm works with the closure of the initial configuration inside droplets, which obstructs obtaining results analogous to those for the covering algorithm in the same way. Yet, it is the spanning algorithm which is the most useful and particularly so for unbalanced models. In \cite{Bollobas14} an inductive multi-scale scheme was used to bootstrap the bounds on the probability of droplets being spanned from a size which is easily controlled by the more rudimentary predecessor of the covering algorithm developed in \cite{Bollobas15}. This fairly technical procedure can be circumvented using our method. Indeed, if one has bounds analogous to the ones for covered droplets up to size $1/q^{\alpha(\cU)}$, one can directly prove the result of \cite{Bollobas14} in one step, which was made there as well.

The reason why in \cite{Bollobas14} one could not directly transfer the easier bounds on covering, which were established there anyway, to spanning is that the covering algorithm there lacks the key property of being essentially closure-invariant in a sense made precise below. This property was one of the main features gained in \cite{Hartarsky20} by using a less wasteful notion of cluster. Therefore, we accomplish our goal as follows. We carry through (a simplified version of) the scheme of \cite{Hartarsky20} to obtain general bounds for droplets covered in the sense of \cite{Hartarsky20} and we use the key closure lemma (see below) to directly transfer those to spanning. On the more technical level, we should mention that analogous bounds on spanning were established in \cite{Bollobas14} in the course of their induction, but the proof needlessly uses that the model is unbalanced and constrains the choice of directions used for defining droplets, which we will need to choose freely. Moreover, \cite{Hartarsky20} made unnecessary use of the existence of strongly stable directions\footnote{Strongly stable directions are those contained in the topological interior of the set of stable directions.}, which is only needed for treating the algorithm with boundary condition. We are thus obliged to review the proofs. The reader familiar with the details of \cite{Hartarsky20} would probably be satisfied by skipping directly to Appendix \ref{app:crossing} and consulting the statements as needed there.

\paragraph{Outline} The appendix is structured as follows. In section \ref{app:sec:covering} we recall several results from \cite{Hartarsky20}, leading up to Lemma \ref{lem:bounds:covered} providing good bounds on the probability of being covered (in a sense made precise below) and to the Closure lemma \ref{lem:closure} relating the results of the covering algorithm for a set and for its closure. In section \ref{app:sec:spanning}, using the latter lemma we transfer the bounds of the former one to the notion of spanning used throughout the body of the paper. Section \ref{app:sec:boundary} establishes, yet again, the same bounds on the probability of spanned droplets occurring, but in the presence of an infected boundary, following the same reasoning and relying more closely on \cite{Hartarsky20}.

\paragraph{Notation}For the remainder of the paper we fix an arbitrary critical update family $\cU$ with difficulty $\alpha$. Following \cite{Hartarsky20} we consider constants
\[1\ll C_1\ll C'_2\ll C_2\ll C_3\ll C_4'\ll C_4\ll C_5\ll C_6\]
such that each one is larger than a suitable function of the previous ones, depending on $\cT$, $\cT_0$, $\cS_u$, etc.\ to be defined below and on $\cU$. These constants do not depend on $q$, which is always assumed small enough, as we are interested in $q\to0$.

For any finite set of directions $\cV\subset S^1$ a $\cV$-droplet is a set of the form $\bigcap_{v\in\cV}\bar \bbH_{v}(a_v)$ for some $a_v\in\bbR$.

\subsection{Covering and closure}
\label{app:sec:covering}
We start by studying the covering algorithm in the spirit of \cite{Hartarsky20}*{Section 5} (but without the boundary and rugged edge present there). The reader is invited to consult that work for most proofs and more details, as indicated below. By definition \ref{def:diff} we can fix a set of non semi-isolated rational stable directions\footnote{Semi-isolated stable directions are the endpoints of intervals of stable directions with nonempty interior.} $\cT_0$ with difficulty at least $\alpha$, such that the convex envelope of the elements of $\cT_0$ contain $0$ in its interior and either
\begin{itemize}
\item $|\cT_0|=3$ or
\item $|\cT_0|=4$ and one has $\cT_0=\{u,v,u+\pi,v+\pi\}$ for some $u,v\in S^1$.
\end{itemize}

Let $\Gamma$ be the graph with vertex set $\bbZ^2$ but with $x\sim y$ iff $\|x-y\|\le C_2$. 

\begin{defn}[\cite{Hartarsky20}*{Definitions 5.1 and 5.3}]
\label{def:cluster:app}Fix a finite set $Z\subset\bbZ^2$. Let $\kappa$ be a connected component of the subgraph of $\Gamma$ induced by the vertex set $Z$. 
\begin{itemize}
    \item $\kappa$ is a \emph{crumb} for $Z$ if there exists a set $P_\kappa\subset \bbZ^2$ such that $[P_\kappa]\supset \kappa$ and $|P_\kappa|=\alpha-1$.
    \item If $\kappa$ is not a crumb for $Z$, we say that a $C\subset\kappa$ is a $\alpha$-\emph{cluster} (or simply cluster) of $Z$ if the following conditions hold
    \begin{itemize}
        \item $\diam(C)\le C_3$.
        \item $C$ is connected in $\Gamma$.
        \item For all $x\in\kappa\setminus C$ and $y\in C$ such that $x\sim y$ in $\Gamma$ we have $\diam(C\cup\{x\})>C_3$.
    \end{itemize}
\end{itemize}
\end{defn}

It can be proved \cite{Hartarsky20}*{Observation 5.4} that any cluster contains at least $\alpha$ sites. Moreover, Corollary 5.17 of \cite{Hartarsky20} yields that a crumb has diameter at most $\alpha C_2$. For a cluster $C$ we denote by $Q(C)$ the smallest $\cT_0$-droplet containing the set $\{x\in \bbR^2:d(x,C)\le C_4\}$.

We next define the covering algorithm we will use. It is an adaptation of the droplet algorithm of \cite{Hartarsky20} and should not be confused with the covering algorithms of \cites{Bollobas15,Bollobas14}.
\begin{defn}[Covering algorithm]
\label{def:covering}
Given a finite set $Z\subset\bbZ^2$ of infections the \emph{covering algorithm} outputs a set $\cD$ of disjoint $\cT_0$-droplets as follows.
\begin{itemize}
\item Form an initial collection $\cD$ of $\cT_0$-droplets consisting of $Q(C)$ for all clusters $C$ of $Z$.
\item Whenever there exist $D_1,D_2\in\cD$ with $D_1\cap D_2\neq\varnothing$, replace them with the smallest $\cT_0$-droplet containing their union, which we denote by $D_1\vee D_2$.
\item Output the collection $\cD$ obtained when all $\cT_0$ droplets in $\cD$ are disjoint.
\end{itemize}
Equivalently, $\cD$ is the minimal collection (with respect to inclusion of the union of its elements) of disjoint $\cT_0$-droplets containing the union of $Q(C)$ for all clusters $C$ of $Z$. In particular, $\cD$ does not depend on the order in which droplets are merged.

We say that a $\cT_0$-droplet $D$ is \emph{covered} by a set $Z$ of infections if the above algorithm for $Z\cap D$ outputs a $\cT_0$-droplet containing $D$.
\end{defn}
We make the convention that all $\cT_0$-droplets have diameter at least $C_4'$ and contain a site of $\bbZ^2$.

We next state some properties of the covering algorithm.
\begin{lem}[Lemma 4.6 of \cite{Bollobas15}]
\label{lem:subadd}
Let $D_1$ and $D_2$ be $\cT_0$-droplets such that $D_1\cap D_2\neq \varnothing$. Then
\[\diam(D_1\vee D_2)\le\diam(D_1)+\diam(D_2).\]
\end{lem}
This immediately implies the Aizenman-Lebowitz lemma (see e.g.\ \cite{Bollobas15}*{Lemma 4.8}).
\begin{lem}[Aizenman-Lebowitz]
\label{lem:AL:app}
Let $Z$ be a set of infections and $D$ be a $\cT_0$-droplet covered by $Z$. Then for all $C_1C_4\le k\le \diam(D)$ there exists a $\cT_0$-droplet $D'$ covered by $Z$ with $k\le\diam(D')\le 2k$.
\end{lem}
A further consequence of Lemma \ref{lem:subadd} is the following.
\begin{lem}[Lemma 5.14 of \cite{Hartarsky20}]
\label{lem:extremal}
Let $Z$ be a set of infections and $D$ be a $\cT_0$-droplet covered by $Z$. Then $D$ contains at least $\lceil \diam(D)/C_4^2\rceil$ disjoint clusters of $Z \cap D$.
\end{lem}
We are now able to deduce the relevant bounds on covering following \cite{Hartarsky20}*{Lemma 5.15}.
\begin{lem}
\label{lem:bounds:covered}
Let $D$ be a $\cT_0$-droplet with $d=\diam(D)$. Let $1>\epsilon>0$. Then we have
\begin{equation}\label{eq:bounds:covered}
\mu(D\text{ is covered})\le\begin{cases} q^{d\epsilon/(3C_4^2)}&\text{ if } d\le \frac{C_1}{q^{\alpha-\epsilon}}\\e^{-C_1C_4d} &\text{ if }\frac{1}{C_1q^{\alpha-\epsilon}}\le d\le \frac{C_1}{e^{C_4^4}q^{\alpha}}\\d^2 e^{-C_1/(C_5q^{\alpha})}&\text{ if } \frac{1}{C_1e^{C_4^4}q^{\alpha}}\le d.
\end{cases}
\end{equation}
\end{lem}
\begin{proof}
Let $Z$ be the (random) set of infections in $D$. By Lemma~\ref{lem:extremal} we have that if $D$ is covered, it contains at least $\lceil d/C_4^2\rceil$ disjoint clusters of $Z$, each one having diameter at most $C_3$ and at least $\alpha$ sites. Thus, the union bound gives
\begin{align*}
\mu(D\text{ is covered})\le\binom{C_3^{2\alpha}d^2}{\lceil d/C_4^2\rceil}q^{\alpha\lceil d/C_4^2\rceil}.
\end{align*}
For $d\le C_4^2$ this gives $C_3^{2\alpha}d^2q^\alpha$, which concludes the proof. For $C_4^2\le d\le C_1/(e^{C_4^4} q^\alpha)$ we use the inequality $\binom n k\le (ne/k)^k$ to obtain the desired bounds. For the case $d\ge 1/(e^{C_4^4}q^\alpha)$ we use Lemma \ref{lem:AL:app} to extract a smaller $\cT_0$-droplet $D'$ covered by $Z$ (hence intersecting $D$) with $1/(2e^{C_4^4}q^\alpha)\le \diam(D')\le 1/(e^{C_4^4}q^{\alpha})$. We then apply the second bound to $D'$ and use the union bound to conclude.
\end{proof}

We would now like to use analogous bounds on the probability of $\cT_0$-droplets being covered with initial condition $[Z]$ instead of $Z$. Unfortunately, we do not have access to the law of $[Z]$ when $Z$ follows $\mu$. Therefore, we rather bound the output of the covering algorithm for the closure using the original output. For that purpose, we define parallel notions of $\Gamma'$, \emph{modified clusters} and \emph{modified covering}, by replacing $C_2$ by $C_2'$ and $C_4$ by $C_4'$.

We then have the following key property, whose proof is identical to the one of \cite{Hartarsky20}*{Proposition 5.20}, up to the relevant simplifications (we do not have rugged edges and there is no boundary).

\begin{lem}[Closure]
\label{lem:closure}
Let $Z\subset \bbZ^2$ be a finite set and let $\cD'$ be the collection of $\cT_0$-droplets given by the modified covering algorithm with input $[Z]$. Let $\cD$ be the output of the covering algorithm for $Z$. Then
\[\forall D'\in\cD'\,\exists D\in\cD, D'\subset D.\]
\end{lem}

\subsection{Spanning}
\label{app:sec:spanning}
Let $\cT$ be an arbitrary finite set of rational directions containing the origin in the interior of its convex envelope. We then generalise the notion of spanning from Definition \ref{def:span}.
\begin{defn}[Spanning]
\label{def:span:app}
Let $D$ be a $\cT$-droplet. We say that $D$ is \emph{spanned} by $Z\subset\bbZ^2$ if there exists a set $C\subset[Z\cap D]$ connected in $\Gamma'$ such that the smallest $\cT$-droplet containing $C$ is $D$.
\end{defn}
We will need the following Aizenman-Lebowitz type lemma. Though this is a very classical result, some additional arguments are needed to prove it, because $\cT$ is not composed of stable directions.
\begin{lem}[Aizenman-Lebowitz]\label{lem:AL:spanned}
Let $Z \subset \bbZ^2$ and $D$ be a $\cT$-droplet spanned by $Z$ with $\diam(D) \geq C_1C_2'$. Then for any $C_2' \leq k \leq \diam(D)$, there exists a $\cT$-droplet $D'$ spanned by $Z \cap D$ with $k \leq \diam(D') \leq C_1 k$.
\end{lem}
\begin{proof}
If $k\ge\diam(D)/C_1$ there is nothing to prove, as $D'=D$ is as desired. Assume $k \leq \diam(D)/C_1$. Let $C$ be a connected component of $[Z\cap D]$ in $\Gamma'$ with maximal diameter. By Definition \ref{def:span:app} $\diam(C)\ge \diam(D)/\sqrt{C_1}$. By Observation \ref{obs:connected:component} and \cite{Bollobas14}*{Lemma 6.18} (we use it although definitions slightly differ from \cite{Bollobas14}, see Footnote \ref{foot:constants}) there exists $C'\subset C$ connected in $\Gamma'$ such that $C'\subset [C'\cap Z\cap D]$ and $k\le \diam(C')\le \sqrt{C_1}k$. Denoting $D'$ the smallest $\cT$-droplet containing $C'$, we are done.
\end{proof}
\begin{obs}
\label{obs:cover:to:span}
Let $D$ be a $\cT$-droplet spanned by $Z\subset\bbZ^2$ such that $\diam(D)\ge C_4$. Then there exists a $\cT_0$-droplet $\bar D$ covered by $Z$, intersecting $D$ and such that $\diam(\bar D)=\Theta(\diam(D))$.
\end{obs}
\begin{proof}
Let $C$ be as in Definition \ref{def:span:app}. Notice that, since $\diam(D)\ge C_3$, we can find modified clusters for $[Z\cap D]$ whose union is a connected set in $\Gamma'$ containing $C$. Then there is a $\cT_0$-droplet in the output of the modified covering algorithm for $[Z\cap D]$ containing $C$. By Lemma \ref{lem:closure} there is also a $\cT_0$-droplet $\bar D$ in the output of the covering algorithm for $Z\cap D$ containing $C$, so that $\diam(\bar D)\ge\diam(C)=\Omega(\diam(D))$. But $\bar D$ is at most the smallest $\cT_0$-droplet containing $\{x \in \bbR^2 : d(x,D) \leq C_4\}$, so $\diam(\bar D)=\Theta(\diam(D))$. Moreover, since $\bar D$ is in the output of the covering algorithm for $Z \cap D$, it is covered by $Z$ and intersects $D$.
\end{proof}
We immediately deduce from this observation and Lemma \ref{lem:bounds:covered} the desired bounds on spanning.
\begin{cor}
\label{cor:bounds:app}
Let $D$ be a $\cT$-droplet with $d=\diam(D)$ and let $1>\epsilon>0$. Then
\begin{equation}
\label{eq:bounds:app}\mu(D\text{ is spanned})\le\begin{cases} q^{d\epsilon/C_5}&\text{ if } d\le q^{-\alpha+\epsilon}\\e^{-2C_4d} &\text{ if }q^{-\alpha+\epsilon}\le d\le \frac{C_1}{C_5q^{\alpha}}\\d^{O(1)} e^{-2/(C_5q^{\alpha})}&\text{ if } \frac{1}{C_5q^{\alpha}}\le d.
\end{cases}
\end{equation}
\end{cor}

\subsection{Boundary and spanning}
\label{app:sec:boundary}
We next turn to the treatment of an infinite infected boundary condition, following \cite{Hartarsky20}, which is applicable only for models with an infinite number of stable directions. Indeed, for a model with a finite number of stable directions a bounded set of infections next to the boundary can induce a set of supplementary infections and, thereby, a droplet of the size of the boundary, making similar algorithms useless. We therefore fix an update family $\cU$ with an infinite number of stable directions and difficulty $\alpha$, to which the treatment of \cite{Hartarsky20} applies. 

For the rest of this section let $\cS_{u}=\{u^-,u^+,v_1,v_2\}$ be a set of $4$ directions chosen as in \cite{Hartarsky20}*{Lemma 4.1}\footnote{It is not hard to see that in \cite{Hartarsky20}*{Lemma 4.1}, with a finite number of exceptions, given any rational strongly stable direction $u\in S^1$ we can define $\cS_u$ correspondingly.} (we rename $(u_1,u_2)$ from that work into $(u^-,u^+)$ to avoid notational conflict) with $u=(u^-+u^+)/2$. The proof of \cite{Hartarsky20} allows us to choose $u^-$ and $u^+$ as close as we want, even depending on $v_1$ and $v_2$. We will choose them close enough for our results to hold. Let $\partial=\bbH_{u}$. For any set $Z\subset \bbZ^2$ we write $[Z]_\partial=[Z\cup\partial]\setminus\partial$. We will use the term \emph{cluster} in the sense of \cite{Hartarsky20}*{Definition 5.3}, extending Definition \ref{def:cluster:app} (crumbs close to $\partial$ are considered as clusters instead and $Z$ is replaced by $Z \setminus \partial$). We replace the notion of DYD from \cite{Hartarsky20} by that of $\cS_u$-droplet and the notion of CDYD becomes that of \emph{cut $\cS_u$-droplet}---a nonempty set of the form
\begin{equation}
\label{eq:def:cut}
\left(\bar \bbH_{u^-}(x)\cap\bar \bbH_{u^+}(y)\right)\setminus\partial
\end{equation}
for some $x,y\in\bbR$, which is a geometric triangle. We further replace the use of the diameter by considering the size $|\cdot|$ from \cite{Hartarsky20}*{Definition 5.7}. Namely for a cut $\cS_u$-droplet $D$ we denote $|D|=\diam(D)/C_1$, while for an $\cS_u$-droplet $D$, $|D|$ denotes the length of its projection parallel to $v_1$. We then define correspondingly an extension of the covering algorithm as in \cite{Hartarsky20}*{Section 5.4} and a notion of covered (cut) $\cS_u$-droplet. For the reader unfamiliar with \cite{Hartarsky20}, let us indicate that the change with respect to the covering algorithm of Definition \ref{def:covering} corresponds to replacing at each stage of the algorithm any $\cS_u$-droplet $D$ intersecting $\partial$ by the smallest cut $\cS_u$-droplet containing $D\setminus\partial$. The properties of \cite{Hartarsky20}*{Section 5.5}, analogous to Lemmas \ref{lem:subadd}-\ref{lem:extremal} and \ref{lem:closure}, remain valid for this setting. Furthermore, combining the proofs of Lemma \ref{lem:bounds:covered} and \cite{Hartarsky20}*{Lemma 5.15} shows that the following holds.

\begin{lem}
\label{lem:bounds:covered:boundary}
Let $D$ be a cut $\cS_u$-droplet or an $\cS_u$-droplet not intersecting $\partial$ with $d=|D|$. Let $1>\epsilon>0$. Then \eqref{eq:bounds:covered} holds.
\end{lem}

We similarly extend Definition \ref{def:span:app} to the setting with boundary. 
\begin{defn}[Spanning with boundary]
\label{def:span:boundary}
We call \emph{whole $\cS_u$-droplet} any $\cS_u$-droplet at distance at least $C_3$ from $\partial$ and, by abuse, we call collectively \emph{$\cS_u$-droplet} any cut or whole $\cS_u$-droplet. We say that an $\cS_u$-droplet $D$ is \emph{spanned} by $Z\subset\bbZ^2$ if there exists a set $C\subset[Z\cap D]_\partial$ connected in $\Gamma'$ such that the smallest $\cS_u$-droplet containing $C$ is $D$.
\end{defn}

We next recall several properties of the spanning algorithm following closely \cite{Bollobas14}.
\begin{defn}[Definition 6.15 of \cite{Bollobas14}]
Let $Z=\{z_1,\dots,z_{k_0}\}$ be a finite set of infections. Set $\cZ^0=\{Z_1^0,\dots,Z^0_{k_0}\}$ with $Z^0_i=\{z_i\}$. For each $t\ge 0$ do the following.
\begin{itemize}
    \item If there exist $Z^t_i$ and $Z^t_j$ such that $[Z^t_i]_\partial\cup[Z^t_j]_\partial$ is connected in $\Gamma'$, then set $\cZ^{t+1}=(\cZ^t\setminus\{Z^t_i,Z^t_j\})\cup\{Z^t_i\cup Z^t_j\}$.
    \item Otherwise, define the \emph{span of $Z$} by $\<Z\>=\{D(Z^t),Z^t\in\cZ^t\}$, where $D(Z')$ denotes the smallest $\cS_u$-droplet containing $Z'$, and terminate the algorithm.
\end{itemize}
Similarly, for any $A\subset\bbR^2$ we denote $\<A\>=\<A\cap\bbZ^2\>$.
\end{defn}

\begin{obs}[Lemma 6.16 of \cite{Bollobas14}]
We have $\<Z\>=\{D(\kappa_1),\dots,D(\kappa_k)\}$, where the $\kappa_i$ are the connected components of $[Z]_\partial$ in $\Gamma'$. 
\end{obs}

\begin{obs}[Lemma 6.17 of \cite{Bollobas14}]
A nonempty $\cS_u$-droplet is spanned iff $D\in \<D\cap Z\>$.
\end{obs}

\begin{lem}[Lemma 6.21 of \cite{Bollobas14}]
\label{lem:6.19}
Let $Z$ be a finite set of at least two infections such that $[Z]_\partial$ is connected in $\Gamma'$. Then there exists a nontrivial partition $Z=Z_1\sqcup Z_2$ such that $[Z_1]_\partial$, $[Z_2]_\partial$ and $[Z_1]_\partial\cup[Z_2]_\partial$ are connected in $\Gamma'$.
\end{lem}

The next lemma follows from the definition of size and \cite{Hartarsky20}*{Lemma 5.12}.
\begin{lem}
\label{lem:sizes}
For any $\cS_u$-droplets $D,D_1,D_2$ with $|D_1| \geq C_3$ or $|D_2| \geq C_3$ such that $\<D_1\>=\{D_1\}$, $\<D_2\>=\{D_2\}$ and $\<D_1\cup D_2\>=\{D\}$ we have $|D_1|/C_1\le |D|\le |D_1|+|D_2|+O(C_2')$.
\end{lem}

This standardly implies (see e.g.\ \cite{Bollobas14}*{Lemma 6.18}) the following.
\begin{lem}[Aizenman-Lebowitz]
\label{lem:AL}
Let $D$ be a spanned $\cS_u$-droplet and $C_3\le k\le |D|$. Then there exists a spanned $\cS_u$-droplet $D'\subset D$ with $k\le|D'|\le 3k$.
\end{lem}

Similarly to Corollary \ref{cor:bounds:app} we obtain the following.
\begin{cor}
\label{cor:bounds:app:cut}
Let $D$ be an $\cS_u$-droplet with $d=|D|\ge 1/(C_5q^\alpha)$. Then
\[\mu(D\text{ is spanned})\le d^{O(1)}e^{-2/(C_5 q^\alpha)}.\]
\end{cor}
\begin{rem}
\label{rem:directions}
Let us note that the results of this section remain valid if $\partial$ is replaced by any sufficiently regular boundary condition. Namely, if $u_\perp=u+\pi/2$ and $f$ is a $\delta$-Lipschitz function for $\delta<\tan((u^+-u^-)/2)$, then we can use any $\partial$ with topological interior 
\[\{x\in\bbR^2,\<x,u\>< f(\<x,u_\perp\>)\}\]
such that $\partial$, $\partial \cup D$ are stable for any cut $\cS_u$-droplet.

Finally, one can also remove the boundary by considering infections sufficiently far from it to recover the setting of the previous section for the directions under consideration.
\end{rem}

\section{Bound on crossing}
\label{app:crossing}
For this appendix we place ourselves in the context of Section \ref{subsec:general} (in particular, $\cT$-droplets will be parallelograms). In sections \ref{app:sec:crossing:finite} and \ref{app:sec:crossing:infinite} we show that crossings are unlikely in directions with respectively finite and infinite difficulty. Of course, though we treat $u_1$, the results are also valid for $u_2$.

\subsection{Crossing in a direction with finite difficulty}
\label{app:sec:crossing:finite}
One can use Corollary \ref{cor:bounds:app} to show that if $u_1$ has finite difficulty, a $u_1$-crossing without large droplets is extremely unlikely. To do that, we will use a concept of \emph{partition} close to the one from \cite{Bollobas14}*{Definition 8.20}.
\begin{defn}
\label{def:partition}
Assume that $0 < \alpha(u_1) < \infty$. Let $R=R(a,b;c,d)$ be a parallelogram and $Z \subset R \cap \bbZ^2$. Set $m=\lfloor (c-a)/(C_1C_6)\rfloor\ge 1$ and
\[S_i = \bbH_{u_1}(-(c-iC_1C_6))\cap\bar\bbH_{u_2}(-b) \cap \bar \bbH_{u_3}(c-(i-1)C_1C_6) \cap \bar \bbH_{u_4}(d)\]
for $1 \leq i \leq m-1$ and $S_m=R(a,b;c-(m-1)C_1C_6,d)$.
A $u_1$-partition of $R$ for $Z$ is a sequence $a_1,\dots,a_k$ of positive integers with $m=a_1+\dots+a_k$ such that, setting $t_j=a_1+\dots+a_j$, we have either
\begin{itemize}
    \item $a_j=1$ and $S_{t_j}$ contains an $\alpha(u_1)$-cluster for $Z$ (see Definition \ref{def:cluster:app}) or
    \item there exists a $\cT$-droplet $D$ spanned by $Z \cap \bigcup_{i=t_{j-1}+1}^{t_j} S_i$, with $C_1C_6a_j\ge \diam(D)\ge a_jC_6$.
\end{itemize}
\end{defn} 
The following lemma is close to \cite{Bollobas14}*{Lemma 8.21}.
\begin{lem}
\label{lem:partition}
Let $R$ be a parallelogram. If $0 < \alpha(u_1) < \infty$ and $R$ is $u_1$-crossed then there exists a $u_1$-partition for $\eta\cap R$.
\end{lem}
\begin{proof}
 For notational convenience we assume that $R=R(-a,0;0,d)$. In this proof, all clusters and crumbs are with respect to $\alpha(u_1)$. The proof is by induction on $m$.

Suppose that the property holds for any $m' \leq m-1$. If $S_1$ contains a cluster of $\eta \cap R$, we set $a_1=1$ and we are done, since $R(-a,0;-C_1C_6,d)$ is $u_1$-crossed. Let us assume $S_1$ contains no cluster of $\eta \cap R$. Then $S_1' = R(-C_1C_6+C_1C_3,0;0,d)$ intersects no cluster of $\eta \cap R$, so if $\cK$ is the set of connected components of $\eta \cap S_1'$ in $\Gamma$, each $\kappa \in \cK$ is a crumb of $\eta \cap S_1'$. In particular, all elements of $[\kappa]$ and $[\kappa \cup \bbH_{u_1}] \setminus \bbH_{u_1}$ are at distance at most $C_1$ of $\kappa$ (see Observation 5.16 and the proof of Corollary 5.17 in \cite{Hartarsky20}). As elements of $\cK$ are at distance at least $C_2$ from one another, this means that $[\eta \cap S_1'] = \bigcup_{\kappa \in \cK}[\kappa]$, and that $\bar Z = \bigcup_{\kappa \in \cK'} [\kappa \cup \bbH_{u_1}]$ is closed, where $\cK'=\{\kappa \in \cK : d(\kappa,\bbH_{u_1}) \leq C_2\}$. Moreover, the diameter of a crumb is at most $\alpha(u_1)C_2$, so all elements of $\bar Z$ are at distance at most $(\alpha(u_1)+2)C_2$ of $\bbH_{u_1}$. Since $R$ is $u_1$-crossed, this implies that there exists $z \in \bar Z$ and $w \in [(\eta \cap R)\setminus \bar Z]$ such that $d(z,w) \leq C_2'$. Then $d(w,\bbH_{u_1})\le(\alpha(u_1)+2)C_2+C_2'$. 

Let $X$ be the connected component in $\Gamma'$ of $[(\eta \cap R)\setminus\bar Z]$ containing $w$. If $X\subset[\eta \cap S_1']$, then $X\subset\bigcup_{\kappa \in \cK}[\kappa]$, so $X\subset[\kappa]$ for some $\kappa\in\cK$, since they are at distance more than $C_2'$ from one another. Moreover, by Observation \ref{obs:connected:component}, $X = [((\eta \cap R)\setminus\bar Z) \cap X]$, so $X\not\subset\bar Z$, so $\kappa \not\in \cK'$. However, this contradicts the fact that $d(w,z)\le C_2'$, as $d(\bar Z,[\kappa])\ge C_2-2C_1$.

Therefore, $X\not\subset [\eta\cap S_1']$, so $X$ intersects $R \setminus S_1'$. Let $a_1=\max\{i\ge 1, X\cap S_i\neq\varnothing\}$
and $D$ be the smallest $\cT$-droplet containing $X$. Clearly, $\diam(D)\ge \diam(X)\ge a_1C_6$, since $d(w,\bbH_{u_1})\le C_3$. Furthermore, since $X = [((\eta \cap R)\setminus\bar Z) \cap X]$, $D$ is spanned by $\eta \cap \bigcup_{i=1}^{a_1}S_i$. We then conclude by Lemma \ref{lem:AL:spanned} and the induction hypothesis for $R(-a,0;-a_1C_1C_6,b)$.
\end{proof}

We next require a more sophisticated version of \cite{Bollobas14}*{Lemma 8.23}.
\begin{lem}
\label{lem:crossing:bounds}
Fix $K$ in Definitions \ref{def:span} and \ref{def:crossing} by
\[K=\begin{cases}
1/(C_5q^\alpha)&\text{if $\cU$ is balanced}\\
q^{-\alpha-1/4}&\text{if $\cU$ is unbalanced}.\end{cases}\]
Assume that $0<\alpha(u_1)<\infty$. Let $R=R(a,b;c,d)$ with $d-b\le \frac{\log (1/q)}{C_6^3q^{\alpha(u_1)}}$ and $1/q^{C_1}\ge c-a\ge 1/q$. Then
\[\mu(C_R^{u_1})\le\begin{cases} \exp\left(-(c-a)\exp\left(-2C_6^2(d-b)q^{\alpha(u_1)}\right)/C_6^2\right)&\text{if $\cU$ is balanced}\\
\exp\left(-(c-a)q^{1/4}/C_6\right)&\text{if $\cU$ is unbalanced}.
\end{cases}\]
\end{lem}
\begin{proof}
For notational convenience, we assume that $R=R(-a,0;0,d)$. If $C_R^{u_1}$ holds, there exists $\eta' \geq \eta$ such that $R$ is $u_1$-crossed for $\eta'$ and there is no spanned critical $\cT$-droplet for $\eta' \cap R$. By Lemma \ref{lem:partition}, there exists a $u_1$-partition for $\eta' \cap R$ and, by Lemma \ref{lem:AL:spanned}, all corresponding spanned $\cT$-droplets have diameter at most $K/C_1$. We notice that any empty site or spanned droplet for $\eta'$ is still an empty site or spanned droplet for $\eta$.

We first assume that $\cU$ is balanced. Given a partition $\cP$ we define its numbers and total sizes of big/small/cluster parts by
\begin{align*}
\cB&{}=\left\{j:1/\sqrt{q}<a_j\le 1/(C_5C_6q^\alpha)\right\}&b&{}=|\cB|&B&{}=\sum_{j\in \cB}a_j\\
\cS&{}=\{j:1<a_j\le 1/\sqrt{q}\}&s&{}=|\cS|&S&{}=\sum_{j\in \cS}a_j\\
\cC&{}=\{j:a_j=1\}&c&{}=|\cC|.
\end{align*}
We denote by $\cP(b,s,c,B,S)$ the set of partitions $\cP$ with the corresponding numbers and total sizes of parts.

Then, using Corollary \ref{cor:bounds:app}, we get that the probability of a given $\cP$ occurring is at most
\begin{align*}
\Pi(\cP)&{}=\prod_{j\in\cC}(1-(1-q^{\alpha(u_1)})^{C_6^2d})\prod_{j\in\cS}q^{a_j\sqrt{C_6}}\prod_{j\in\cB}e^{-C_3C_6 a_j}\\
&{}=(1-(1-q^{\alpha(u_1)})^{C_6^2d})^cq^{S\sqrt{C_6}}e^{-C_3C_6 B}
\end{align*}
by the union bound on all possible droplets and their positions, recalling that $d=q^{-O(1)}$. Indeed, the probability that there is no set of $\alpha(u_1)$ zeroes connected in $\Gamma'$ in a given $S_i$ is the probability that for any possible such set $C$, $\eta_C \neq 0$, which, by the Harris inequality, is bigger than the product of this probability for each set $C$.

Assuming for simplicity that $1/\sqrt{q}$ and $1/(C_5C_6q^\alpha)$ are integers, we can count $\cP(b,s,c,B,S)$ in the following way (the first binomial coefficient corresponds to the decomposition of $\cB$ into ordered parts, the second one to the decomposition of $\cS$, and the last two to the ordering of the parts of $\cB$, $\cS$, $\cC$):
\begin{align*}|\cP(b,s,c,B,S)|&{}\le \binom{B-b/\sqrt{q}-1}{b-1}\binom{S-s-1}{s-1}\binom{b+s+c}{b}\binom{s+c}{s}\\
&{}\le 2^{B+S}(b+s+c)^b(s+c)^s\le e^{B+S}q^{-C_1s},
\end{align*}
recalling that $C_6(B+S+c)<a\le 1/q^{C_1}$. Therefore, denoting by $m=\lfloor a/(C_1C_6)\rfloor =B+S+c$ the total number of strips, we have
\begin{multline*}
\sum_{B,S,b,s}\sum_{\cP\in\cP(b,s,m-B-S,B,S)}\Pi(\cP)\\\begin{aligned}[t]&{}\le m^4\max_{B,S}\left(1-\left(1-q^{\alpha(u_1)}\right)^{C_6^2d}\right)^{m-B-S}q^{S\sqrt{C_6}/2}e^{-C_3C_6 B/2}\\
&{}\le m^4\max_{0\le c\le m}e^{-c\exp\left(-2C_6^2dq^{\alpha(u_1)}\right)}e^{-C_2C_6(m-c)}\\
&{}\le\exp\left(-\frac{m}{2}\exp\left(-2C_6^2dq^{\alpha(u_1)}\right)\right),
\end{aligned}
\end{multline*}
which concludes the proof in the balanced case, recalling the hypotheses of the lemma.

We next consider $\cU$ to be unbalanced. Notice that, since $K=q^{-\alpha-1/4}$, there may be droplets with diameter larger than $1/(C_5q^\alpha)$. Therefore, we further set
\begin{align*}
\cH&{}=\left\{j: 1/(C_5C_6q^\alpha)<a_j\le 1/(C_6q^{\alpha+1/4})\right\}&h&{}=|\cH|&H&{}=\sum_{j\in \cH}a_j.
\end{align*}
Then Corollary \ref{cor:bounds:app} gives that the probability of a given $\cP$ occurring is at most
\[\Pi(\cP)\times \left(q^{O(1)}e^{-2/(C_5q^\alpha)}\right)^{h}\le\Pi(\cP)\times\exp\left(-Hq^{1/4} C_6/C_5\right).\]
We further easily check that
\begin{align*}\binom{H-h/(C_5C_6q^\alpha)-1}{h-1}&{}\le e^{H\sqrt{q}}&\binom{h+b+s+c}{h}&{}\le e^{H\sqrt{q}},
\end{align*}
so, as above the probability of any $\cP$ occurring is at most
\[m^6\exp\left(-m.\min\left(C_6q^{1/4}/(2C_5),\exp\left(-2C_6^2dq^{\alpha(u_1)}\right)\right)\right)\le\exp\left(-\frac{C_6mq^{1/4}}{3C_5}\right),\]
which concludes the proof.
\end{proof}

\subsection{Crossing in a direction with infinite difficulty}
\label{app:sec:crossing:infinite}

If $\cU$ has an infinite number of stable directions, we need to treat an infected boundary condition. This is essential, as we will work in exponentially large regions, for which the bounds from the previous section cannot be applied.

We place ourselves in the setting of Section \ref{subsec:infinite}.  We will write (cut/whole) droplet for (cut/whole) $\cS_{u_1}$-droplet in the sense of Definition \ref{def:span:boundary}, with $u_1^-$ and $u_1^+$ sufficiently close to $u_1$. These should not be confused with $\cT$-droplets, which are called parallelograms to avoid any confusion.
 
We will seek to apply Corollary \ref{cor:bounds:app:cut} rather than \ref{cor:bounds:app} to prove the following.
\begin{lem}
\label{lem:crossing:bounds:boundary}
Fix 
\[K=\begin{cases}q^{-\alpha}&\text{if $\cU$ is balanced}\\
q^{-\alpha-{1/4}}&\text{if $\cU$ is unbalanced}
\end{cases}\]
for Definitions \ref{def:span} and \ref{def:crossing}. Let $R=R(a,b;c,d)$ with $C_1 \leq d-b\le \exp(q^{-3\alpha})$ and $c-a\ge q^{-4\alpha}$. Then
\[\mu(C_R^{u_1})\le\exp\left(-q^{-3\alpha}\right).\]
\end{lem}

Our strategy is as follows. Instead of considering $u_1$-partitions, we directly retrace the spanning algorithm to obtain a hierarchy of droplets reaching a cut droplet of size roughly $c-a$. We reassure the reader familiar with \cite{Holroyd03} that our hierarchies will be very simple and imprecise, as the a priori hypothesis that there are no critical parallelograms removes the metastability (it is no longer easy for large droplets to grow) together with the need of fine tuning. Namely, their seeds will be of size roughly $K$ which will also be the increment of the size of unary vertices (the reader unfamiliar with hierarchies is invited to consult the definitions below). The lack of critical parallelograms entails that all droplets in the hierarchy are cut (so they are simply very flat triangles). The bound on the probability of seeds being spanned is provided by Corollary \ref{cor:bounds:app:cut} and entropy is easily subdominant, so we can focus on the probability that the infections around a cut droplet are such that if that droplet is infected, the infection can expand to fill a slightly larger cut droplet. However, this would imply that there is a (smaller scale) $u_1^\pm$-crossing from the side of the smaller one to side of the larger one (see Figure \ref{fig:small:crossing}). The probability of this event is again bounded directly by Corollary \ref{cor:bounds:app:cut}, taking into account Remark \ref{rem:directions}.

Let us begin by introducing our hierarchies following Holroyd \cite{Holroyd03}. Let $T=q^{-\alpha-1/4}$. Fix a droplet $D$. A \emph{hierarchy} $\cH$ for $D$ is a rooted unary-binary tree with each vertex $x$ labelled by a droplet $D_x\subset D$, so that the label of the root is $D$. We denote by $N(x)$ the set of children of $x\in V(\cH)$, so that $|N(x)|\in\{0,1,2\}$ for all $x$. The leaves are called \emph{seeds} and the binary vertices are called \emph{splitters}. We alert the reader that in reality there will only be cut droplets in our hierarchies, but for technical reasons we define them in general. A hierarchy is defined to satisfy the following conditions.
\begin{itemize}
    \item If $y\in N(x)$, then $D_y\subset D_x$.
    \item If $D_x$ is a whole droplet, then $x$ is a seed and $T/3\le|D_x|$.
    \item If $D_x$ is a cut droplet, then $T/3\le |D_x|\le T$ if and only if $x$ is a seed.
    \item If $N(x)=\{y\}$ and $|N(y)|=1$, then $T<|D_x|-|D_y|\le 2T$.
    \item If $N(x)=\{y\}$, then $|D_x|-|D_y|\le 2T$.
    \item If $N(x)=\{y,z\}$, then $|D_x|-|D_y|> T$ and $\<D_y\cup D_z\>=\{D_x\}$.
\end{itemize}
We set 
\begin{align*}
S(\cH)={}&\{x\in V(\cH):\,|N(x)|=0\}\\
N(\cH)={}&\{(x,y)\in (V(\cH))^2:\,N(x)=\{y\},|N(y)|=1\}\end{align*}
and remark that $|S(\cH)|-1$ is the number of splitters. We say that a hierarchy $\cH$ \emph{occurs} if the following events occur disjointly (are witnessed by disjoint sets of infections, see \cite{BK85}).
\begin{itemize}
    \item For every seed $x\in S(\cH)$ we have that $D_x$ is spanned.
    \item For every $x\in V(\cH)$ such that $N(x)=\{y\}$ we have $D_x\in\<D_y\cup(\eta\cap D_x)\>$.
\end{itemize}
\begin{lem}
\label{lem:hierarchies:existence}
If $D$ is a spanned droplet with $|D|\ge T/3$, then some hierarchy for $D$ occurs.
\end{lem}
\begin{proof}
The proof is very similar e.g.\ to \cite{Bollobas14}*{Lemma 8.7}. Assuming that $D_0$ is a spanned droplet with $|D_0|\ge T/3$, we construct an occurring hierarchy by induction on $D_0$ with respect to inclusion. If $|D_0|\le T$ or $D_0$ is a whole droplet, the hierarchy with only vertex labelled by $D_0$ is as desired.

Assume that $D_0$ is a cut droplet and $|D_0|>T$. Let $Z$ be a connected component for $\Gamma'$ of $[D_0 \cap \eta]_\partial$ such that the smallest droplet containing $Z$ is $D_0$, and let $Z_0=Z \cap \eta$. We then have $[Z_0]_\partial=Z$. By Lemma \ref{lem:6.19} there exist sequences $Z_1,\dots,Z_m$ and $Z_1',\dots,Z_m'$ of subsets of $Z_0$ and $D_1,\dots,D_m$ and $D_1',\dots,D_m'$ such that the following conditions hold for all $0<i\le m$.
\begin{itemize}
    \item $Z_{i-1}=Z_i\sqcup Z_i'$, $[Z_i]_\partial$, $[Z_i']_\partial$ and $[Z_i]_\partial\cup[Z_i']_\partial$ are connected in $\Gamma'$.
    \item $D_i=D([Z_i]_\partial)$ and $D_i'=D([Z_i']_\partial)$.
    \item $\<D_i\cup D_i'\>=\{D_{i-1}\}$ and $|D_i|\ge |D_i'|$
    \item $m\ge 1$ is the minimal index such that one of the following holds:
    \begin{enumerate}
        \item\label{cond:ii} $|D_0|-|D_m|> T$;
        \item\label{cond:i} $D_m$ is a whole droplet;
        \item\label{cond:iii} $|D_m|\le T$.
    \end{enumerate}
\end{itemize}

If \ref{cond:ii} does not hold, then we attach a seed labelled by $D_m$ to the root and we are done, as $|D_m|\ge |D_{m-1}|/3> T/3$ by Lemma \ref{lem:sizes} and minimality of $m$. Indeed, $D_m$ being spanned is witnessed by $Z_m$, while $D_0\in\<D_m\cup(\eta\cap D_0)\>$ is witnessed by $Z_0\setminus Z_m$.

Assume that \ref{cond:ii} holds. Then we consider two cases. If $T<|D_0|-|D_m|\le 2T$, we attach a hierarchy for $D_m$ (occurring for $Z_m$) to the root $D_0$ and we are done using Lemma \ref{lem:sizes} to get that $|D_m|> T/3$ as above. Otherwise we attach a splitter labelled by $D_{m-1}$ to the root $D_0$ (if $m=1$, then $D_0$ is the splitter) and hierarchies for $D_m$ and $D_m'$ to that splitter. Then we are done, recalling Lemma \ref{lem:sizes}, to get that $|D_m|\ge |D_m'|\ge |D_{m-1}|-|D_m|-O(C_2')\ge T-O(C_2')$.
\end{proof}

In order to bound the probability that a hierarchy occurs, we will need the following result.
\begin{lem}
\label{lem:growth}
Let $D_1\subset D_2$ be two cut droplets for $\partial=\bbH_{u_1}$ such that $T<|D_2|-|D_1|\le 2T$ and $|D_2|\le q^{-4\alpha}$. Then
\[\mu(D_2\in\<D_1\cup(\eta\cap D_2)\>)\le e^{-q^{-\alpha}/C_5}.\]
\end{lem}

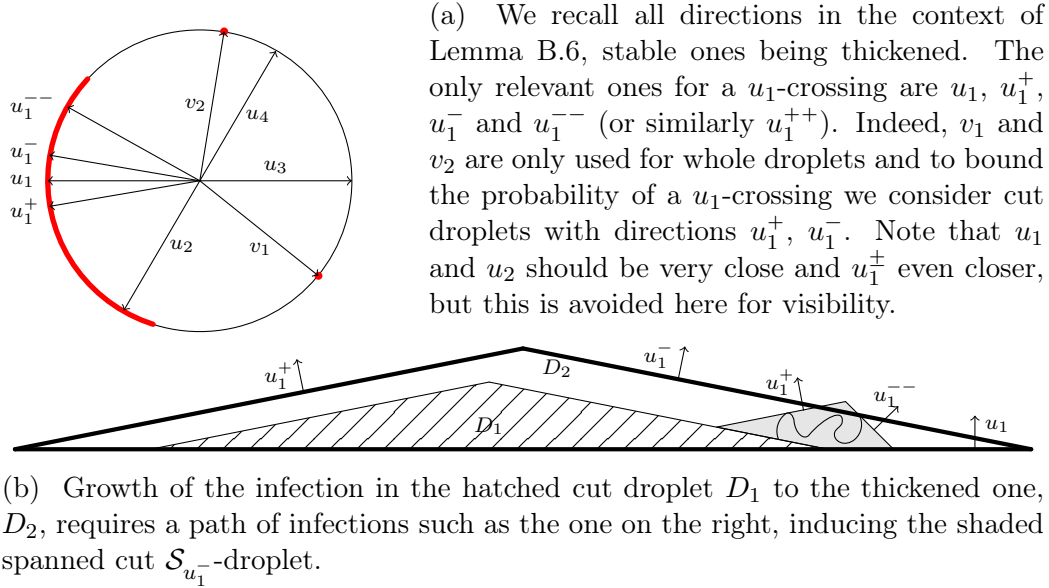
\begin{figure}
\parbox[b]{.4\linewidth}{\begin{tikzpicture}[line cap=round,line join=round,x=2.0cm,y=2.0cm]
\clip(-1.3,-1.05) rectangle (1.05,1.05);
\draw(0,0) circle (2cm);
\fill [color=red] (0.16,0.99) circle (1.5pt);
\fill [color=red] (0.78,-0.63) circle (1.5pt);
\draw [shift={(0,0)},line width=2pt,color=red]  plot[domain=2.4:4.4,variable=\t]({1*1*cos(\t r)+0*1*sin(\t r)},{0*1*cos(\t r)+1*1*sin(\t r)});
\draw [->] (0,0) -- (-0.99,0.17);
\draw [->] (0,0) -- (-0.87,0.49);
\draw [->] (0,0) -- (-0.99,-0.17);
\draw [->] (0,0) -- (0.78,-0.63);
\draw [->] (0,0) -- (0.16,0.99);
\draw [->] (0,0) -- (1,0);
\draw [->] (0,0) -- (0.5,0.86);
\draw [->] (0,0) -- (-0.5,-0.86);
\draw [->] (0,0) -- (-1,0);
\begin{scriptsize}
\draw[color=black,anchor=west] (-1.3,0.2) node {$u_1^-$};
\draw[color=black,anchor=west] (-1.3,0.5) node {$u_1^{--}$};
\draw[color=black,anchor=west] (-1.3,-0.2) node {$u_1^+$};
\draw[color=black,anchor=west] (-1.3,0) node {$u_1$};
\draw[color=black] (0.4,-0.45) node {$v_1$};
\draw[color=black,anchor=east] (0.1,0.5) node {$v_2$};
\draw[color=black,anchor=west] (0.25,0.43) node {$u_4$};
\draw[color=black,anchor=west] (-0.25,-0.43) node {$u_2$};
\draw[color=black,anchor=south] (0.5,0) node {$u_3$};
\end{scriptsize}
\end{tikzpicture}}
\parbox[b]{.59\linewidth}
{\subcaption{\label{fig:directions} We recall all directions in the context of Lemma \ref{lem:growth}, stable ones being thickened. The only relevant ones for a $u_1$-crossing are $u_1$, $u_1^+$, $u_1^-$ and $u_1^{--}$ (or similarly $u_1^{++}$). Indeed, $v_1$ and $v_2$ are only used for whole droplets and to bound the probability of a $u_1$-crossing we consider cut droplets with directions $u_1^+$, $u_1^-$. Note that $u_1$ and $u_2$ should be very close and $u_1^\pm$ even closer, but this is avoided here for visibility.
}}
\\
\begin{subfigure}{\textwidth}
\centering
\begin{tikzpicture}[line cap=round,line join=round,x=0.065\textwidth,y=0.065\textwidth]
\clip(-7.1,-0.1) rectangle (8.1,1.6);
\fill[fill=black,fill opacity=0.1]  (5.26,0.71) -- (3.36,0.33) -- (5,0) -- (5.97,0) -- cycle;
\fill[fill=black,pattern=my north east lines] (-5,0) -- (0,1) -- (5,0) -- cycle;
\draw (4.27,0.15) to [curve through ={(4.38,0.53)..([tension=0.55]4.5,0.42)..(4.72,0.21)..(5.21,0.4)..(5.1,0.25)..(5.4,0.15)}] (5.46,0.51);
\draw (5.26,0.71) -- (3.36,0.33);
\draw (3.36,0.33) -- (5,0);
\draw (5,0) -- (5.97,0);
\draw (5.97,0) -- (5.26,0.71);
\draw (-5,0) -- (0,1);
\draw (0,1) -- (5,0);
\draw (5,0) -- (-5,0);
\draw [line width=1.6pt] (0.5,1.5)-- (8,0);
\draw [line width=1.6pt] (8,0)-- (-7,0);
\draw [line width=1.6pt] (-7,0)-- (0.5,1.5);
\draw [->] (-2.74,0.85) -- (-2.84,1.34);
\draw [->] (2.79,1.04) -- (2.89,1.53);
\draw [->] (4.64,0.58) -- (4.55,1.07);
\draw [->] (7.17,0) -- (7.17,0.5);
\draw [->] (5.68,0.29) -- (6.04,0.64);
\begin{scriptsize}
\draw (-3.1,1.1) node {$u_1^+$};
\draw (2.5,1.4) node {$u_1^-$};
\draw (4.3,1) node {$u_1^+$};
\draw (6,0.82) node {$u_1^{--}$};
\draw (7.5,0.35) node {$u_1$};
\draw (0,0.35) node {$D_1$};
\draw (1,1.2) node {$D_2$};
\end{scriptsize}
\end{tikzpicture}
\caption{\label{fig:small:crossing} Growth of the infection in the hatched cut droplet $D_1$ to the thickened one, $D_2$, requires a path of infections such as the one on the right, inducing the shaded spanned
cut $\cS_{u_1^{-}}$-droplet.}
\end{subfigure}
\caption{\label{fig:growth} Illustration of the proof of Lemma \ref{lem:growth} bounding the probability that the infections around a cut droplet, $D_1$, allow an infection filling $D_1$ to grow and fill the slightly larger cut droplet, $D_2$.}
\end{figure}

\begin{proof}
The proof is illustrated in Figure \ref{fig:growth}. Let us denote $D_i=(\bar \bbH_{u_1^+}(x_i)\cap\bar \bbH_{u_1^-}(y_i))\setminus \bbH_{u_1}$ for $i\in\{1,2\}$. Define the strips $X=\bar\bbH_{u_1^+}(x_2)\setminus\bar \bbH_{u_1^+}(x_1)$ and $Y=\bar \bbH_{u_1^-}(y_2)\setminus\bar \bbH_{u_1^-}(y_1)$ and assume without loss of generality that $y_2-y_1=\Omega(T)$. Assume that $D_2\in\<D_1\cup(\eta\cap D_2)\>$ occurs. Setting $\eta'=\eta\cap Y\cap D_2$, this implies $D_2\in\<(D_2\setminus Y)\cup \eta'\>$. We consider two cases. 

Assume that $D_2\in \<\eta'\>$. By Corollary \ref{cor:bounds:app:cut} the probability of this event is at most $q^{-O(1)}e^{-2/(C_5q^\alpha)}$.

Assume that, on the contrary, $D_2\not\in\<\eta'\>$ and set $\partial'=\bbH_{u_1}\cup\bar\bbH_{u_1^-}(y_1)$. Then by Observation \ref{obs:connected:component} there exists a set $C\subset [D_2 \cap\eta]_{\partial'}$ connected in $\Gamma'$ such that $d(C,\bbH_{u_1^-}(y_1))\le C_2'$ and $C\not\subset \bbH_{u_1^-}(y_2)$. By definition this implies the existence of a cut $\cS_{u_1^-}$-droplet spanned by $D_2\cap\eta$ with boundary $\partial'$, where the two directions of cut droplets in $\cS_{u_1^-}$ are $u_1^+$ and $u_1^{--}=2u_1^--u_1^+$
(recall Remark \ref{rem:directions}). Hence, by the union bound over all possible such droplets and Corollary \ref{cor:bounds:app:cut} we obtain the desired result.
\end{proof}

We are now ready to assemble the proof of Lemma \ref{lem:crossing:bounds:boundary} as outlined at the beginning of the section.
\begin{proof}[Proof of Lemma \ref{lem:crossing:bounds:boundary}]
Assume that $C_R^{u_1}$ occurs and let $\eta'\ge\eta$ be as in Definition \ref{def:crossing}. Then there exists a spanned cut droplet for $\eta'\cap R$ with boundary $\bbH_{u_1}$ of diameter at least $c-a$. By Lemma \ref{lem:AL} this implies the existence of a droplet $D$ spanned for $\eta'\cap R$ with $q^{-4\alpha}/C_1\le |D|\le 3q^{-4\alpha}/C_1$. We set $Z=\eta' \cap R\cap D$.

Let us assume for a contradiction that there exists a whole droplet of size at least $q^{-\alpha-1/4}/3$ spanned for $Z$. It is easy to check that there exists a parallelogram of diameter at least $q^{-\alpha-1/4}/C_1$ spanned by $Z$ (consider a connected component satisfying Definition \ref{def:span:boundary}, take the smallest parallelogram containing it and use Observation \ref{obs:connected:component}). By Lemma \ref{lem:AL:body} this contradicts the absence of spanned critical parallelograms for $\eta'\cap R$.

Therefore, $D$ is a cut droplet and by Lemma \ref{lem:hierarchies:existence} there exists a hierarchy for $D$ occurring for the zero set $Z$, whose labels are all cut droplets. Let $\mathscr H(D)$ denote the set of such hierarchies. Now, by the BK inequality \cite{BK85}, for any hierarchy $\cH$ we have the following analogue of \cite{Holroyd03}*{Equation (37)}:
\[\mu(\cH\text{ occurs}) \leq \prod_{x\in S(\cH)}\mu(D_x\text{ is spanned})\prod_{(x,y)\in N(\cH)}\mu(D_x\in\<D_y\cup(\eta\cap D_x)\>).\]
Thanks to Corollary \ref{cor:bounds:app:cut} and Lemma \ref{lem:growth}, we deduce 
\[\mu(C_R^{u_1})\le\sum_D\sum_{\cH\in\mathscr H(D)}\exp\left(-q^{-\alpha}(|S(\cH)|+|N(\cH)|)/C_5\right).\]
The number of choices for $D$ is $O((d-b)q^{-4\alpha})$. We separate the sum over hierarchies according to their number of vertices $v(\cH)=\Theta(|S(\cH)|+|N(\cH)|)$. By Lemma \ref{lem:sizes} we have that $v(\cH)=\Omega (|D|/T)=\Omega(q^{-3\alpha+1/4}/C_1)$. Finally, the number of hierarchies for a given cut droplet $D$ with $v$ vertices is at most $q^{-O(v)}$. Combining these bounds we have
\[\mu(C_R^{u_1})\le (d-b)\sum_{v=\Omega(q^{-3\alpha+1/4}/C_1)}\exp\left(-q^{-\alpha}\Omega(v)/C_5\right),\]
which concludes the proof.
\end{proof}
\bibliographystyle{plain}
\bibliography{Bib}

\begin{bibdiv}
\begin{biblist}

\bib{Aizenman88}{article}{
      author={Aizenman, M.},
      author={Lebowitz, J.~L.},
       title={Metastability effects in bootstrap percolation},
        date={1988},
        ISSN={0305-4470, 1361-6447},
     journal={J. Phys. A},
      volume={21},
      number={19},
       pages={3801\ndash 3813},
      review={\MR{968311}},
}

\bib{Balister16}{article}{
      author={Balister, P.},
      author={Bollob{\'a}s, B.},
      author={Przykucki, M.},
      author={Smith, P.},
       title={Subcritical {$\mathcal {U}$}-bootstrap percolation models have
  non-trivial phase transitions},
        date={2016},
        ISSN={0002-9947, 1088-6850},
     journal={Trans. Amer. Math. Soc.},
      volume={368},
      number={10},
       pages={7385\ndash 7411},
      review={\MR{3471095}},
}

\bib{Bollobas17}{article}{
      author={Bollob{\'a}s, B.},
      author={{Duminil-Copin}, H.},
      author={Morris, R.},
      author={Smith, P.},
       title={The sharp threshold for the {{Duarte}} model},
        date={2017},
        ISSN={0091-1798},
     journal={Ann. Probab.},
      volume={45},
      number={6B},
       pages={4222\ndash 4272},
      review={\MR{3737910}},
}

\bib{Bollobas14}{article}{
      author={Bollob{\'a}s, B.},
      author={{Duminil-Copin}, H.},
      author={Morris, R.},
      author={Smith, P.},
       title={Universality of two-dimensional critical cellular automata},
        date={To appear},
     journal={Proc. Lond. Math. Soc.},
}

\bib{Bollobas15}{article}{
      author={Bollob{\'a}s, B.},
      author={Smith, P.},
      author={Uzzell, A.},
       title={Monotone cellular automata in a random environment},
        date={2015},
        ISSN={0963-5483},
     journal={Combin. Probab. Comput.},
      volume={24},
      number={4},
       pages={687\ndash 722},
      review={\MR{3350030}},
}

\bib{Cancrini08}{article}{
      author={Cancrini, N.},
      author={Martinelli, F.},
      author={Roberto, C.},
      author={Toninelli, C.},
       title={Kinetically constrained spin models},
        date={2008},
        ISSN={0178-8051},
     journal={Probab. Theory Related Fields},
      volume={140},
      number={3-4},
       pages={459\ndash 504},
      review={\MR{2365481}},
}

\bib{Cerf99}{article}{
      author={Cerf, R.},
      author={Cirillo, E. N.~M.},
       title={Finite size scaling in three-dimensional bootstrap percolation},
        date={1999},
        ISSN={00911798},
     journal={Ann. Probab.},
      volume={27},
      number={4},
       pages={1837\ndash 1850},
      review={\MR{1742890}},
}

\bib{Chung01}{article}{
      author={Chung, F.},
      author={Diaconis, P.},
      author={Graham, R.},
       title={Combinatorics for the {{East}} model},
        date={2001},
        ISSN={0196-8858},
     journal={Adv. Appl. Math.},
      volume={27},
      number={1},
       pages={192\ndash 206},
      review={\MR{1835679}},
}

\bib{DeGregorio09}{incollection}{
      author={{de Gregorio}, P.},
      author={Lawlor, A.},
      author={Dawson, K.~A.},
       title={Bootstrap percolation},
        date={2009},
   booktitle={Encyclopedia of complexity and systems science},
      editor={Meyers, R.~A.},
   publisher={{Springer, New York, NY}},
       pages={608\ndash 626},
}

\bib{Duminil-Copin13}{article}{
      author={{Duminil-Copin}, H.},
      author={{van Enter}, A. C.~D.},
       title={Sharp metastability threshold for an anisotropic bootstrap
  percolation model},
        date={2013},
        ISSN={0091-1798},
     journal={Ann. Probab.},
      volume={41},
      number={3A},
       pages={1218\ndash 1242},
      review={\MR{3098677}},
}

\bib{Faggionato13}{article}{
      author={Faggionato, A.},
      author={Martinelli, F.},
      author={Roberto, C.},
      author={Toninelli, C.},
       title={The {{East}} model: recent results and new progresses},
        date={2013},
        ISSN={1024-2953},
     journal={Markov Process. Related Fields},
      volume={19},
      number={3},
       pages={407\ndash 452},
      review={\MR{3156959}},
}

\bib{Fredrickson84}{article}{
      author={Fredrickson, G.~H.},
      author={Andersen, H.~C.},
       title={Kinetic {{Ising}} model of the glass transition},
        date={1984},
     journal={Phys. Rev. Lett.},
      volume={53},
      number={13},
       pages={1244\ndash 1247},
}

\bib{Fredrickson85}{article}{
      author={Fredrickson, G.~H.},
      author={Andersen, H.~C.},
       title={Facilitated kinetic {{Ising}} models and the glass transition},
        date={1985},
     journal={J. Chem. Phys.},
      volume={83},
      number={11},
       pages={5822\ndash 5831},
}

\bib{Garrahan11}{incollection}{
      author={Garrahan, P.},
      author={Sollich, P.},
      author={Toninelli, C.},
       title={Kinetically constrained models},
        date={2011},
   booktitle={Dynamical heterogeneities in glasses, colloids and granular media
  and jamming transitions},
      editor={Berthier, L.},
      editor={Biroli, G.},
      editor={Bouchaud, J.-P.},
      editor={Cipelletti, L.},
      editor={{van Saarloos}, W.},
      series={International series of monographs on physics 150},
   publisher={{Oxford University Press, Oxford}},
       pages={341\ndash 369},
}

\bib{Gravner08}{article}{
      author={Gravner, J.},
      author={Holroyd, A.~E.},
       title={Slow convergence in bootstrap percolation},
        date={2008},
        ISSN={1050-5164},
     journal={Ann. Appl. Probab.},
      volume={18},
      number={3},
       pages={909\ndash 928},
      review={\MR{2418233}},
}

\bib{Hartarsky21}{article}{
      author={Hartarsky, I.},
       title={{$\mathcal U$}-bootstrap percolation: critical probability,
  exponential decay and applications},
        date={2021},
        ISSN={0246-0203},
     journal={Ann. Inst. Henri Poincar\'e Probab. Stat.},
      volume={57},
      number={3},
       pages={1255\ndash 1280},
      review={\MR{4291442}},
}

\bib{Hartarsky20II}{article}{
      author={Hartarsky, I.},
       title={Refined universality for critical {{KCM}}: upper bounds},
        date={2021},
     journal={arXiv e-prints},
      eprint={2104.02329},
}

\bib{Hartarsky20}{article}{
      author={Hartarsky, I.},
      author={Mar{\^e}ch{\'e}, L.},
      author={Toninelli, C.},
       title={Universality for critical {{KCM}}: infinite number of stable
  directions},
        date={2020},
        ISSN={0178-8051, 1432-2064},
     journal={Probab. Theory Related Fields},
      volume={178},
      number={1},
       pages={289\ndash 326},
      review={\MR{4146539}},
}

\bib{Hartarsky20FA}{article}{
      author={Hartarsky, I.},
      author={Martinelli, F.},
      author={Toninelli, C.},
       title={Sharp threshold for the {{FA-2f}} kinetically constrained model},
        date={2020},
     journal={arXiv e-prints},
      eprint={2012.02557},
}

\bib{Hartarsky21a}{article}{
      author={Hartarsky, I.},
      author={Martinelli, F.},
      author={Toninelli, C.},
       title={Universality for critical {{KCM}}: finite number of stable
  directions},
        date={2021},
     journal={Ann. Probab.},
      volume={49},
      number={5},
       pages={2141\ndash 2174},
      review={\MR{4317702}},
}

\bib{Hartarsky19}{article}{
      author={Hartarsky, I.},
      author={Morris, R.},
       title={The second term for two-neighbour bootstrap percolation in two
  dimensions},
        date={2019},
        ISSN={0002-9947},
     journal={Trans. Amer. Math. Soc.},
      volume={372},
      number={9},
       pages={6465\ndash 6505},
      review={\MR{4024528}},
}

\bib{Holroyd03}{article}{
      author={Holroyd, A.~E.},
       title={Sharp metastability threshold for two-dimensional bootstrap
  percolation},
        date={2003},
        ISSN={0178-8051},
     journal={Probab. Theory Related Fields},
      volume={125},
      number={2},
       pages={195\ndash 224},
      review={\MR{1961342}},
}

\bib{Liggett05}{book}{
      author={Liggett, T.~M.},
       title={Interacting particle systems},
      series={Classics in mathematics},
   publisher={{Springer, Berlin, Heidelberg}},
        date={2005},
        ISBN={3-540-22617-6},
        note={Originally published by Springer, New York (1985)},
      review={\MR{2108619}},
}

\bib{Mareche20combi}{article}{
      author={Mar{\^e}ch{\'e}, L.},
       title={Combinatorics for general kinetically constrained spin models},
        date={2020},
        ISSN={0895-4801},
     journal={SIAM J. Discrete Math.},
      volume={34},
      number={1},
       pages={370\ndash 384},
      review={\MR{4062795}},
}

\bib{Mareche20Duarte}{article}{
      author={Mar{\^e}ch{\'e}, L.},
      author={Martinelli, F.},
      author={Toninelli, C.},
       title={Exact asymptotics for {{Duarte}} and supercritical rooted
  kinetically constrained models},
        date={2020},
        ISSN={0091-1798},
     journal={Ann. Probab.},
      volume={48},
      number={1},
       pages={317\ndash 342},
      review={\MR{4079438}},
}

\bib{Martinelli19a}{article}{
      author={Martinelli, F.},
      author={Morris, R.},
      author={Toninelli, C.},
       title={Universality results for kinetically constrained spin models in
  two dimensions},
        date={2019},
        ISSN={0010-3616},
     journal={Comm. Math. Phys.},
      volume={369},
      number={2},
       pages={761\ndash 809},
      review={\MR{3962008}},
}

\bib{Martinelli19}{article}{
      author={Martinelli, F.},
      author={Toninelli, C.},
       title={Towards a universality picture for the relaxation to equilibrium
  of kinetically constrained models},
        date={2019},
        ISSN={0091-1798},
     journal={Ann. Probab.},
      volume={47},
      number={1},
       pages={324\ndash 361},
      review={\MR{3909971}},
}

\bib{Morris17}{article}{
      author={Morris, R.},
       title={Bootstrap percolation, and other automata},
        date={2017},
        ISSN={0195-6698},
     journal={European J. Combin.},
      volume={66},
       pages={250\ndash 263},
      review={\MR{3692148}},
}

\bib{Morris17a}{incollection}{
      author={Morris, R.},
       title={Monotone cellular automata},
        date={2017},
   booktitle={Surveys in combinatorics 2017},
      series={London math. {{Soc}}. {{Lecture}} note ser.},
      volume={440},
   publisher={{Cambridge University Press, Cambridge}},
       pages={312\ndash 371},
      review={\MR{3728111}},
}

\bib{Ritort03}{article}{
      author={Ritort, F.},
      author={Sollich, P.},
       title={Glassy dynamics of kinetically constrained models},
        date={2003},
     journal={Adv. Phys.},
      volume={52},
      number={4},
       pages={219\ndash 342},
}

\bib{Sollich99}{article}{
      author={Sollich, P.},
      author={Evans, M.~R.},
       title={Glassy time-scale divergence and anomalous coarsening in a
  kinetically constrained spin chain},
        date={1999},
     journal={Phys. Rev. Lett.},
      volume={83},
      number={16},
       pages={3238\ndash 3241},
}

\bib{BK85}{article}{
      author={{van den Berg}, J.},
      author={Kesten, H.},
       title={Inequalities with applications to percolation and reliability},
        date={1985},
        ISSN={0021-9002},
     journal={J. Appl. Probab.},
      volume={22},
      number={3},
       pages={556\ndash 569},
      review={\MR{799280}},
}

\end{biblist}
\end{bibdiv}
\end{document}